\documentclass[12pt,english,times]{amsart}
\usepackage[T1]{fontenc}
\usepackage{amsmath,amscd}
\usepackage{amssymb}
\usepackage{amsthm}
\usepackage{yfonts}
\usepackage{mathrsfs,mathtools}
\usepackage{textcomp}
\usepackage{pictexwd,dcpic}
\usepackage{xfrac}
\usepackage{babel}
\usepackage{enumerate}
\usepackage{hyperref}
\usepackage{tikz}
\usetikzlibrary{cd}
\usepackage{graphicx}
\usepackage{anysize}
\usepackage{mdwlist}
\usepackage{imakeidx}
\usepackage{combelow}

\newtheorem{thm}{Theorem}[section]
\newtheorem{cor}[thm]{Corollary}
\newtheorem{lem}[thm]{Lemma}
\newtheorem{prop}[thm]{Proposition}
\theoremstyle{definition}
\newtheorem*{defi}{Definition}
\newtheorem{conj}{Conjecture}[]
\newtheorem{ex}{Example}[]
\newtheorem*{akn}{Acknowledgements}
\theoremstyle{remark}
\newtheorem{remark}[thm]{Remark}

\newcommand{\forkindep}[1][]{%
  \mathrel{
    \mathop{
      \vcenter{
        \hbox{\oalign{\noalign{\kern-.3ex}\hfil$\vert$\hfil\cr
              \noalign{\kern-.7ex}
              $\smile$\cr\noalign{\kern-.3ex}}}
      }
    }\displaylimits_{#1}
  }
}

\numberwithin{equation}{section}

\title{Categoricity of Shimura Varieties}
\author{Sebastian Eterovi\'c}
\date{\today}
\address{Mathematical Institute, University of Oxford, UK}
\email{eterovic$@$maths.ox.ac.uk}

\begin{document}

\maketitle

\begin{abstract}
    We propose a model-theoretic structure for Shimura varieties and give necessary and sufficient conditions to obtain categoricity. We show that these conditions are directly related to important conjectures in number theory coming from Galois representations attached the points of a Shimura variety. We end by showing that the existing literature is enough to prove categoricity of $\mathcal{A}_{2}$ and $\mathcal{A}_{3}$.
\end{abstract}

\tableofcontents

\section{Introduction}

Suppose that $p:X^{+}\rightarrow S(\mathbb{C})$ is a (connected) Shimura variety given by the Shimura datum $(G,X)$, that $E^{S}$ is the minimal abelian extension of the reflex field $E(G,X)$ over which $S$ is defined, and let $\Sigma$ be the set of coordinates of special points in $S(\mathbb{C})$. Briefly speaking then, the categoricity problem for $S$ is: if $q:X^{+}\rightarrow S(\mathbb{C})$ is any other map (of sets) satisfying the same algebraic conditions as $p$ (this will be made precise later: we will ask that $q$ satisfy the Special Subvarieties and Standard Fibers conditions), is there then a commutative diagram:
\begin{equation}
\label{eq:com}
\begin{CD}
X^{+} @>\xi>> X^{+}\\
@VVpV @VVqV\\
S(\mathbb{C}) @>\sigma>> S(\mathbb{C}),
\end{CD}    
\end{equation}
where $\xi:X^{+}\rightarrow X^{+}$ is a $G^{\mathrm{ad}}(\mathbb{Q})^{+}$-equivariant bijection which respects special domains (see \textsection \ref{subsec:special}), and $\sigma$ is a map induced by an automorphism of $\mathbb{C}$ fixing $E^{S}(\Sigma)$?

Another way of thinking about categoricity in this case is: how rigid is the special structure of $S$? Or, to what extent is $p$ determined by the algebraic properties of special subvarieties of $S$? A precise definition of categoricity is given at the beginning of \textsection \ref{sec:mt}.

Categoricity for Shimura curves was established in \cite{daw-harris}, and the authors acknowledged in their paper that the language they used may not be rich enough to treat the higher dimensional cases. The first objective of the present work then is to provide an enriched language which keeps the nice properties of the language used in the one-dimensional case: quantifier elimination, a useful description of types, and superstability. The second objective is to give necessary and sufficient conditions for Shimura varieties of any dimension to be categorical. It is immediate from the work done in \cite{daw-harris} (and from other similar categoricity results which will be addressed later in this Introduction) that some form of arithmetic open image condition is needed. Our main goal is to postulate the correct open image condition (here we split it into two conditions: FIC1 and FIC2) and prove that it is equivalent to categoricity. This open image condition is known to hold in dimension $1$ thanks to \cite[Theorem 1.1]{daw-harris}, but not known in general. In fact, a version of this condition had already been conjectured by Pink in \cite{pink} (see Conjecture \ref{conj:pink}), and he also shows the link between his conjecture and the Mumford-Tate conjecture. We will be able, however, to show that enough results are known in the literature to obtain categoricity of two well-known Shimura varieties: $\mathcal{A}_{2}$ and $\mathcal{A}_{3}$, the moduli spaces of principally polarised abelian varieties of dimension 2 and 3, respectively. As far as we know, this result is new.

The categoricity problem will be approached adapting techniques on $\ell$-atomicity and independent systems developed in \cite{bays-pillay} for studying commutative groups of finite Morley rank (e.g. abelian varieties). We will also show an approach using quasiminimal classes that gives the same result, but we will prefer the first method as it proves to be more robust than the quasiminimal approach. We show both approaches as they emphasize different aspects of the problem, the methods of \cite{bays-pillay} focusing more on the model-theoretic properties of the structures, while the quasiminimal approach focuses on the existence of certain curves within the Shimura variety.

After going over some necessary concepts and results from model-theory in \textsection\ref{sec:mt}, we give an introduction to Shimura varieties in \textsection\ref{sec:svi}. In \textsection\ref{sec:ss}, \textsection\ref{sec:sf} and \textsection\ref{sec:covers} we present the model-theoretic structures for Shimura varieties, we define the special subvarieties condition (here called $\mathrm{SS}$), we prove quantifier elimination of this language and give an axiomatisation for the complete first-order theory of $p:X^{+}\rightarrow S(\mathbb{C})$. We also give a description of the types in models of our theory that satisfy the standard fibres condition (here called SF). Using what we will call Shimura covers, we prove some stability properties of the first-order theory of a Shimura variety.

In \textsection\ref{sec:quasi} and \textsection\ref{sec:classification} we present our main categoricity results, first using quasiminimal classes, and then through independent systems. At the end of \textsection\ref{sec:remarks} we show that the categoricity conditions FIC 1 \& 2 for $\mathcal{A}_{2}$ and $\mathcal{A}_{3}$ can be obtained from the literature. 

This work follows an important series of results regarding categoricity of certain algebraic varieties arising from arithmetic: \cite{daw-harris} for categoricity of Shimura curves, \cite{bays-pillay} for categoricity of abelian and split semi-abelian varieties, \cite{bays-kirby} for categoricity conditions of pseudo-exponential maps, \cite{zilber2} for the general quasiminimal strategy for semi-abelian varieties, \cite{bays-zilber} and \cite{zilber3} for categoricity of multiplicative groups of algebraically closed fields, \cite{misha} for categoricity of abelian varieties under continuum hypothesis, \cite{bays} for categoricity of elliptic curves without complex multiplication; among others.

\section{Model Theory}
\label{sec:mt}
In this section we give some standard definitions and results from model theory needed for our work. We will always assume that the language $\mathcal{L}$ we are using is countable. For us the word \emph{type} means \emph{complete type}. Standard definitions and results from model theory can be found, for example, in \cite{tent-ziegler}.

\begin{defi}
Let $\kappa$ be an infinite cardinal. An $\mathcal{L}$-theory $T$ is \emph{$\kappa$-categorical} if all models of $T$ of cardinality $\kappa$ are isomorphic. We will say that $T$ is \emph{categorical} \index{categorical} if it is $\kappa$-categorical for all infinite $\kappa$.
\end{defi}

\begin{ex}
Consider the theory of algebraically closed fields: ACF\index{ACF}. We use $\mathrm{ACF}_{0}$ \index{a@$\mathrm{ACF}_{0}$} to denote the theory of algebraically closed fields of characteristic 0. Then $\mathrm{ACF}_{0}$ is $\kappa$-categorical for any uncountable $\kappa$, but it is not $\aleph_{0}$-categorical, as $\overline{\mathbb{Q}}$ and $\overline{\mathbb{Q}(X)}$ are non-isomorphic countable models of $\mathrm{ACF}_{0}$.
\end{ex}

\begin{prop}[see {{\cite[\textsection 3.3.4]{tent-ziegler}}}]
Let $K\models\mathrm{ACF}$ and $A\subseteq K$. Let $k$ be the subfield of $K$ generated by $A$. Then the $n$-types over $A$ are determined by prime ideals in the poynomial ring $k[X_{1},\ldots,X_{n}]$.  
\end{prop}

This means that if $m_{1},\ldots,m_{n}\in K$, then $\mathrm{tp}(m_{1},\ldots,m_{n}/A)$ is determined by the smallest algebraic subvariety $W\subseteq K^{n}$ defined over $k$ which contains the point $(m_{1},\ldots,m_{n})$. Any other point of $W$ which is generic in $W$ over $k$, will also realise the type $\mathrm{tp}(m_{1},\ldots,m_{n}/A)$.

\begin{defi}
An $\mathcal{L}$-structure $M$ is \emph{strongly minimal} \index{strongly minimal} if every definable subset of $M$ is either finite or cofinite. 
\end{defi}

\begin{ex}
$\mathrm{ACF}_{0}$ is strongly minimal (see \cite[\textsection 5.7]{tent-ziegler}).
\end{ex}

\subsection{Quantifier Elimination}

We recall a few standard definitions and results about theories with quantifier elimination. The exposition is taken from \cite{pillay}.

\begin{defi}
Let $M$ and $N$ be $\mathcal{L}$-structures. A \emph{finite partial isomorphism} \index{finite partial isomorphism} between $M$ and $N$ is a pair $\left(\overline{a},\overline{b}\right)$ with $\overline{a}\in M$ and $\overline{b}\in N$, such that $\mathrm{qftp}_{M}\left(\overline{a}\right) = \mathrm{qftp}_{N}\left(\overline{b}\right)$. 
\end{defi}

In other words, a finite partial isomorphism between $M$ and $N$ is the same thing as an isomorphism of $\mathcal{L}$-structures between a finitely generated substructure of $M$ and a finitely generated substructure of $N$. 

\begin{defi}
Let $M$ and $N$ be $\mathcal{L}$-structures. Let $I$ be a set of partial isomorphisms between $M$ and $N$. $I$ has \emph{back-and-forth} \index{back-and-forth} if for every partial isomorphism $\left(\overline{a},\overline{b}\right)\in I$ and every $c\in M$, there is $d\in N$ such that $\left(\overline{a}c,\overline{b}d\right)\in I$ and dually. 
\end{defi}

\begin{defi}
An $\mathcal{L}$-structure $M$ is \emph{$\aleph_{0}$-homogeneous} \index{h@$\aleph_{0}$-homogeneous} if for every finite subset $A\subseteq M$ and very $m\in M$, any elementary map $A\rightarrow M$ can be extended to an elementary map $A\cup\left\{m\right\}\rightarrow M$. 
\end{defi}

When talking about the class of models of a theory $T$, the previous concept will appear in the following two variants:
\begin{enumerate}
    \item \emph{$\aleph_{0}$-homogeneity over the empty set}.\index{h@$\aleph_{0}$-homogeneous!over the empty set} This will mean that, given any two models of $T$, the set $I$ of all finite partial isomorphism between them has back-and-forth. 
    \item \emph{$\aleph_{0}$-homogeneity over countable models}.\index{h@$\aleph_{0}$-homogeneous!over countable models} For this we need to adapt the definition of finite partial isomorphism. Given $M,M',N\models T$, where $N$ is a countable submodel common to $M$ and $M'$, a \emph{partial isomorphism over $N$} is a pair $\left(\overline{a},\overline{b}\right)$ with $\overline{a}\in M$ and $\overline{b}\in M'$, such that $\mathrm{qftp}_{M}\left(\overline{a}/N\right) = \mathrm{qftp}_{M'}\left(\overline{b}/N\right)$. Now we can define $\aleph_{0}$-homogeneity over countable models in an analogous way to $\aleph_{0}$-homogeneity over the empty set.
\end{enumerate}

The following proposition gives us the strategy we will use later for proving quantifier elimination for Shimura structures. 

\begin{prop}[see {{\cite[Proposition 2.29]{pillay}}}]
\label{prop:qe}
Let $T$ be an $\mathcal{L}$-theory. The following are equivalent:
\begin{enumerate}[(a)]
    \item $T$ has quantifier elimination.
    \item If $M$ and $N$ are $\omega$-saturated models of $T$ and $I$ is the set of all finite partial isomorphism between $M$ and $N$, then $I$ has back-and-forth.
\end{enumerate}
\end{prop}

The same strategy can also be used to prove completeness of the theory. 

\begin{prop}[see {{\cite[Proposition 2.30]{pillay}}}]
\label{prop:completeness}
Suppose that for any two $\omega$-saturated models $M$ and $N$ of $T$, the set of finite partial isomorphisms between $M$ and $N$ has back-and-forth and is non-empty. Then $T$ is complete. 
\end{prop}

\subsection{Pregeometries}

\begin{defi}
Let $X$ be a set. A function $\mathrm{cl}:\mathcal{P}(X)\rightarrow\mathcal{P}(X)$ (here $\mathcal{P}(X)$ denotes the power set of $X$) is called a \emph{pregeometry} \index{pregeometry} on $X$ if it satisfies the following properties for every $A,B\in\mathcal{P}(X)$:
\begin{enumerate}[(a)]
\item $A\subseteq\mathrm{cl}(A)$,
\item If $A\subseteq B$, then $\mathrm{cl}(A)\subseteq\mathrm{cl}(B)$,
\item $\mathrm{cl}(\mathrm{cl}(A)) = \mathrm{cl}(A)$,
\item Finite character: if $a\in\mathrm{cl}(A)$, then there is a finite subset $A_{0}\subseteq A$ such that $a\in \mathrm{cl}(A_{0})$,
\item Exchange: if $a,b\in X$ are such that $a\in\mathrm{cl}(A\cup\left\{b\right\})$ and $a\notin\mathrm{cl}(A)$, then $b\in\mathrm{cl}(A\cup\left\{a\right\})$.
\end{enumerate}
If $\mathrm{cl}$ is a pregeometry on $X$, then a subset $A\subseteq X$ is called \emph{$\mathrm{cl}$-closed} \index{pregeometry!closure} if $\mathrm{cl}(A) = A$.
\end{defi}

A crucial aspect of pregeometries is that they allow us to have well-defined notions of independence and dimension. Suppose $\mathrm{cl}$ is a pregeometry on $X$. A set $A\subseteq X$ is \emph{$\mathrm{cl}$-independent} \index{pregeometry!independence} if for every $a\in A$ we have that $a\notin\mathrm{cl}(A\setminus\left\{a\right\})$. The \emph{dimension} \index{pregeometry!dimension} of a set $B\subseteq X$ with respect to $\mathrm{cl}$ is the cardinality of any $\mathrm{cl}$-independent set $A\subseteq B$ such that $\mathrm{cl}(A)=B$. For more basic properties of pregeometries, see \cite[Appendix C]{tent-ziegler}.

\begin{ex}[Some pregeometries]
We give two well-known examples of pregeometries and the dimensions they define.
\begin{enumerate}
    \item If $X$ is a $K$-vector space, we can take $\mathrm{cl}(A) := \mathrm{lin.span}_{K}(A)$. In this case $\dim A = \mathrm{lin.dim.}_{K} (\mathrm{cl}(A))$. 
    \item If $X$ is a field, we can take $\mathrm{cl}(A)$ to be the set of elements of $X$ which are algebraic over $A$. In this case $\dim A = \mathrm{tr.deg.}(A)$. This pregeometry is usually denoted (and this is the name we will use): $\mathrm{acl}$. \index{acl} This example is an instance of a more general result, namely that the model-theoretic notion of algebraic closure defines a pregeometry on strongly minimal sets (see \cite[Theorem 5.7.5]{tent-ziegler}). 
\end{enumerate} 
\end{ex}

\subsection{Quasiminimality}
\label{subsec:quasi}
In this subsection we introduce the terminology related to quasiminimality and quasiminimal structures as presented in \cite{fiveguys} and \cite{kirby2}. The final theorem of this subsection provides the strategy we will follow to obtain one of our main results: Theorem \ref{thm:main1}. As always, we assume $\mathcal{L}$ is countable.

\begin{defi}
An $\mathcal{L}$-structure $M$ is \emph{quasiminimal} \index{quasiminimal} is every definable subset of $M$ is countable or co-countable.
\end{defi}

\begin{defi}
Let $M$ be an $\mathcal{L}$-structure. Let $\mathrm{cl}$ be a pregeometry on $M$. Then the pair $(M,\mathrm{cl})$ is called a \emph{quasiminimal pregeometry structure} \index{quasiminimal!pregeometry structure} if:
\begin{description}
    \item[(QM1)] $\mathrm{cl}$ is determined by the language: given $a,a',b_{1},\ldots,b_{n},b_{1}',\ldots,b_{n},\in M$ such that $\mathrm{tp}\left(a,\overline{b}\right)=\mathrm{tp}\left(a',\overline{b}'\right)$, then $a\in\mathrm{cl}\left(\overline{b}\right)$ if and only if $a'\in\mathrm{cl}\left(\overline{b}'\right)$. 
    \item[(QM2)] $M$ is infinite dimensional with respect to $\mathrm{cl}$. 
    \item[(QM3)] \emph{Countable Closure Property}. If $A\subseteq M$ is finite, then $\mathrm{cl}(A)$ is countable.
    \item[(QM4)] \emph{Uniqueness of the generic type}. Suppose $C,C'\subseteq M$ are countable closed subsets, enumerated such that $\mathrm{tp}(C) = \mathrm{tp}(C')$. If $a\in M\setminus C$ and $a'\in M\setminus C'$, then $\mathrm{tp}(C,a) = \mathrm{tp}\left(C',a'\right)$ (with respect to the same enumerations for $C$ and $C'$).
    \item[(QM5)] \emph{$\aleph_{0}$-homogeneity over closed sets and the empty set}. Let $C,C'\subseteq M$ be countable closed subsets or empty, enumerated such that $\mathrm{tp}(C) = \mathrm{tp}\left(C'\right)$, and let $\overline{b},\overline{b}'$ be finite tuples from $M$ such that $\mathrm{tp}\left(C,\overline{b}\right) = \mathrm{tp}\left(C',\overline{b}'\right)$, and let $a\in\mathrm{cl}\left(H,\overline{b}\right)$. Then there is $a'\in M$ such that $\mathrm{tp}\left(C,\overline{b},a\right) =\mathrm{tp}\left(C,'\overline{b}',a'\right)$. 
\end{description}
$(M,\mathrm{cl})$ is said to be a \emph{weakly quasiminimal pregeometry structure} \index{quasiminimal!weak pregeometry structure} if it satisfies all the axioms except maybe (QM2).
\end{defi}

\begin{defi}
Let $(M,\mathrm{cl})$ be a quasiminimal pregeometry structure. Let $\mathscr{K}^{-}(M)$ be the smallest class of $\mathcal{L}$-structures that contains $M$, all the $\mathrm{cl}$-closed substructures of $M$, and that is closed under isomorphism. Let $\mathscr{K}(M)$ be the smallest class containing $\mathscr{K}^{-}(M)$ and that is closed under unions of chains of closed embeddings. Any class of structures of the form $\mathscr{K}(M)$ \index{k@$\mathscr{K}(M)$} is called a \emph{quasiminimal class}. \index{quasiminimal!class}
\end{defi}

Now we present the main result about categoricity and quasiminimal classes. 

\begin{thm}[see {{\cite[Theorem 2.3]{fiveguys}}}]
\label{thm:quasi}
If $\mathscr{K}$ is a quasiminimal class, then every structure $A\in\mathscr{K}$ is a weakly quasiminimal pregeometry structure and, up to isomorphism, there is exactly one structure in $\mathscr{K}$ of each cardinal dimension. In particular, $\mathscr{K}$ is uncountably categorical. Furthermore, $\mathscr{K}$ is the class of models of an $\mathcal{L}_{\omega_{1},\omega}$-sentence. 
\end{thm}

It will be important for us that this theorem can be applied to structures defined by $\mathcal{L}_{\omega_{1},\omega}$-sentences \index{l@$\mathcal{L}_{\omega_{1},\omega}$-sentence} (which are sentences that can be infinitely long, with countably many conjunctions and disjunctions, but with only finitely many variables) because the Standard Fibres condition (see Subection \ref{sec:sf}) will be such a sentence.

\begin{prop}[see {{\cite[Corollary 2.2]{kirby2}}}]
\label{prop:kirby}
Let $\mathscr{K}$ be a quasiminimal class. Then models of $\mathscr{K}$  of dimension less than or equal to $\aleph_{1}$ are determined up to isomorphism by their dimension. There is at most one model of cardinality $\aleph_{1}$.
\end{prop}

\subsection{\texorpdfstring{$\ell$}{l}-Isolation}

The definitions and results of this subsection can be found in \cite[\textsection 3]{bays-pillay}.

\begin{defi}
Given a type $p\in S(B)$ and a formula $\phi(x,y)$, the \emph{complete $\phi$-type implied by $p$}, \index{p@$\phi$-type} denoted $p_{\phi}$, is the set of all formulas of the form $\phi\left(x,\overline{b}\right)$ that are in $p$, where $\overline{b}$ is a tuple from $B$.

A type $p$ is \emph{$\ell$-isolated} \index{l@$\ell$!isolated} (which is short for \emph{locally isolated}) \index{locally isolated type} if for each $\phi(x,y)$, there exists $\psi(x)\in p$ such that $\psi$ implies the complete $\phi$-type implied by p, i.e. $\psi\models p_{\phi}$.

$A$ is said to be \emph{atomic} \index{atomic} over $B$ if $\mathrm{tp}(a/B)$ is isolated for each tuple $a\in A$. $A$ is \emph{constructible} \index{constructible} over $B$ if $A$ has an enumeration $(a_{i})_{i\leq\lambda}$ such that $\mathrm{tp}\left(a_{i}/Ba_{\leq i}\right)$ is isolated for each $i\leq\lambda$, where $Ba_{\leq i} = B\cup\left\{a_{j} : j\leq i\right\}$.

We say $A$ is \emph{$\ell$-atomic} \index{l@$\ell$!atomic} and \emph{$\ell$-constructible} \index{l@$\ell$!constructible} over $B$ by replacing isolation in the definitions above with $\ell$-isolation. 
\end{defi}

\begin{remark}
It is easy to see that isolated types are $\ell$-isolated, so atomicity implies $\ell$-atomicity and constructibility implies $\ell$-constructibility. Also constructibility implies atomicity and $\ell$-constructibility implies $\ell$-atomicity (\cite[Theorem IV.3.2]{shelah}). The converse statements hold for countable sets (\cite[Lemma IV.3.16]{shelah}). 
\end{remark}

\begin{lem}[see {{\cite[Lemma 3.6-a)]{bays-pillay}}}]
\label{lem:ell}
Let $\mathbb{M}$ be a monster model of a complete stable theory $T$. 
\begin{enumerate}[(a)]
\item For $A\subseteq\mathbb{M}^{\mathrm{eq}}$, there exists $M\prec\mathbb{M}$ such that $A\subseteq M^{\mathrm{eq}}$ and $M^{\mathrm{eq}}$ is $\ell$-constructible over $A$.
\item Let $M\prec\mathbb{M}$. If $\phi$ is a formula over $M$ such that $\phi(M)\subseteq A\subseteq \mathbb{M}^{\mathrm{eq}}$ and $\mathrm{dcl}^{\mathrm{eq}}(A)\cap\mathrm{dcl}^{\mathrm{eq}}(\phi(\mathbb{M}))\subseteq M^{\mathrm{eq}}$, and if $b$ is $\ell$-isolated over $A$ and $\models \phi(b)$, then $b\in\phi(M)$.
\end{enumerate}
\end{lem}

\subsection{Independent Systems}
Here we give some definitions and results from \cite[\textsection 3.5]{bays-pillay} necessary for the classification of Shimura covers in Section \ref{sec:covers}. Let $T$ be a stable theory and $\mathbb{M}$ be a monster model of $T$.

\begin{defi}
Let $I$ be a downward-closed set of sets. An \emph{$I$-system} \index{i@$I$-system} in $\mathbb{M}$ is a collection $\left(M_{s}\right)_{s\in I}$ of elementary submodels of $\mathbb{M}$ such that for $s\subseteq t$, $M_{s}$ is an elementary submodel of $M_{t}$. Given $J\subseteq I$, let $M_{J}:= \bigcup_{s\in J}M_{s}\subseteq\mathbb{M}$.

For $s\in I$, define $<s:=\mathcal{P}(s)\setminus\left\{s\right\}$ and $\ngeq s:= I\setminus\left\{t : s\subseteq t\right\}$.\index{$<s$} \index{$\ngeq s$}

The $I$-system $\left(M_{s}\right)_{s\in I}$ is called \emph{constructible} (resp. \emph{atomic}, \emph{$\ell$-constructible}, \emph{$\ell$-atomic})\index{i@$I$-system!atomic}\index{i@$I$-system!constructible}\index{i@$I$-system!$\ell$-atomic}\index{i@$I$-system!$\ell$-constructible} if $M_{s}$ is constructible (resp. atomic, $\ell$-constructible, $\ell$-atomic) over $M_{<s}$ for all $s\in I$ with $|s|>1$.

The $I$-system is called \emph{independent} \index{i@$I$-system!independent} (or \emph{stable}) if $M_{s}\forkindep[M_{<s}] M_{\ngeq s}$ for all $s\in I$.

The $I$-system is \emph{Noetherian} \index{i@$I$-system!noetherian} if each $s\in I$ is finite.

An \emph{enumeration}\index{enumeration}\index{i@$I$-system!enumeration} of $I$ is a sequence $(s_{i})_{i\leq\lambda}$, where $\lambda$ is an ordinal, such that $I=\left\{s_{i} : i\in \lambda\right\}$ and $s_{i}\subseteq s_{j}\implies i\leq j$. We write $s_{<i}$ for $\left\{s_{j} : j<i\right\}$. \index{$s_{<i}$}
\end{defi}

\begin{lem}[see  {{\cite[Lemma XII.2.3(1)]{shelah}}}]
\label{lem:fact}
Let $(M_{s})_{s}$ be an $I$-system, and let $(s_{i})_{i\in\lambda}$ be an enumeration. Then the system is independent if and only if $M_{s_{i}}\forkindep[M_{<s_{i}}]M_{s_{<i}}$ holds for all $i\in\lambda$.
\end{lem}





\begin{lem}[see {{\cite[Lemma 3.21]{bays-pillay}}}]
\label{lem:3.21}
Let $(M_{s})_{s}$ be a constructible Noetherian independent $I$-system. Suppose that for each $p\in\bigcup I$, $B_{p}$ is a subset of $M_{\left\{p\right\}}$ for which $M_{\left\{p\right\}}$ is constructible over $M_{\emptyset}B_{p}$. Then $M_{I}$ is constructible over $A:=M_{\emptyset}\cup\bigcup_{p\in\bigcup I}B_{p}$.
\end{lem}

\begin{lem}[see {{\cite[Lemma 3.22]{bays-pillay}}}]
\label{lem:3.22}
Suppose $\bigcup I$ is finite. For an $\ell$-atomic $I$-system $(M_{s})_{s}$ to be independent, it suffices to show that for each $p\in\bigcup I$
\begin{equation*}
M_{\left\{p\right\}}\forkindep[M_{\emptyset}]M\ngeq\left\{p\right\}.
\end{equation*}
\end{lem}

\section{Introduction to Shimura Varieties}
\label{sec:svi}
A standard introduction to this subject is \cite{milne}, which is a very complete treatment of Shimura varieties. On the other hand, as we will only be dealing with connected varieties, \cite{pink} is also a very good starting point. We will use both sources in this work. An important observation though is that in \cite{pink} the author treats ``connected mixed Shimura varieties'', whereas \cite{milne} mainly refers to ``pure Shimura varieties''. In our work we will work with pure Shimura varieties, only because there are a few results which hold for the pure case for which we could not find a reference in the mixed case. We expect that the methods of this paper can be generalised to the mixed case. Whenever possible, we will point out if a certain result can be easily extended from the pure case to the mixed case.

\subsection{Congruence Subgroups}
Let $G\subseteq\mathrm{GL}_{n}(\mathbb{C})$ be a linear algebraic group over $\mathbb{Q}$ and let $G(\mathbb{Z})$ denote $G(\mathbb{Q})\cap\mathrm{GL}_{n}(\mathbb{Z})$. 

\begin{defi}
 A subgroup $\Gamma$ of $G(\mathbb{Q})$ is \emph{arithmetic} \index{arithmetic group} if $\Gamma$ has finite index in both $\Gamma$ and $G(\mathbb{Z})$, in other words, $\Gamma$ and $G(\mathbb{Z})$ are commensurable. 
\end{defi}

Next we define a distinguished family of arithmetic subgroups of $G(\mathbb{Z})$. 

\begin{defi}
Given an embedding of $G$ into $\mathrm{GL}_{n}$, the \emph{principal congruence subgroup of level $N$} \index{congruence subgroup!principal} in $G$ is:
\begin{equation*}
    \Gamma(N):=\left\{g \in G(\mathbb{Z}) : g\equiv \mathrm{id}\; \textrm{mod}\: N\right\}.
\end{equation*}
This definition depends on the embedding. To address this issue, we define a \emph{congruence subgroup} \index{congruence subgroup} of $G$ to be a subgroup $\Gamma$ of $G$ such that for some $N$, $\Gamma(N)$ is a finite-index subgroup of $\Gamma$. This definition does not depend on the embedding. 
\end{defi}

\begin{defi}
An element $g\in G(\mathbb{Q})$ is called \emph{neat}\index{neat!element} if for some (for all) faithful representation $\rho:G\hookrightarrow\mathrm{GL}(V)$, the eigenvalues of $\rho(g)$ in $\mathbb{C}$ generate a torsion-free subgroup of $\mathbb{C}^{\times}$. A subgroup of $G(\mathbb{Q})$ is \emph{neat}\index{neat!subgroup}, if all its elements are neat.
\end{defi}

\begin{prop}[see {{\cite[Proposition 3.5]{milne}}}]
\label{prop:borel}
Let $G$ be an algebraic group over $\mathbb{Q}$, and let $\Gamma$ be an arithmetic subgroup of $G(\mathbb{Q})$. Then $\Gamma$ contains a neat subgroup $\Gamma'$ of finite index. Moreover, $\Gamma'$ can be defined by congruence relations (i.e. for some embedding $G\hookrightarrow\mathrm{GL}_{n}$ and integer $N$, $\Gamma' = \left\{g\in\Gamma : g\equiv\mathrm{id}\, \mathrm{mod }\, N\right\}$).
\end{prop}

\subsection{Connected Shimura Varieties}
Let $\mathbb{S}$\index{s@$\mathbb{S}$} be the \emph{Deligne torus}, \index{Deligne torus} that is, the linear algebraic group over $\mathbb{R}$ such that $\mathbb{S}(\mathbb{R})=\mathbb{C}^{\times}$. 

A \emph{Shimura datum} \index{Shimura datum} is a pair $(G,X)$ \index{$(G,X)$} consisting of a linear algebraic group $G$ over $\mathbb{Q}$ and the $G(\mathbb{R})$-conjugacy class $X$ of a morphism $h:\mathbb{S}\rightarrow G_{\mathbb{R}}$ satisfying some conditions on $G$, the complex structure of $X$, and the action of $G$ on $X$: see axioms SV1-SV3 in \cite[p. 302]{milne} for the ``pure'' axioms, or \cite[Definition 2.1]{pink} for the ``mixed'' ones. More important than the precise statement of these axioms is that they guarantee that quotients of connected components of $X$ by certain arithmetic subgroups will give an algebraic variety. For this work, we will use the definitions and notation of \cite{milne} for pure Shimura varieties. 

\begin{remark}
\label{rem:behaviour}
The fact that we do not need the precise statement of axioms of Shimura varieties could be interpreted as saying that the work we will do here may be replicated on other algebraic varieties obtained as the quotient of some complex domain by the action of a discrete group. However, as it will become apparent, what really makes the model-theory work is the nice behaviour of the geometry of the ``special'' subvarieties (see \textsection \ref{subsec:special}). Therefore, the present work, as well as \cite{bays-pillay} and \cite{daw-harris}, should be seen as giving insight to the essential properties of the ``special structure'' of these algebraic varieties. Finding an abstract definition of ``special structure'' is an important question in Diophantine geometry and some attempts at this definition have been made, see for instance \cite{ullmo}.
\end{remark}

Let $G^{\mathrm{ad}}$ \index{g@$G^{\mathrm{ad}}$} denote $G$ modulo its centre, let $G^{\mathrm{ad}}(\mathbb{R})^{+}$ be the connected component of the identity in $G^{\mathrm{ad}}(\mathbb{R})$, and let $G^{\mathrm{ad}}(\mathbb{Q})^{+}:=G^{\mathrm{ad}}(\mathbb{Q})\cap G^{\mathrm{ad}}(\mathbb{R})^{+}$.\index{g@$G^{\mathrm{ad}}(\mathbb{Q})^{+}$} Let $X^{+}$ \index{x@$X^{+}$} denote a connected component of $X$; its stabiliser in $G^{\mathrm{ad}}(\mathbb{R})$ is $G^{\mathrm{ad}}(\mathbb{R})^{+}$. We call $(G,X^{+})$ a \emph{connected Shimura datum}. \index{Shimura datum!connected} 

Let $K$ be a compact open subgroup of $G\left(\mathbb{A}_{f}\right)$ (where $\mathbb{A}_{f}$ denotes the finite rational ad\`eles). Define the double coset space:
\begin{equation}\index{s@$\mathrm{Sh}_{K}(G,X)$}
\label{eq:shimuradef}
    \mathrm{Sh}_{K}(G,X) := G(\mathbb{Q})\backslash\left(X\times\left(G\left(\mathbb{A}_{f}\right)/K\right)\right),
\end{equation}
where $G(\mathbb{Q})$ acts diagonally. We will denote the elements of this space by $[x,gK]$. The following lemma will show that $\mathrm{Sh}_{K}(G,X)$ can be seen as a finite disjoint union of connected components. To that end, let $G(\mathbb{Q})_{+}$ denote the stabiliser of $X^{+}$ in $G(\mathbb{Q})$, and let $\mathscr{C}$ be a set of representatives for the finite set $G(\mathbb{Q})_{+}\backslash G(\mathbb{A}_{f})/K$. Given $g\in\mathscr{C}$, let $\Gamma_{g}:=G(\mathbb{Q})_{+}\cap gKg^{-1}$ (which is a congruence subgroup of $G$), and let $[x]_{\Gamma_{g}}$ denote the class of an element $x\in X^{+}$ in the quotient $\Gamma_{g}\backslash X^{+}$. 

\begin{lem}[{{see \cite[Lemma 5.13]{milne}}}]
If $X^{+}$ is a connected component of $X$ and we endow $G(\mathbb{A}_{f})$ with the ad\`elic topology, then there is a homeomorphism
\begin{equation}
\label{eq:union}
    \coprod_{g\in\mathscr{C}}\Gamma_{g}\backslash X^{+} \cong G(\mathbb{Q})\backslash\left(X\times\left(G\left(\mathbb{A}_{f}\right)/K\right)\right)
\end{equation}
which is given by
\begin{align*}
    \Gamma_{g}\backslash X^{+}&\rightarrow G(\mathbb{Q})\backslash\left(X\times\left(G\left(\mathbb{A}_{f}\right)/K\right)\right)\\
    [x]_{\Gamma_{g}}&\mapsto [x,gK].
\end{align*}
\end{lem}

The axioms of Shimura data ensure that, when the $\Gamma_{g}$ are sufficiently small, the members of the disjoint union in (\ref{eq:union}) have canonical realisations as quasi-projective varieties over $\mathbb{C}$. In fact, by \cite[Theorem 3.12]{milne} and \cite[Fact 2.3]{pink}, the $\Gamma_{g}$ are sufficiently small when their images in $G^{\mathrm{ad}}(\mathbb{Q})^{+}$ are neat. \textbf{Unless stated otherwise, we will assume from now on that $\Gamma$ is neat.}

\begin{remark}
Even though we assume $\Gamma$ to be neat, a straight-forward generalisation of \cite[\textsection 5.3]{daw-harris} shows that one can obtain a commutative diagram like (\ref{eq:com}) for the general case from the neat case. This is done in \textsection \ref{subsec:neat}.
\end{remark}

$\mathrm{Sh}_{K}(G,X)$ possesses a canonical model over its \emph{reflex field} \index{reflex field} $E:=E(G,X)$ which is independent of the choice of $K$ (see \cite[\textsection 14]{milne} for an exposition and references). For a precise definition of the reflex field, look at \cite[\textsection 12]{milne}. We point out that $E$ is a number field (see \cite[Remark 12.3-a)]{milne}). 

In the notation of (\ref{eq:shimuradef}), let $\Gamma = \Gamma_{e} = G(\mathbb{Q})_{+}\cap K$ \index{$\Gamma$} and let $S(\mathbb{C}) := \Gamma\backslash X^{+}$, which (by the theory of Galois actions on connected components of the canonical model) is defined over a finite abelian extension of $E$ which we will call $E^{S}$. Let $p:X^{+}\rightarrow S(\mathbb{C})$ denote the surjective holomorphic map given by automorphic forms (i.e. $p$ is the quotient map composed with the embedding into projective space). In \cite{milne} and \cite{pink}, $\mathrm{Sh}_{K}(G,X)$ is called a \emph{Shimura variety}, \index{Shimura variety} while the connected component $S$ is called a \emph{connected Shimura variety}. \index{Shimura variety!connected} As our work is only concerned with the connected case, we will use the term ``Shimura variety'' to refer to connected components, like $S$.

We will indulge in a slight abuse of notation because, as we will almost always be working in $G^{\mathrm{ad}}(\mathbb{Q})^{+}$, we will use the same name for $\Gamma$ and for the image of $\Gamma$ in $G^{\mathrm{ad}}(\mathbb{Q})^{+}$. 

\begin{ex}[The $j$ function]
Take $G = \mathrm{GL}_{2}$. Let $X$ be the $\mathrm{GL}_{2}(\mathbb{R})$-conjugacy class of the morphism $h:\mathbb{S}\rightarrow\mathrm{GL}_{2,\mathbb{R}}$, which on real points is given by
\begin{equation*}
    a+ib\mapsto\left(\begin{matrix} a & -b\\
    b & a\end{matrix}\right).
\end{equation*}
Then we can see $X$ as the union of the upper and lower half-planes of $\mathbb{C}$; $X = \mathbb{H}^{+}\cup\mathbb{H}^{-}$. Let $K=\mathrm{GL}_{2}\left(\widehat{\mathbb{Z}}\right)$, so that $\Gamma =\mathrm{SL}_{2}(\mathbb{Z})$. Then the map $p:X^{+}\rightarrow S(\mathbb{C})$ becomes the $j$-function: $j:\mathbb{H}^{+}\rightarrow\mathbb{C}$. Note that this is a case where $\Gamma$ is not torsion-free, but, because it is so well-known, it is a good example to keep in mind as we go along.
\end{ex}

\begin{ex}[Siegel moduli spaces]
\label{ex:ppav}
Here we generalise the previous example to construct the moduli space of principally polarised abelian varieties of dimension $g$ (following well-established tradition, for this example we reserve the letter $g$ to be the dimension of the abelian varieties, \underline{not} an element of the group $G$). Let:
\begin{equation*}
    J_{g} := \left(\begin{matrix} 0 & I_{g}\\
    -I_{g} & 0
    \end{matrix}\right),
\end{equation*}
where $I_{g}$ is the $g\times g$ identity matrix. Consider:
\begin{equation*}
    \mathrm{GSp}_{2g}(\mathbb{R}) := \left\{M\in\mathrm{GL}_{2g}(\mathbb{R}) : M^{t}J_{g}M = v(M)J_{g}\right\},
\end{equation*}
where $v(M)\in\mathbb{R}^{\times}$. Then $v:\mathrm{GSp}_{2g}\rightarrow\mathbb{G}_{m}$ is a homomorphism of linear algebraic groups. So we can define $\mathrm{Sp}_{2g}$ as the kernel of $v$, that is:
\begin{equation*}
    \mathrm{Sp}_{2g} := \left\{M\in\mathrm{GL}_{2g} : M^{t}J_{g}M = J_{g}\right\}.
\end{equation*}
Just like in the previous example, we define:
\begin{align*}
    h:\,&\mathbb{C}^{\times}\longrightarrow\mathrm{GSp}_{2g}(\mathbb{R})\\
    &a+ib\mapsto aI_{2g}+bJ_{g}
\end{align*}
In this case, $X := \mathrm{GSp}_{2g}(\mathbb{R})\cdot h$ can be identified with the disjoint union $\mathbb{H}_{g}^{+}\cup\mathbb{H}_{g}^{-}$, where:
\begin{equation*}
    \mathbb{H}_{g}^{+} := \left\{Z=A+iB\in M_{g\times g}(\mathbb{C}) : Z=Z^{t}, B>0\right\}
\end{equation*}
is called the \emph{Siegel upper half-space}. $\mathbb{H}_{g}^{-}$ is defined in the same way, but changing $B>0$ for $B<0$. The action of $\mathrm{GSp}_{2g}(\mathbb{R})$ on $\mathbb{H}_{g}^{+}\cup\mathbb{H}^{-}$ is given by:
\begin{equation*}
    \left(\begin{matrix}
    A & B\\
    C & D
    \end{matrix}\right)\cdot Z := (AZ + B)(CZ+D)^{-1}.
\end{equation*}
Now we take $K = \mathrm{GSp}_{2g}\left(\widehat{\mathbb{Z}}\right)$, so $\Gamma = \mathrm{Sp}_{2g}(\mathbb{Z})$. The quotient $\mathcal{A}_{g}(\mathbb{C}) := \Gamma\backslash\mathbb{H}_{g}^{+}$ \index{a@$\mathcal{A}_{g}$} is the (coarse) moduli space of principally polarised abelian varieties of dimension $g$. These spaces are also known as \emph{Siegel moduli spaces}.
\end{ex}

\begin{remark}[Products of Shimura data]
\label{rem:shimprod}
Let $\left(G_{1},X_{1}\right)$ and $\left(G_{2}, X_{2}\right)$ be Shimura data, with associated connected Shimura varieties $S_{1}(\mathbb{C}) = \Gamma_{1}\backslash X_{1}^{+}$ and $S_{2}(\mathbb{C})= \Gamma_{2}\backslash X_{2}^{+}$. Then $\left(G_{1}\times G_{2}, X_{1}\times X_{2}\right)$ is a Shimura datum, $X_{1}^{+}\times X_{2}^{+}$ is a connected component of $X_{1}\times X_{2}$ whose stabiliser in $(G_{1}\times G_{2})^{\mathrm{ad}}(\mathbb{R}) = (G_{1}^{\mathrm{ad}}\times G_{2}^{\mathrm{ad}})(\mathbb{R})$ is $G_{1}^{\mathrm{ad}}(\mathbb{R})^{+}\times G_{2}^{\mathrm{ad}}(\mathbb{R})^{+}$. And so we can see $S_{1}\times S_{2}$ as the Shimura variety with datum $\left(G_{1}\times G_{2}, X_{1}\times X_{2}\right)$ whose complex points are given by $\left(\Gamma_{1}\backslash X_{1}^{+}\right)\times\left(\Gamma_{2}\backslash X_{2}^{+}\right)$.
\end{remark}

\subsection{Classifying Points}
Let $(G,X)$ be a Shimura datum. From the definition, a point $x\in X$ corresponds to a homomorphism $x:\mathbb{S}\rightarrow G_{\mathbb{R}}$. We can then classify the points of $X^{+}$ in terms of the image of the homomorphism they represent in the following way: the \emph{Mumford-Tate group} \index{Mumford-Tate group} of $x$, denoted $\mathrm{MT}(x)$, is the smallest algebraic subgroup $H$ of $G$ defined over $\mathbb{Q}$ such that $x$ factors through $H_{\mathbb{R}}$. 

\begin{defi}
Let $x\in X^{+}$ and $p(x) = z\in S(\mathbb{C})$. The \emph{Hecke-orbit} of $z$ is the set of points $z'\in S(\mathbb{C})$ for which there exists $g\in G^{\mathrm{ad}}(\mathbb{Q})^{+}$ such that $p(gx) = z$.
\end{defi}

Points in the same Hecke-orbit will have the same Mumford-Tate group. 

\begin{defi}
A point $x\in X$ is said to be \emph{special} \index{special point} if $\mathrm{MT}(x)$ is a subtorus of $G$. On the other hand, $x$ is called \emph{Hodge-generic} \index{Hodge-generic!point} if $\mathrm{MT}(x)=G$. 
\end{defi}

If we want to develop model theory of Shimura varieties, we should use a language that is capable of recognising points of $X^{+}$ with different Mumford-Tate group (this is done in Remark \ref{rem:mt2}). We will show some results in the following subsection on special subvarieties that will eventually allow us to do that. In the mean time, we recall two results that show that the action of $G^{\mathrm{ad}}(\mathbb{Q})^{+}$ on $X^{+}$ can already distinguish between special points and Hodge-generic points in the case of pure Shimura varieties.

\begin{thm}[see {{\cite[Theorem 2.3]{daw-harris}}}]
\label{thm:special}
Let $(G,X)$ be a pure Shimura datum and let $X^{+}$ be a connected component of $X$. For any $x\in X^{+}$ the following are equivalent:
\begin{enumerate}
    \item $x$ is a special point.
    \item There exists $g\in G^{\mathrm{ad}}(\mathbb{Q})^{+}$ with the property that $x$ is the unique fixed point of $g$ in $X^{+}$.
\end{enumerate}
\end{thm}

\begin{prop}[see {{\cite[Lemma 2.2]{ullmo-yafaev}}}]
\label{prop:hodgegen}
Let $(G,X)$ be a pure Shimura datum and let $X^{+}$ be a connected component of $X$. Let $x\in X^{+}$ be Hodge-generic and suppose $g\in G^{\mathrm{ad}}(\mathbb{Q})^{+}$ fixes $x$. Then $g$ is the identity. 
\end{prop}

\begin{ex}[Special points of $\mathcal{A}_{g}$]
Using Theorem \ref{thm:special} we know that the points of $\mathbb{H}^{+}_{g}$ which are special are those which are fixed by some element of $\mathrm{GSp}_{2g}(\mathbb{Q})^{+}$. These are precisely the points that correspond to principally polarised abelian varieties with complex multiplication (CM for short)\index{CM abelian variety}. 
\end{ex}

\subsection{Special Subvarieties} 
\label{subsec:special}
Let $(G,X)$ and $(H,Y)$ be Shimura data. Let $\phi:H\rightarrow G$ be a homomorphism of algebraic groups over $\mathbb{Q}$ that induces a map $Y^{+}\rightarrow X^{+}$ given by $y\mapsto\phi\circ y$. In this case we say that $\phi:(H,Y^{+})\rightarrow(G,X^{+})$ is a \emph{morphism of connected Shimura data}. \index{Shimura datum!morphism} If furthermore $\phi$ is an inclusion, then we say that $(H,Y^{+})$ is a \emph{connected Shimura subdatum}\index{Shimura datum!subdatum} of $(G,X^{+})$.

Let $\Gamma^{G}\leq G$ and $\Gamma^{H}\leq H$ be congruence subgroups and let $S_{1} = \Gamma^{H}\backslash Y^{+}$ and $S_{2} = \Gamma^{G}\backslash X^{+}$ be the corresponding connected Shimura varieties. If in addition to being a morphism of Shimura varieties, $\phi\left(\Gamma^{H}\right)\subseteq\Gamma^{G}$, then $\phi$ induces a map in the quotients $[\phi]:\Gamma^{H}\backslash Y^{+}\rightarrow\Gamma^{G}\backslash X^{+}$ given by $[y]\mapsto[\phi\circ y]$ which is well-defined, holomorphic, and algebraic with respect to the algebraic structures of $S_{1}$ and $S_{2}$ (see \cite[Fact 2.6]{pink} for references). The morphisms $S_{1}\rightarrow S_{2}$ which are obtained in this way are called \emph{morphisms of Shimura varieties}. \index{Shimura variety!morphism} 

\begin{defi}
A \emph{special subvariety} of a Shimura variety $S$ is an irreducible component of the image of a morphism of Shimura varieties $T\rightarrow S$. A \emph{special set} \index{special set} in $S$ is a finite union of special subvarieties. 
\end{defi}

This definition does not conflict with the definition of special point we gave in the previous subsection. The Mumford-Tate group of a point $x\in X^{+}$ is a torus if and only if $p(x)\in S(\mathbb{C})$ is a zero-dimensional special subvariety (see \cite[\textsection 18]{daw}, \cite[Remark 12.6]{milne}, or \cite[\textsection 4]{pink}). 

\begin{ex}[Special subvarieties of $\mathbb{C}^{2}$]
Let us start by looking at $j:\mathbb{H}\rightarrow\mathbb{C}$ as a Shimura variety. Given that is it one-dimensional, $\mathbb{C}$ does not have special subvarieties other than the special points. We can instead consider $j:\mathbb{H}^{2}\rightarrow\mathbb{C}^{2}$, where $j$ is now defined on each component, which gives us a Shimura surface, and so we can look at its one-dimensional special subvarieties. These consist of the subvarieties cut out by the family $\left\{\Phi_{N}(X,Y)\right\}_{N\in\mathbb{N}}$ of modular polynomials of the $j$ function, plus the subvarieties of the form $\left\{\tau\right\}\times\mathbb{C}$ or $\mathbb{C}\times\left\{\tau\right\}$, where $\tau$ is a special point. This result will be generalised in Lemma \ref{lem:2.6}.
\end{ex}

\begin{remark}
\label{rem:sub}
Let $S$ be a Shimura variety and $V\subseteq S$ be a special subvariety. By definition this means that it can be realised as the image of a morphism of Shimura varieties $[\phi]:T\rightarrow S$. If $(G,X^{+})$ is the connected Shimura datum of $S$, let $H$ be the image of $\phi$ in $G$. Therefore $H$ is a $\mathbb{Q}$-subgroup of $G$. Also, let $X^{+}_{V}$ be the image of $\phi$ in $X^{+}$. Thus $\left(H, X^{+}_{V}\right)$ is a connected Shimura datum and, if $\Gamma$ is the congruence subgroup defining $S$, then $\Gamma^{H}:= H(\mathbb{Q})_{+}\cap\Gamma$ is a congruence subgroup of $H$ such that the inclusion $i:H\hookrightarrow G$ defines a morphism between the connected Shimura varieties $\Gamma^{H}\backslash X_{V}^{+}$ and $S$ with image $V$. In other words, we can always find a subgroup of $G$ that will define a Shimura datum for $V$. 

We hope that the following commutative diagram becomes a useful visual aid to help the reader with the notation we have introduced so far:
\begin{center}
\begin{tikzcd}[column sep=small]
& X^{+}_{V}\subseteq X^{+} \arrow[dl,"p^{H}"'] \arrow[dr,"p"] & \\
\Gamma^{H}\backslash X^{+}_{V} \arrow{rr}{[i]} & & V\subseteq S(\mathbb{C})
\end{tikzcd}
\end{center}
\end{remark}

\begin{defi}
Let $V\subseteq S$ be a special subvariety. A point $x\in X^{+}$ is said to be \emph{Hodge-generic in $V$} \index{Hodge-generic!in $V$} if $V$ is the smallest special subvariety of $S$ containing $p(x)$. Given a subset $A\subseteq S(F)$ (where $F$ is some algebraically closed field of characteristic zero), the \emph{special closure} \index{special closure} of $A$, denoted $\mathrm{spcl}(A)$,\index{spcl} is the smallest special set of $S(F)$ that contains $A$. So a point is always Hodge-generic in its special closure.
\end{defi}

\begin{remark}
\label{rem:mt}
The special closure of a point $x\in X^{+}$ is given by its Mumford-Tate group. This is because  $\left(\mathrm{MT}(x),\mathrm{MT}(x)(\mathbb{R})\cdot x\right)$ is a Shimura subdatum of $(G,X)$, and so the image of $\left(\mathrm{MT}(x)(\mathbb{R})^{+}\cdot x\right)\times\left\{g\right\}$ in $\mathrm{Sh}_{K}(G,X)(\mathbb{C})$ defines the smallest special subvariety containing $[x,g]_{K}$.
\end{remark}

In what remains of this subsection, fix a Shimura variety $p:X^{+}\rightarrow S(\mathbb{C})$, let $\Sigma$\index{$\Sigma$} be the set of coordinates of special points of $S(\mathbb{C})$ and let $E^{S}$ be the minimal abelian extension of the reflex field associated to the Shimura datum of $S$ over which $S(\mathbb{C})$ is defined.

\begin{remark}[see {{\cite[Proposition 4.14]{pink}}}]
Every special subvariety contains a dense set of special points, so every special subvariety is defined as a variety over $E^{S}(\Sigma)$.
\end{remark}

Now we define the notion of \emph{special domain}. \index{special domain} Essentially, a special domain for a special subvariety $V\subseteq S$ is a subset of $X^{+}$ that can act as the connected component of some Shimura datum which defines $V$.

\begin{defi}
Let $V\subseteq S^{n}$ be a special subvariety and choose a $\mathbb{Q}$-subgroup $H$ of $G^{n}$ and a point $x\in p^{-1}(V)$. So $(H, H(\mathbb{R})\cdot x)$ is a Shimura subdatum of $(G^{n},X^{n})$ for $V$. This can be done more canonically (but not completely) if we choose $x\in p^{-1}(V)$ to be Hodge-generic in $V$ and then $H$ can be taken to be the Mumford-Tate group of $x$. In this case we call $H(\mathbb{R})\cdot x$ a \emph{special domain} for $V$.
\end{defi}

Note that if we choose a point $y\in p^{-1}(V)$ different from $x$ that is Hodge-generic in $V$, then $H(\mathbb{R})\cdot y$ is another special domain for $V$ which is either equal to $H(\mathbb{R})\cdot x$ or disjoint from it. In fact, for every special domain $X^{+}_{V}\subseteq \left(X^{+}\right)^{m}$ of $V$, there is $\gamma\in\Gamma^{m}$ such that $\gamma X^{+}_{V} = \left(H(\mathbb{R})\cdot x\right)^{+}$. Therefore, we can find at most countably many different special domains for $V$, so we will number them and denote them by $X^{+}_{V,i}$,\index{x@$X^{+}_{V,i}$} where $i$ is a natural number between $1$ and the number of special domains for $V$ (usually $\omega$). This way if $i\neq j$, then $X^{+}_{V,i}\cap X^{+}_{V,j}=\emptyset$. 

\begin{remark}
\label{rem:bihol}
Let $X_{V,i}^{+}$ be a special domain, and suppose $\gamma\in \Gamma$ is such that $\gamma X_{V,i}^{+} = X_{V,i}^{+}$. If $H$ is the subgroup of $G$ corresponding to $X_{V,i}^{+}$, then $\gamma\in H(\mathbb{Q})$, because a consequence of the axioms of Shimura datum is that $H(\mathbb{R})^{+}$ maps surjectively to the group of biholomorphisms of $X^{+}_{V,i}$ (see \cite[Proposition 4.8]{milne}).
\end{remark}

Usually, the order in which we index the special domains will not be important, but we will make one distinction. When we define the special subvarieties $Z^{V,i}_{\overline{g}}$ in \textsection \ref{subsec:galrep}, we will want to index the special domains of $Z_{\overline{g}}^{V,i}$ correspondingly with the special domains of $V$, as it will be explained in Remark \ref{rem:indexzg}. One could define a more intrinsic form of indexing the special domains, but as we will not really use this extra information, we have opted for a simpler form of notation.

\begin{defi}
A \emph{special set} \index{special set} in $\left(X^{+}\right)^{m}$ is a finite union of special domains. The \emph{special closure} \index{special closure} of $B\subseteq \left(X^{+}\right)^{m}$ is the smallest special set containing $B$, which again will be denoted by $\mathrm{spcl}(B)$.\index{spcl}
\end{defi}

\begin{ex}[Trivial special closures]
If $x\in X^{+}$ is Hodge-generic, then $\mathrm{spcl}(x)=X^{+}$ and $\mathrm{spcl}(p(x)) = S(\mathbb{C})$. If $x\in X^{+}$ is special, then $\mathrm{MT}(x)$ is a torus, and as it is commutative, we have that $\mathrm{MT}(x)(\mathbb{R})^{+}\cdot x = \left\{x\right\}$. So, for some $i$, $X_{x,i}^{+} = \left\{x\right\}=\mathrm{spcl}(x)$ and $\mathrm{spcl}(p(x))=\left\{p(x)\right\}$. 
\end{ex}

\begin{remark}[A Word on Reflex Fields]
We conclude this subsection with an observation about reflex fields. Let $V\subseteq S$ be a special subvariety and let $X^{+}_{V,i}$ be a special domain for $V$. Let $\Gamma^{H}:= \Gamma\cap H(\mathbb{Q})_{+}$. Then we can view $V$ as a the image of a Shimura variety as: $S_{H} = \Gamma^{H}\backslash X^{+}_{V,i}$ under a Shimura morphism. Let $E^{S_{H}}$ be the (finite abelian extension of the) reflex field of $S_{H}$. By \cite[Remark 12.3]{milne}, $E^{S}\subseteq E^{S_{H}}$. By the general theory of canonical models (\cite[\textsection 12]{milne}), we have that $E^{S_{H}}\subseteq E^{S}(\Sigma)$. 
\end{remark}

\subsection{Galois Representations I}
\label{subsec:galgen}
Following \cite[\textsection 2]{ullmo-yafaev} and \cite[\textsection 6]{pink}, we describe Galois representations attached to points in Shimura varieties, and we will end by stating Pink's conjecture, which will later influence the way we state conditions FIC 1 \& 2. 

Given a congruence subgroup $\Gamma'$ of $\Gamma$ we have a morphism of connected Shimura varieties $[\mathrm{id}]:\Gamma'\backslash X^{+}\rightarrow\Gamma\backslash X^{+}$ induced by the identity map $\mathrm{id}:G\hookrightarrow G$. So as to distinguish elements in the quotients of $X^{+}$ with respect to the different congruence groups, let us denote by $[x]_{\Gamma}$ the class of $x\in X^{+}$ in the quotient $\Gamma\backslash X^{+}$. Remember we are assuming $\Gamma$ is neat.

\begin{lem}
\label{lem:nofix}
For every $\gamma\in\Gamma$ and every $x\in X^{+}$, $\gamma x\neq x$. 
\end{lem}
\begin{proof}
Suppose that $\gamma x=x$ for some $x\in X^{+}$ and some $\gamma\in\Gamma$. By definition $\Gamma = G(\mathbb{Q})_{+}\cap K$, but if $\gamma x=x$, then this also means that $\gamma$ belongs to $Z_{G_{\mathbb{R}}}(x(\mathbb{S}))(\mathbb{R})$. As $\Gamma$ is neat, we can use \cite[Lemma 2.2]{ullmo-yafaev} and \cite[Lemma 2.3]{ullmo-yafaev} to conclude straightaway that $\gamma=1$.
\end{proof}

Choose a point $z\in S(\mathbb{C}) = \Gamma\backslash X^{+}$ and choose $x\in X^{+}$ so that $[x]_{\Gamma} = z$. Lemma \ref{lem:nofix} shows that $[\mathrm{id}]^{-1}(z)$ carries a transitive $\Gamma$-action. If $\Gamma'$ is normal in $\Gamma$, then this action factors through $\Gamma/\Gamma'$, and can be described by: for $y\in X^{+}$ such that $[y]_{\Gamma} = [x]_{\Gamma}$ and for every $\gamma\in\Gamma$, define $\gamma[y]_{\Gamma'} := [\gamma y]_{\Gamma'}$. The action of $\Gamma/\Gamma'$ on $[\mathrm{id}]^{-1}(z)$ is transitive. Thus we obtain:

\begin{cor}
The action of $\Gamma/\Gamma'$ on $[\mathrm{id}]^{-1}(z)$ is simply transitive. 
\end{cor}

The varieties of the form $\Gamma'\backslash X^{+}$, where $\Gamma'$ is a normal congruence subgroup of $\Gamma$, form an inverse system with morphisms the maps $[\mathrm{id}]$ described above. Let $\mathbb{X}^{+}$ be the inverse limit of the system. Any $\tilde{x}\in\mathbb{X}^{+}$ can be denoted as $([x_{i}]_{\Gamma_{i}})_{\Gamma_{i}}$, so that if $\Gamma_{i}\subseteq \Gamma_{j}$, then $[x_{i}]_{\Gamma_{j}} = [x_{j}]_{\Gamma_{j}}$. Also, there are natural maps
\begin{align*}
   \pi_{\Gamma_{i}}:\, \mathbb{X}^{+}\quad&\xrightarrow{\quad\quad}\Gamma'\backslash X^{+}\\
     \left(\left[x_{i}\right]_{\Gamma_{i}}\right)_{\Gamma_{i}}&\xmapsto{\phantom{\mathbb{X}^{+}}\qquad} \left[x_{i}\right]_{\Gamma_{i}},
\end{align*}
compatible with the maps of the system. As we mentioned before, the Shimura variety $\mathrm{Sh}_{K}(G,X)$ is canonically defined over its reflex field $E$, which is independent of the choice of $K$. But when we look at a connected component of this variety like $S(\mathbb{C}) = \Gamma\backslash X^{+}$, this variety is defined over a finite abelian extension of $E$ that we have denominated $E^{S}$. The map $[\mathrm{id}]: S'(\mathbb{C}) = \Gamma'\backslash X^{+}\rightarrow S(\mathbb{C})$ is defined over the compositum of the fields $E^{S'}$ and $E^{S}$. And so, the whole inverse system is defined over $E^{\mathrm{ab}}$, the maximal abelian extension of $E$.

Let $L\subseteq\mathbb{C}$ be a finitely generated field extension of $E^{\mathrm{ab}}$ such that $z$ is defined over $L$. Let $\overline{L}$ denote the algebraic closure of $L$ in $\mathbb{C}$. Then the Galois group $\mathrm{Gal}\left(\overline{L}/L\right)$ also acts on the fibre over $z$ of the morphism $[\mathrm{id}]:\Gamma'\backslash X^{+}\rightarrow\Gamma\backslash X^{+}$. We would like to compare the actions of $\mathrm{Gal}\left(\overline{L}/L\right)$ and $\Gamma/\Gamma'$ on $[\mathrm{id}]^{-1}(z)$. For this we need some elementary results. 

\begin{lem}
\label{lem:elem}
Let $H$ and $G$ be groups and $X$ be a set with a left $G$-action and a compatible right $H$-action. If the action of $H$ is transitive and free, then, fixing a point $x_{0}\in X$, there is a unique group homomorphism $\rho_{x_{0}}:G\rightarrow H$ such that for every $g\in G$ $gx_{0} = x_{0}\rho_{x_{0}}(g)$. Moreover, if $x_{1} = x_{0}h$, then the associated homomorphism $\rho_{x_{1}}$ satisfies $\rho_{x_{1}} = h^{-1}\rho_{x_{0}}h$.
\end{lem}
\begin{proof}
As $H$ acts transitively on $X$, for every $g\in G$ there is $h\in H$ such that $gx_{0} = x_{0}h$. If $h_{1}, h_{2}\in H$ are such that $x_{0}h_{1} = x_{0}h_{2}$, then $x_{0}$ is a fixed point of $h_{1}h_{2}^{-1}$, but as the action of $H$ is free, then $h_{1} = h_{2}$. Therefore we can define a map $\rho_{x_{0}}:G\rightarrow H$ by setting $\rho_{x_{0}}(g) = h$, where $gx_{0} = x_{0}h$. 

If $g_{1}x_{0} = x_{0}h_{1}$ and $g_{2}x_{0} = x_{0}h_{2}$, then $g_{1}g_{2}x_{0} = g_{1}x_{0}h_{2} = x_{0}h_{1}h_{2}$, and so $\rho_{x_{0}}$ is a group homomorphism.

Finally, if $x_{1}=x_{0}h$, then $x_{1}h^{-1}\rho_{x_{0}}(g)h = x_{0}\rho_{x_{0}}(g)h = gx_{0}h = gx_{1}$, and so $\rho_{x_{1}} = h^{-1}\rho_{x_{0}}h$.
\end{proof}

Although $\Gamma$ acts on the left on $X^{+}$, we can define a corresponding right action of $\Gamma$ on $X^{+}$ in the usual way by $x\gamma:=\gamma^{-1}x$, for all $x\in X^{+}$ and $\gamma\in\Gamma$. 

\begin{lem}
\label{lem:actcom}
Assuming that $\Gamma'$ is a normal congruence subgroup of $\Gamma$, the actions of $\mathrm{Gal}\left(\overline{L}/L\right)$ and $\Gamma/\Gamma'$ on $[\mathrm{id}]^{-1}(z)$ commute. 
\end{lem}
\begin{proof}
Every $\gamma\in\Gamma$ defines a homomorphism $G\rightarrow G$ given by $g\mapsto \gamma g\gamma^{-1}$, which itself induces a morphism of Shimura data $(G,X)\rightarrow(G,X)$ which on $X$ is given by $x\mapsto \gamma x$. As $\Gamma'$ is normal in $\Gamma$, this in turn induces a morphism of Shimura varieties $[\gamma]: S'(\mathbb{C}) = \Gamma'\backslash X^{+}\rightarrow S(\mathbb{C}) = \Gamma'\backslash X^{+}$. This morphism is defined over the compositum $E^{S}E^{S'}$. As $L$ contains $E^{S}E^{S'}$, then $[\gamma]$ is fixed under the action of any element of $\mathrm{Gal}\left(\overline{L}/L\right)$, which proves the statement.
\end{proof}

With these results we can now compare the actions of $\mathrm{Gal}\left(\overline{L}/L\right)$ and $\Gamma/\Gamma'$ on $[\mathrm{id}]^{-1}(z)$. Choosing $[y]_{\Gamma'}\in[\mathrm{id}]^{-1}(z)$, by Lemma \ref{lem:elem} (whose conditions are ensured by Lemma \ref{lem:actcom}) we obtain a homomorphism 
\begin{equation*}
    \rho_{[y]_{\Gamma'}}:\mathrm{Gal}\left(\overline{L}/L\right)\rightarrow \Gamma/\Gamma'.
\end{equation*}
We remark that for this we only need $L$ to contain $E^{S}$, $E^{S'}$ and the coordinates of $z$; it is not necessary that $L$ contain all of $E^{\mathrm{ab}}$.

Let $\hat{\Gamma} = \varprojlim \Gamma/\Gamma'$, where the limit is taken over the normal congruence subgroups $\Gamma'$ of $\Gamma$. Then both $\mathrm{Gal}\left(\overline{L}/L\right)$ and $\hat{\Gamma}$ act on the fibre $\pi_{\Gamma}^{-1}(z)$. In fact the action of $\hat{\Gamma}$ can be understood more generally in the following way. For the Shimura datum $(G,X)$ define the space
\begin{equation*}
    \mathrm{Sh}(G,X):=G(\mathbb{Q})\backslash\left(X\times G(\mathbb{A}_{f})\right),
\end{equation*}
whose elements we denote as $[x,g]$, and for $K$ a compact open subgroup of $G(\mathbb{A}_{f})$, let
\begin{align*}
    \pi_{K}: \mathrm{Sh}(G,X)&\rightarrow \mathrm{Sh}_{K}(G,X)\\
    [x,g]&\mapsto [x,gK].
\end{align*}
When $K$ is neat, the fibre over any point of $\mathrm{Sh}_{K}(G,X)$ under $\pi_{K}$ has a simply transitive right action by $K$ (see \cite[Lemma 2.1]{ullmo-yafaev}) given by $[x,g]k:=[x,gk]$. On the other hand, there is an injective map $X^{+}\rightarrow\mathrm{Sh}(G,X)$ given by $x\mapsto[x,1]$. Observe that for all $\gamma\in\Gamma$ we have that $[\gamma x,1] = [x,\gamma^{-1}]$. The action of $\hat{\Gamma}$ on $\pi_{\Gamma}^{-1}(z)$ is then just the restriction of the action of $K$, and so $\hat{\Gamma}$ acts simply transitively on $\pi_{\Gamma}^{-1}(z)$. Choosing a point $\tilde{y}\in\pi_{\Gamma}^{-1}(z)$, we get a continuous homomorphism: 
\begin{equation*}
    \rho_{\tilde{y}} : \mathrm{Gal}\left(\overline{L}/L\right)\rightarrow\hat{\Gamma}.
\end{equation*}

\begin{remark}
In the case that $z$ is a special point, the homomorphism $\rho_{\tilde{y}}$ can be described using the reciprocity map coming from the theory of complex multiplication (see \cite[Remark 2.7]{ullmo-yafaev}).
\end{remark}

\begin{defi}
With the above notation, $z\in S(\mathbb{C})$ is called \emph{Galois generic} \index{Galois generic point} if the image of $\rho_{\tilde{y}}$ is open in $\hat{\Gamma}$ (observe that openness of the image does not depend on the choice of $\tilde{y}$). The point $z$ is called \emph{strictly Galois generic} if the image of $\rho_{\tilde{y}}$ equals $\hat{\Gamma}$.
\end{defi}

The following proposition says that the notion of ``Galois generic'' is well-defined. 

\begin{prop}[see {{\cite[Proposition 6.4]{pink}}}]
The notion of being `Galois generic' is independent of the choice of $L$. 
\end{prop}

\begin{prop}[see {{\cite[Proposition 6.7]{pink}}}]
Every Galois generic point on $S(\mathbb{C})$ is Hodge generic.
\end{prop}

We are now in a position to state Pink's conjecture, which is the converse of the last proposition. \index{Pink's conjecture}

\begin{conj}[see {{\cite[Conjecture 6.8]{pink}}}]
\label{conj:pink}
Every Hodge generic point on $S(\mathbb{C})$ is Galois generic. 
\end{conj}

In the case of Siegel moduli spaces, \cite[Remark 6.10]{pink} and \cite[p. 265]{ullmo-yafaev} explain the relation between this conjecture and the Mumford-Tate conjecture, concluding that ad\`elic refinements of Mumford-Tate imply Conjecture \ref{conj:pink}. A result of Serre provides such a refinement for some cases: 

\begin{thm}[see {{\cite[Theorem $6.13$]{pink}}}]
\label{thm:pink}
Let $(\mathrm{GSp_{2g}, \mathbb{H}_{2g}})$ be the connected Shimura data of Example \ref{ex:ppav}. Assume that $g$ is odd or equal to $2$ or $6$. Then for any suitable congruence subgroup $\Gamma$, the Hodge-generic points of the corresponding Shimura variety $S(\mathbb{C})$ are Galois generic.
\end{thm}

\subsection{Galois Representations II}
\label{subsec:galrep}
Here we will describe the aspects of Galois representations described in \textsection\ref{subsec:galgen} that can be ``seen'' by the model theory. Essentially, the difference with the previous treatment is that we need to restrict the normal subgroups $\Gamma'$ of $\Gamma$ to be of a specific form. The results of \textsection\ref{subsec:normal} show that this difference is merely technical, and does not affect the results nor the general theory of Galois representations attached to points. 

Let $S$ be a Shimura variety with Shimura datum $\left(G,X^{+}\right)$ and $S(\mathbb{C}) = \Gamma\backslash X^{+}$.

\begin{prop}[see {{\cite[\textsection 4.1 Corollary 1]{platonov}}}]
\label{prop:platonov}
Let $\Gamma$ be an arithmetic subgroup of an algebraic group $G$ defined over $\mathbb{Q}$. Then for any $g\in G(\mathbb{Q})$, $g\Gamma g^{-1}$ is also an arithmetic subgroup of $G(\mathbb{Z})$. 
\end{prop}

In particular $\Gamma\cap g\Gamma g^{-1}$ has finite index in $\Gamma$. Not only that, but choosing $N$ large enough so that the denominators that appear in $g$ and $g^{-1}$ are coprime with $N$, we get that $\Gamma(N)\subseteq\Gamma\cap g\Gamma g^{-1}$, and therefore $\Gamma\cap g\Gamma g^{-1}$ is a congruence subgroup of $G(\mathbb{Z})$. In the following Lemma we define the notation for the special subvariety $Z^{V,i}_{\overline{g}}$.\index{z@$Z^{V,i}_{\overline{g}}$} These subvarieties play a major role in the model theory of Shimura varieties.

\begin{lem}[cf. {{\cite[Lemma 2.6]{daw-harris}}}]
\label{lem:2.6}
Let $\overline{g}=(g_{1},\ldots,g_{n})$ be a tuple of elements belonging to $G^{\mathrm{ad}}(\mathbb{Q})^{+}$ and let $X_{V,i}^{+}$ be a special domain in $\left(X^{+}\right)^{m}$. The image of the map 
\begin{align*}
    f:\quad X_{V,i}^{+}\qquad &\xrightarrow{\qquad} \quad S(\mathbb{C})^{m+n}\\ (x_{1},\ldots,x_{m})&\mapsto\left(p\left(\overline{g}x_{1}\right),\ldots,p\left(\overline{g}x_{m}\right)\right)
\end{align*}
 is a special subvariety which we will denote $Z^{V,i}_{\overline{g}}$.
\end{lem}
\begin{proof}
The Shimura datum of $S(\mathbb{C})^{m+n}$ is $\left(G^{m+n},\left(X^{+}\right)^{m+n}\right)$. Let $H$ be the $\mathbb{Q}$-subgroup of $G^{m}$ acting on $X^{+}_{V,i}$ given by the definition of $X^{+}_{V,i}$. Consider the morphism of algebraic groups 
\begin{align*}
    \phi:\quad H\qquad &\xrightarrow{\qquad\qquad\qquad\qquad}\quad G^{m+n}\\
    \left(h_{1},\ldots,h_{m}\right)&\mapsto\left(g_{1}h_{1}g_{1}^{-1},\ldots,g_{n}h_{1}g_{n}^{-1},g_{1}h_{2},g_{1}^{-1},\ldots,g_{n}h_{m}g_{n}^{-1}\right).
\end{align*}
This induces a morphism of connected Shimura data $\phi:(H,X_{V,i}^{+})\rightarrow\left(G^{m+n},\left(X^{+}\right)^{m+n}\right)$. Let $K$ be a compact open subgroup of $G(\mathbb{A}_{f})$ such that if $\Gamma = G(\mathbb{Q})_{+}\cap K$, then $S(\mathbb{C})=\Gamma\backslash X^{+}$. Then $S(\mathbb{C})^{m+n} = \Gamma^{m+n}\backslash \left(X^{+}\right)^{m+n}$. 

Let $\Gamma' = \left(g_{1}^{-1}\Gamma g_{1}\cap\cdots\cap g_{n}^{-1}\Gamma g_{n}\right)^{m}\cap H(\mathbb{A}_{f})$. Notice that $\phi(\Gamma')\subseteq\Gamma^{m+n}$. Therefore, $\phi$ induces a morphism of connected Shimura varieties $\Gamma'\backslash X^{+}_{V,i} \rightarrow \Gamma^{m+n}\backslash\left(X^{+}\right)^{m+n}$ whose image is precisely:
\begin{equation*}
    Z_{\overline{g}}^{V,i}:=\left\{\left(p\left(\overline{g}x_{1}\right),\ldots,p\left(\overline{g}x_{m}\right)\right) : \left(x_{1},\ldots,x_{m}\right)\in X_{V,i}^{+}\right\}.
\end{equation*}
By definition, $Z_{\overline{g}}^{V,i}$ is a special subvariety.
\end{proof}

\begin{remark}[Indexing of the special domains of $Z_{\overline{g}}^{V,i}$]\index{z@$Z^{V,i}_{\overline{g}}$!indexing}
\label{rem:indexzg}
Given the way we have constructed the special domains of $Z_{\overline{g}}^{V,i}$ in Lemma \ref{lem:2.6}, we will fix the indexing and say that $X^{+}_{Z_{\overline{g}}^{V,i},i}$ is the special domain of $Z^{V,i}_{\overline{g}}$ which is parametrised by $X^{+}_{V,i}$. This means that for any $\overline{x}\in X^{+}_{Z_{\overline{g}}^{V,i},i}$ there is $y\in X^{+}_{V,i}$ such that $\overline{x} = \overline{g}y$, so we can write $X^{+}_{Z_{\overline{g}}^{V,i},i} = \overline{g}X^{+}_{V,i}$. Later, when we define Shimura structures in \textsection\ref{sec:ss}, it will be evident that this parametrisation is definable in the language of Shimura structures.
\end{remark}

\begin{remark}
When $X_{V,i}^{+}=X^{+}$ (and so $V=S$),  we will simplify the notation to be $Z_{\overline{g}} := Z_{\overline{g}}^{S}$.\index{z@$Z_{\overline{g}}$}.
\end{remark}

The special subvarieties $Z_{\overline{g}}^{V,i}$ encode a lot of the information on Galois representations attached to points in $S$, and they will become essential in our model theoretic approach. It is rather crucial to understand how they relate to each other, so in what remains of this section we will recall the treatment done in \cite[\textsection 2.5]{daw-harris} and extend it to the subvarieties $Z_{\overline{g}}^{V,i}$.

Let $\overline{g}=\left(e,g_{1},\ldots,g_{n}\right)$ be a tuple of distinct elements of $G^{\mathrm{ad}}(\mathbb{Q})^{+}$. The corresponding variety $Z_{\overline{g}}$ is biholomorphic to $\Gamma_{\overline{g}}\backslash X^{+}$, where:\index{$\Gamma_{\overline{g}}$}
\begin{equation*}
    \Gamma_{\overline{g}}:=\Gamma\cap g_{1}^{-1}\Gamma g_{1}\cap\cdots\cap g_{n}^{-1}\Gamma g_{n}.
\end{equation*}
Suppose $\left\{h_{1},\ldots,h_{m}\right\}$ is a subset of $G$ such that $\left\{g_{1},\ldots,g_{n}\right\}\subseteq\left\{h_{1},\ldots,h_{m}\right\}$. Let $\overline{h}:=(e,h_{1},\ldots,h_{m})$. Then there is a finite morphism \index{$\psi_{\overline{h},\overline{g}}$}
\begin{equation*}
    \psi_{\overline{h},\overline{g}}:Z_{\overline{h}}\rightarrow Z_{\overline{g}}
\end{equation*}
induced by the natural map $\Gamma_{\overline{h}}\backslash X^{+}\rightarrow \Gamma_{\overline{g}}\backslash X^{+}$. If $\Gamma_{\overline{h}}$ is normal in $\Gamma_{\overline{g}}$, then the fibres carry a simply transitive action of $\Gamma_{\overline{g}}/\Gamma_{\overline{h}}$.  \textbf{For the remainder of this section, we will only consider the case when $\Gamma_{\overline{h}}$ is normal in $\Gamma_{\overline{g}}$.}

The action of $\Gamma_{\overline{g}}/\Gamma_{\overline{h}}$ on the fibers of $\psi_{\overline{h},\overline{g}}$ in $Z_{\overline{h}}$ is given by regular maps defined over $E^{S}(\Sigma)$. Choose $z\in Z_{\overline{g}}$ and let $L$ be a finitely generated field extension of $E^{S}(\Sigma)$ such that $z$ is defined on $L$. Then $\mathrm{Aut}(\mathbb{C}/L)$ acts on the fibre of $\psi_{\overline{h},\overline{g}}$ over $z$ and, because the action of $\Gamma_{\overline{g}}/\Gamma_{\overline{h}}$ is given by regular maps defined over $E^{S}(\Sigma)$, then this action commutes with the action of $\Gamma_{\overline{g}}/\Gamma_{\overline{h}}$. The action of $\Gamma_{\overline{g}}/\Gamma_{\overline{h}}$ on the fibres of $\psi_{\overline{h},\overline{g}}$ is transitive and (as $\Gamma$ is torsion free in $G^{\mathrm{ad}}(\mathbb{Q})^{+}$) free. Therefore, just as in \textsection \ref{subsec:galgen}, for any point in the fibre above $z\in Z_{\overline{g}}$ we obtain a continuous homomorphism 
\begin{equation*}
    \mathrm{Aut}(\mathbb{C}/L)\rightarrow\Gamma_{\overline{g}}/\Gamma_{\overline{h}}.
\end{equation*}

If $\overline{z}=(z_{1},\ldots,z_{m})\in Z_{\overline{g}}^{V,i}\subseteq \left(Z_{\overline{g}}\right)^{m}$ for some special subvariety $V\subseteq\left(S(\mathbb{C})\right)^{m}$, then the map $\psi_{\overline{h},\overline{g}}$ induces a map 
\begin{equation*}\index{$\psi_{\overline{h},\overline{g}}^{V,i}$}
    \psi_{\overline{h},\overline{g}}^{V,i}:Z_{\overline{h}}^{V,i}\rightarrow Z_{\overline{g}}^{V,i}.
\end{equation*}
The stabilizer of the fiber of $\psi_{\overline{h},\overline{g}}^{V,i}$ over $\overline{z}$ is a subgroup of $\left(\Gamma_{\overline{g}}/\Gamma_{\overline{h}}\right)^{m}$ which we will denote $\Gamma_{\overline{g},\overline{h}}^{V,i}$. Notice that in this case, if $L$ also contains the coordinates of special points in $S(\mathbb{C})$, the homomorphism $\mathrm{Aut}(\mathbb{C}/L)\rightarrow\left(\Gamma_{\overline{g}}/\Gamma_{\overline{h}}\right)^{m}$ factors through $\mathrm{Aut}(\mathbb{C}/L)\rightarrow \Gamma_{\overline{g},\overline{h}}^{V,i}$. 

For every special subvariety $V\subseteq S^{m}$, we have obtained an inverse system of varieties $Z_{\overline{g}}^{V,i}$, where the tuples $\overline{g}$ are taken so that $\Gamma_{\overline{g}}\lhd\Gamma$. This induces an inverse system of groups $\Gamma_{\overline{g},\overline{h}}^{V,i}$. We can therefore define the inverse limits
\begin{equation*}
    \overline{\Gamma}:=\varprojlim_{\overline{g}}\Gamma/\Gamma_{\overline{g}}\;\mbox{ and }\; \overline{\Gamma^{V,i}}:=\varprojlim_{\overline{g}}\Gamma_{e,\overline{g}}^{V,i},
\end{equation*} \index{$\overline{\Gamma}$} \index{$\overline{\Gamma^{V,i}}$}
where the limits are taken over all tuples of distinct elements $\overline{g} = (e,g_{1},\ldots,g_{n})$ such that $\Gamma_{\overline{g}}$ is normal in $\Gamma$. If $z$ has coordinates in an extension $L$ of $E^{S}(\Sigma)$, then, by our discussion in \textsection\ref{subsec:galgen}, we get the conjugacy class of a continuous homomorphism $\mathrm{Aut}(\mathbb{C}/L)\rightarrow\overline{\Gamma^{V,i}}$. 
Analogously, if we take a tuple $\overline{z}=(z_{1},\ldots,z_{m})\in S(\mathbb{C})^{n}$ with coordinates in an extension $L$ of $E^{S}(\Sigma)$, then we get the conjugacy class of a continuous homomorphism $\mathrm{Aut}(\mathbb{C}/L)\rightarrow\overline{\Gamma}^{m}$.

\subsection{Normal Subgroups in \texorpdfstring{$\Gamma$}{G}}
\label{subsec:normal}
We now address the question of whether there are normal subgroups of $\Gamma$ of the form $\Gamma_{\overline{g}}$. 

\begin{thm}[see {{\cite[\textsection 4.1 Theorem 4.2]{platonov}}}]
Let $\Gamma$ be an arithmetic subgroup of an algebraic group $G$ defined over $\mathbb{Q}$. Then $\Gamma$ is finitely presented as an abstract group.
\end{thm}

In particular, $\Gamma$ is finitely generated. Therefore, for any $n\in \mathbb{N}$, there are only finitely many subgroups of $\Gamma$ of index $n$ (see e.g. \cite[Proposition 5.11]{drutu}). Let $g\in G^{\mathrm{ad}}(\mathbb{Q})^{+}$. Then $\Gamma':= \Gamma\cap g\Gamma g^{-1}$ has finite index in $\Gamma$ by Proposition \ref{prop:platonov}. For every $\gamma\in\Gamma$, $\gamma\Gamma'\gamma^{-1}$ has the same index as $\Gamma'$ in $\Gamma$. Given that there are only finitely many subgroups of $\Gamma$ of this index, the intersection
\begin{equation*}
    \Gamma'': = \bigcap_{\gamma\in\Gamma}\gamma\Gamma'\gamma^{-1}
\end{equation*}
is a finite intersection. On the other hand, it is clear that $\Gamma''$ is normal in $\Gamma$. Also $\gamma\Gamma'\gamma^{-1} = \Gamma\cap\gamma g\Gamma g^{-1}\gamma^{-1}$. And so $\Gamma'' = \Gamma_{\overline{g}}$ for some tuple $\overline{g}$ from $G^{\mathrm{ad}}(\mathbb{Q})^{+}$. 

Observe that this argument can be replicated for any subgroup of $\Gamma$ of the form $\Gamma_{\overline{g}}$ to obtain a normal subgroup of $\Gamma$ of the form $\Gamma_{\overline{h}}$ such that $\overline{h}$ contains the tuple $\overline{g}$.

\section{Shimura Structures}
\label{sec:ss}
Now we will present our first model-theoretic treatment of Shimura varieties (the second approach will be given in \textsection \ref{sec:covers}). For every Shimura variety $S$ we will define a corresponding Shimura structure. Our treatment differs from \cite{daw-harris} in that we have enhanced the language to include names for the special domains of special subvarieties.  

\begin{defi}
A \emph{Shimura structure for $S$}\index{Shimura structure} is a two-sorted structure $\mathbf{q}:=\left<\mathbf{D},\mathbf{S},q\right>$ where:\index{q@$\mathbf{q}$} \index{d@$\mathbf{D}$} \index{s@$\mathbf{S}$}
\begin{enumerate}
    \item $\mathbf{D}:=\left<D;\left\{g\right\}_{g\in G^{\mathrm{ad}}(\mathbb{Q})^{+}},\mathscr{R}_{D}\right>$ is a $G^{\mathrm{ad}}(\mathbb{Q})^{+}$-\emph{covering structure}.\index{covering structure} This means that $D$ is a set with an action of $G^{\mathrm{ad}}(\mathbb{Q})^{+}$ (so the elements of $G^{\mathrm{ad}}(\mathbb{Q})^{+}$ denote function symbols), and $\mathscr{R}_{D}$ is a countable set of subsets of powers of $D$. The elements of $\mathscr{R}_{D}$ are named $D_{V,i}$, \index{d@$D_{V,i}$} where the indexing comes from the indexing of special domains of $X^{+}$.
    \item $\mathbf{S}:=\left<S(F);\mathscr{R}_{S}\right>$ is a \emph{variety structure}. \index{variety structure} This means that $S(F)$ is the set of $F$-points of $S$, where $F$ is an algebraically closed field. Variety structures will be interpreted in the field sort $\mathbf{F}:=\left<F;+,\cdot,F_{0}\right>$,\index{f@$\mathbf{F}$} where $F_{0}$\index{f@$F_{0}$} is a countable subfield of $F$, which is algebraic over $\mathbb{Q}$, contains $E^{S}(\Sigma)$, and which is interpreted as a set of constants in $\mathbf{F}$. $\mathscr{R}_{S}$ is the set of all Zariski closed subsets of $S(F)^{n}$ defined over $F_{0}$ for all $n\in\mathbb{N}$. 
    \item $q:D\rightarrow S(F)$ is a function. 
\end{enumerate}
Every Shimura variety $S$ determines (up to choice of a field of constants $F_{0}$) a natural Shimura structure which we will denote $\mathbf{p}$,\index{p@$\mathbf{p}$} where we take $F=\mathbb{C}$, $D=X^{+}$, $q = p$, and $D_{V,i} = X^{+}_{V,i}$. For $F_{0}$ we will normally use $E^{S}(\Sigma)$.
\end{defi}

We have implicitly defined the language we will use for Shimura structures (which depends on $S$). Let $\mathcal{L}_{D}$ \index{l@$\mathcal{L}_{D}$} be the language we have used in the $G^{\mathrm{ad}}(\mathbb{Q})^{+}$-covering sort, $\mathcal{L}_{F}$ \index{l@$\mathcal{L}_{F}$} the language of the field sort, and let $\mathcal{L}$ \index{l@$\mathcal{L}$} be the language of Shimura structures. $F_{0}$ is countable, so $\mathcal{L}$ is countable. Finally, let $\mathrm{Th}(\mathbf{p})$ \index{t@$\mathrm{Th}(\mathbf{p})$} be the complete first-order theory of the two-sorted structure $\mathbf{p}$ in the language $\mathcal{L}$.

\subsection{Theory of the Variety Sort}
Let $T:=\mathrm{Th}(S(\mathbb{C}))$,\index{t@$T$} which is the complete first-order theory of $S(\mathbb{C})$ in the language $\mathcal{L}_{F}$. We make a few observations, which are standard results from model theory.

\begin{enumerate}
\item The models of $T$ are bi-interpretable with $F$. So any model of $T$ is of the form $S(F)$, for $F$ an algebraically closed field. So $T$ has quantifier elimination, and hence it has finite Morley rank.
\item By \cite[Theorem 6.2.7]{tent-ziegler}, $T$ is totally transcendental.
\item By \cite[Theorem 5.3.3]{tent-ziegler}, $T$ has a prime model. In fact, it is $S\left(\overline{\mathbb{Q}}\right)$.
\item By \cite[Theorem 5.2.6]{tent-ziegler}, as $T$ is countable and totally transcendental, it is also $\omega$-stable.
\end{enumerate}

\begin{remark}[Projections of special sets]
\label{rem:proj}
Let $V\subseteq S^{n}$ be a special subvariety. The Shimura datum of $S^{n}$ is $(G^{n}, X^{n})$ where $G^{n}$ acts coordinate-wise on $X^{n}$. A coordinate projection $\mathrm{pr}:S(F)^{n}\rightarrow S(F)^{m}$ (where $m<n$) is a Shimura morphism as it is induced by the corresponding coordinate projections $G(\mathbb{R})^{n}\rightarrow G(\mathbb{R})^{m}$ and $X^{n}\rightarrow X^{m}$. Therefore, $\mathrm{pr}(V)$ is a special set. 

We can obtain an analogous result for the projections of special domains. Suppose that $\mathrm{pr}(V) = W_{1}\cup\cdots\cup W_{r}$, with each $W_{t}$ a special subvariety. If we take a point in the corresponding coordinate projection $\mathrm{pr}\left(X^{+}_{V,i}\right)$, then this point must belong a special domain of some of the $W_{t}$. So we can write $\mathrm{pr}\left(X^{+}_{V,i}\right) \subseteq X^{+}_{W_{1},j_{1}}\cup\cdots\cup X^{+}_{W_{n},j_{r}}$. This union will consist of exactly one special domain for each $W_{t}$ as $X^{+}_{V,i}$ is connected, and so its projection needs to be connected (remember that distinct special domains of the same special subvariety are disjoint). On the other hand, every point in $W_{t}$ can be lifted to a point in $V$, so we get $\mathrm{pr}\left(X^{+}_{V,i}\right) \subseteq X^{+}_{W_{1},j_{1}}\cup\cdots\cup X^{+}_{W_{n},j_{r}}$.
\end{remark}

\begin{remark}
\label{rem:special}
In our model theoretic setting of Shimura structures, our approach has been purely algebraic, effectively forgetting about the analytic structure on both sorts, only retaining the algebraic vareity structure on $S$, and the special structure and the $G^{\mathrm{ad}}(\mathbb{Q})^{+}$-action on $X^{+}$. But one can instead study Shimura varieties again in a two-sorted way from the perspective of o-minimality (so both $X^{+}$ and $S$ are now defined over real closed fields), in order to see part of the analytic structure. This has been done by B. Zilber in \cite{zilber}. 

In this paper, Zilber studies the structure of special subvarieties in a very general class of arithmetic varieties using o-minimality. Most results are framed in terms of ``weakly special subvarieties'' and ``weakly special domains'', but one can obtain results about special subvarieties using the fact that a weakly special subvariety is special if and only if it contains a special point. The paper proves that weakly special subvarietes define a Noetherian Zariski structure (see \cite[Theorem 4.11]{zilber}) and it also gives a characterisation of special subvarieties (see \cite[Example 6.4]{zilber}). The paper also proves that the special structure of a Shimura curve has trivial geometry in the model-theoretic sense. This result can be interpreted as saying that if we can only see the special subvarieties on the powers of a Shimura curve and we are given a subset $A$ of the curve, then the only points that are algebraic (in the model-theoretic sense) over $A$ are those that belong to the Hecke-orbit of some point in $A$.

It seems from the end of \cite[\textsection 3.2]{zilber} that the author is not sure if mixed Shimura varieties are in the class of arithmetic varieties he is considering. But \cite[\textsection 10.1]{gao} shows that mixed Shimura varieties do satisfy the axioms of \cite[\textsection 3.1]{zilber}.
\end{remark}

\subsection{Axiomatisation}
\label{subsec:axiom}
We will define the special subvarieties condition (SS), which is an axiom scheme in $\mathrm{Th}(\mathbf{p})$ that speaks about the behaviour of $q$. In the following subsection we will show that SS, along with the full theories of both the domain and the variety sort, axiomatise $\mathrm{Th}(\mathbf{p})$. 

First, let us set some notation. If $\overline{g}=(g_{1},\ldots,g_{n})$ is a tuple from $G^{\mathrm{ad}}(\mathbb{Q})^{+}$, then $\overline{g}x = (g_{1}x,\ldots,g_{n}x)$. Whenever we are dealing with powers of the Shimura structure, the covering map will also be denoted by $q$, so that $q(\overline{g}x) = (q(g_{1}x),\ldots,q(g_{n}x))$. 

\begin{defi}
Given a special subvariety $V$ of $S^{n}$, define:
\begin{enumerate}
    \item $\mathrm{SS}^{1}_{V,i}:=\forall\overline{x}\in D_{V,i}\left(q(\overline{x})\in V\right)$.
    \item $\mathrm{SS}^{2}_{V,i}:=\forall z\in V\exists\overline{x}\in D_{V,i}\left(q(\overline{x})=z\right)$.
\end{enumerate}
We define the \emph{special subvarieties condition} \index{SS} to be the axiom scheme  $\mathrm{SS}:=\bigcup_{V,i}\mathrm{SS}^{1}_{V,i}\wedge\mathrm{SS}^{2}_{V,i}$.
\end{defi}

\begin{defi}
Given a tuple $\overline{g}=(g_{1},\ldots,g_{n})$ of elements in $G^{\mathrm{ad}}(\mathbb{Q})^{+}$, consider the sentences:
\begin{enumerate}
    \item $\mathrm{MOD}_{\overline{g},V,i}^{1}:= \forall x\in D_{V,i}\left((q(g_{1}x),\ldots,q(g_{n}x))\in Z_{\overline{g}}^{V,i}\right)$.
    \item $\mathrm{MOD}_{\overline{g},V,i}^{2}:= \forall z\in Z_{\overline{g}}^{V,i} \exists x\in D_{V,i}\left((q(g_{1}x),\ldots,q(g_{n}x))=z\right)$.
\end{enumerate}
Define the \emph{modularity condition}\index{MOD} as the axiom scheme $\mathrm{MOD}: = \bigcup_{\overline{g},V,i}\mathrm{MOD}_{\overline{g},V,i}^{1}\wedge\mathrm{MOD}_{\overline{g},V,i}^{2}.$
\end{defi}

\begin{remark}
Observe that MOD follows from SS. To see this recall Remark \ref{rem:indexzg}, and observe that whenever we have a definable set $A\subseteq D$, we have a corresponding definable set $Z_{\overline{g}}^{q(A)}$ (for any tuple $\overline{g}$ of elements of $G^{\mathrm{ad}}(\mathbb{Q})^{+}$) whose elements all have the form $q(\overline{g}a)$ for some $a\in A$, and there is a definable subset $B$ of $D^{n}$ (where $n$ is the length of $\overline{g}$) such that $q(B)=Z_{\overline{g}}^{q(A)}$, with a definable parametrisation from $A$, i.e. the map $A\rightarrow B$ given by $a\mapsto\overline{g}a$ is definable. This is because we can write in a first-order way the parametrisation of $D_{Z_{\overline{g}}^{V,i},i}$ in terms of $D_{V,i}$: $\overline{y}\in D_{Z_{\overline{g}}^{V,i},i}\iff \exists x\in D_{V,i}(\overline{y} = \overline{g}x)$. 
\end{remark}

Set $\mathrm{T}(\mathbf{p}):= \mathrm{Th}\left<X^{+}; \left\{g\right\}_{g\in G^{\mathrm{ad}}(\mathbb{Q})^{+}},\mathscr{R}_{D}\right>\cup T\cup\mathrm{SS}$. We will show in the next subsection that $\mathrm{T}(\mathbf{p})$ axiomatises $\mathrm{Th}(\mathbf{p})$.

\begin{remark}
\label{rem:mt2}
For $x\in X^{+}$, note that $\mathrm{qftp}(x)$ ``knows'' about the Mumford-Tate group of $x$.\index{Mumford-Tate group} Recall that the Mumford-Tate group determines the smallest special subvariety of $S(\mathbb{C})$ which contains $p(x)$ (see Remark \ref{rem:mt}). So, if $x_{1},x_{2}\in X^{+}$ have different Mumford-Tate groups, then $\mathrm{qftp}(x_{1})\neq\mathrm{qftp}(x_{2})$. 
\end{remark}

\begin{remark}
Suppose that $S$ is a pure Shimura variety. As special points of $S$ are zero-dimensional special subvarieties, then SS implies that if $x\in X^{+}$ is a special point and $g_{x}\in G^{\mathrm{ad}}(\mathbb{Q})^{+}$ fixes $x$ and only $x$ (which exists by Theorem \ref{thm:special}), then 
\begin{equation*}
    \mathrm{T}(\mathbf{p})\models\forall y\in D\left(g_{x}y = y\implies q(y)=p(x)\right),
\end{equation*}
where we view $p(x)$ as a constant symbol (because the coordinates of $p(x)$ are in $F_{0}$). In \cite{daw-harris} this is called the \emph{special points condition}. 
\end{remark}

\subsection{Quantifier Elimination}
Let $p:X^{+}\rightarrow S(\mathbb{C})$ be a Shimura variety and $\mathbf{p}$ its corresponding two-sorted structure. In this section we will show that $\mathrm{Th}(\mathbf{p})$ has quantifier elimination. To do that, we will show that $\mathrm{T}(\mathbf{p})$ has quantifier elimination and is complete. The strategy for this is to use Propositions \ref{prop:qe} and \ref{prop:completeness}.

\begin{prop}
\label{prop:qetp}
$\mathrm{T}(\mathbf{p})$ has quantifier elimination and is complete. In particular, $\mathrm{T}(\mathbf{p}) = \mathrm{Th}(\mathbf{p})$.
\end{prop}
\begin{proof}
Let $\mathbf{q}:=\left<\mathbf{D},\mathbf{S},q\right>$ and $\mathbf{q}':=\left<\mathbf{D}',\mathbf{S}',q'\right>$ be $\omega$-saturated models of $\mathrm{T}(\mathbf{p})$, and suppose we have $x_{1},\ldots,x_{m}\in D$ and $x_{1}',\ldots,x_{m}'\in D'$ such that $\mathrm{qftp}(x_{1},\ldots,x_{m}) = \mathrm{qftp}(x_{1}',\ldots,x_{m}')$. By Propositions \ref{prop:qe} and \ref{prop:completeness}, we need to show that given $y\in D$, there is $y'\in D'$ such that $\mathrm{qftp}(x_{1},\ldots,x_{m},y) = \mathrm{qftp}(x_{1}',\ldots,x_{m}',y')$. Let $L$ be the field generated over $E^{S}(\Sigma)$ by the coordinates of the $q(x_{i})$. 

Let $D_{V,i}$ be the special closure of $(x_{1},\ldots,x_{m},y)$. Using quantifier elimination of $T$ and $\omega$-saturation of $\mathbf{q}'$, we will first show that we can find $y'\in D'$ so that $(x_{1}',\ldots,x_{m}',y')\in D_{U,i}'$ and 
\begin{equation}
\label{eq:qe}
    \bigcup_{\overline{g}}\mathrm{qftp}\left(q(\overline{g}x_{1}),\ldots,q(\overline{g}x_{m}),q(\overline{g}y)/L\right) = \bigcup_{\overline{g}}\mathrm{qftp}\left(q'(\overline{g}x_{1}'),\ldots,q'(\overline{g}x_{m}'),q'(\overline{g}y')/L\right),
\end{equation}
where the unions are taken over all finite tuples $\overline{g}$ of elements of $G^{\mathrm{ad}}(\mathbb{Q})^{+}$. To see the equality of (\ref{eq:qe}) note that for a tuple $\overline{g}$, $\mathrm{qftp}\left(q(\overline{g}x_{1}),\ldots,q(\overline{g}x_{m}),q(\overline{g}y)/L\right)$ is determined by the smallest algebraic variety $W$ defined over $L$ which contains the point $\left(q(\overline{g}x_{1}),\ldots,q(\overline{g}x_{m}),q(\overline{g}y)\right)$. Notice that $W\subseteq  Z_{\overline{g}}^{V,i}$. 

Let $\mathrm{pr}$ be the coordinate projection that sends $(\overline{g}x_{1},\ldots,\overline{g}x_{m},\overline{g}y)\mapsto(\overline{g}x_{1},\ldots,\overline{g}x_{m})$. Applying $\mathrm{pr}$ to $\overline{g}D_{V,i}$ (remember Remark \ref{rem:proj}) will give us finitely many special domains, say $D_{U_{1},j_{1}},\ldots,D_{U_{r},j_{r}}$ (which can be proven in $T(\mathbf{p})$). On the other hand, by quantifier elimination of $T$, $\mathrm{pr}(W)$ is formed of finitely many subvarieties defined over $L$, say $W_{1},\ldots,W_{s}$.  Therefore, the formula 
\begin{equation*}
    \exists u\left((v_{1},\ldots,v_{m},u)\in D_{V,i}\wedge\left(q(\overline{g}v_{1}),\ldots,q(\overline{g}v_{m}),q(\overline{g}u)\right)\in W\right)
\end{equation*}
is equivalent to the quantifier-free formula 
\begin{equation*}
    (\overline{g}v_{1},\ldots,\overline{g}v_{m})\in \left(D_{U_{1},j_{1}}\cup\cdots\cup D_{U_{r},j_{r}}\right)\cap q^{-1}\left(W_{1}\cup\cdots\cup W_{s}\right).
\end{equation*}
As this formula is quantifier-free and it is satisfied by $(x_{1},\ldots,x_{m})$, it is also satisfied by $(x_{1}',\ldots,x_{m}')$, and therefore there is $y'\in D'$ such that $(x_{1}',\ldots,x_{m}',y')\in D_{V,i}'$ and
\begin{equation*}
   \mathrm{qftp}\left(q(\overline{g}x_{1}),\ldots,q(\overline{g}x_{m}),q(\overline{g}y)/L\right) = \mathrm{qftp} \left(q'(\overline{g}x_{1}'),\ldots,q'(\overline{g}x_{m}'),q'(\overline{g}y')/L\right).
\end{equation*}
Not only that, but as $\overline{g}D_{V,i}=\mathrm{spcl}(\overline{g}x_{1},\ldots,\overline{g}x_{m},\overline{g}y)$, then we have that $W\subseteq Z_{\overline{g}}^{V,i}$, and so $(\overline{g}x'_{1},\ldots,\overline{g}x'_{m},\overline{g}y')\in D'_{V,i}$. This means that the left-hand-side of (\ref{eq:qe}) plus the set of formulas
\begin{equation*}
    \left\{\left(\overline{g}v_{1},\ldots,\overline{g}v_{m},\overline{g}u\right)\in \overline{g}D_{V,i}\right\}_{\overline{g}}
\end{equation*}
form a finitely satisfiable set in $\mathbf{q}'$, so $\omega$-saturation gives us the $y'$ we wanted. 

Now we show that $\omega$-saturation allows us to choose $y'$ so that, in addition, $D_{V,j}'$ is the special closure of $\left(\overline{g}x_{1}',\ldots,\overline{g}x_{m}',\overline{g}y'\right)$.  A special domain cannot be covered by finitely many proper special subdomains, so we can avoid finitely many proper special subdomains of $D_{V,i}'$ when choosing $y'$ above. This finishes the proof.
\end{proof}

\subsection{Special Locus}

\begin{defi}
For $D_{V,i}\subseteq D^{m}$ and $a_{1},\ldots,a_{n}\in D$, where $n<m$, let $D_{V,i}(\overline{a})$ denote the set of tuples $\overline{y}=(y_{1},\ldots,y_{m-n})$ such that $(\overline{y},\overline{a})\in D_{V,i}$. Let $y_{1},\ldots,y_{m}\in D$ and $A\subseteq D$. We define the \emph{special locus}\index{special locus} of the tuple $\overline{y}$ over $A$, denoted $\mathrm{sploc}(\overline{y}/A)$,\index{sploc} to be the smallest set of the form $D_{V,i}(\overline{a})$ containing $\overline{y}$, where $\overline{a}$ is ranges through all finite tuples of $A$. 
\end{defi}

The following Lemma shows that the special locus is well-defined. 

\begin{lem}
\label{lem:sploc}
The special locus always exists and, if $C\subseteq D^{m}$ is the special locus of $\overline{y}$ over $A$, then $C$ is definable over $A$ and $q(C)$ is a subvariety of $S^{m}$ defined over $E^{S}(\Sigma,q(A))$. 
\end{lem}
\begin{proof}
We will assume $m=1$; the same proof goes through for higher dimensions. We will preserve the notation from the definition of special locus. To see that the special locus exists, consider first a set of the form $D_{V,i}(\overline{a})$. We know that $V = q(D_{V,i})$ is a special subvariety of $S^{n+1}$. Intersection with the hyperplanes $X_{k} = q(a_{k})$, $k=2,\ldots,n+1$, is still a subvariety (although probably not special), and then if we project onto the first coordinate we get that the set $q(D_{V,i}(\overline{a}))$ is a (possibly reducible) subvariety of $S$. So an arbitrary intersection of sets of the form $D_{V,i}(\overline{a})$ is actually a finite intersection. The finite intersection of special domains can be written as a finite union of some other special domains. So, for this reason, $C = \mathrm{sploc}(y/A)$ is of the form $D_{V,i}(\overline{a})$, for some finite tuple $\overline{a}$ of $A$, and as a set, it is definable over $A$. Also, by the above construction, $q(C)$ is definable over $E^{S}(\Sigma,q(A))$. 
\end{proof}

The next proposition follows from the proof of Proposition \ref{prop:qetp}. We will need this independent result later on, and so we will state it for future reference. We remark that this proposition does not require that the models be $\omega$-saturated. First we introduce the following notation.

Let $\mathbf{q}:=\left<\mathbf{D},\mathbf{S},q\right>$ and $\mathbf{q'}:=\left<\mathbf{D'},\mathbf{S}',q'\right>$ be models of $\mathrm{Th}(\mathbf{p})$. Suppose that there is  a partial isomorphism between $\mathbf{q}$ and $\mathbf{q}'$, represented by the partial isomorphism $\xi:D\dashrightarrow D'$ with finitely generated domain $U\subseteq D$, and an embedding $\sigma:L\rightarrow F'$, where $L$ is the field generated over $F_{0}$ by the coordinates of the images of the points in $U$. 

Observe that a set of the form $D_{V,i}(\overline{a})$ is a subset of $D^{n}$ (for some $n$) which is definable with parameters (given by $\overline{a}$). In what follows, when we write $\xi(\mathrm{sploc}(\overline{x}/U))$, what we mean is $D'_{V,i}(\xi(\overline{a}))$, where $D_{V,i}(\overline{a}) = \mathrm{sploc}(\overline{x}/U)$. 

\begin{prop}
\label{prop:genrealisation}
Let $x\in D_{V,i}\subseteq D$. Then for every tuple $\overline{g}=(g_{1},\ldots,g_{n})$ of $G^{\mathrm{ad}}(\mathbb{Q})^{+}$ there exists $x'\in D'_{V,i}\subseteq D'$ such that $(g_{1}x',\ldots,g_{n}x')\in\xi\left(\mathrm{sploc}(g_{1}x,\ldots,g_{n}x/U)\right)$ and $\left(q'(g_{1}x'),\ldots,q'(g_{n}x')\right)$ realises $\sigma\left(\mathrm{qftp}\left((q(g_{1}x),\ldots,q(g_{n}x))/L\right)\right)$. 
\end{prop}

In the following lemmas, we will preserve the notation used in the definition of special locus.

\begin{lem}
\label{lem:sploc2}
If $\overline{\alpha}$ is a subtuple of $\overline{g}$ and $y\in X^{+}$ is such that $\overline{g}y\in\mathrm{sploc}(\overline{g}x/A)$, for some $x\in X^{+}$ and $A\subseteq X^{+}$, then $\overline{\alpha}y\in\mathrm{sploc}(\overline{\alpha}x/A)$. 
\end{lem}
\begin{proof}
Let $\overline{\beta}$ be the complement subtuple of $\overline{\alpha}$ in $\overline{g}$, so that $\overline{g}=(\overline{\alpha},\overline{\beta})$. The result follows from noting that 
\begin{equation*}
    \mathrm{sploc}(\overline{g}x/A)\subseteq\mathrm{sploc}(\overline{\alpha}x/A)\times\mathrm{sploc}(\overline{\beta}x/A).
\end{equation*}
\end{proof}

\begin{lem}
\label{lem:sploc3}
Let $\overline{g}_{1},\ldots,\overline{g}_{n}$ be tuples of elements of $G^{\mathrm{ad}}(\mathbb{Q})^{+}$. Let $y_{1},\ldots,y_{n}\in D$, $A\subseteq D$ and let $C:=\mathrm{sploc}(y_{1},\ldots,y_{n}/A)$. Let $\epsilon:C\rightarrow D^{r}$ be given by $(x_{1},\ldots,x_{n})\mapsto(\overline{g}_{1}x_{1},\ldots,\overline{g}_{n}x_{n})$, where $r$ is the sum of the lengths of the $\overline{g}_{i}$. Then the image of $\epsilon$ equals $\mathrm{sploc}(\overline{g}_{1}y_{1}\ldots,\overline{g}_{n}y_{n}/A)$.  
\end{lem}
\begin{proof}
First we make a quick observation about $X^{+}$. If $X^{+}_{V,i}\subseteq \left(X^{+}\right)^{n}$ is a special domain and $g_{1},\ldots,g_{n}\in G^{\mathrm{ad}}(\mathbb{Q})^{+}$, then the image of the map:
\begin{align*}
X^{+}_{V,i}\quad &\xrightarrow{\qquad} \quad \left(X^{+}\right)^{r}\\
(x_{1},\ldots,x_{n}) &\mapsto (\overline{g}_{1}x_{1},\ldots,\overline{g}_{n}x_{n})
\end{align*}
is again a special domain. This is because we can use a straightforward adaptation of the proof of Lemma \ref{lem:2.6}. 

Now we return to the setting of the Lemma. We know that the special locus has the form $C = D_{V,i}(\overline{a})$ for some tuple of elements of $A$. Let $m$ be the length of $\overline{a}$. Before tackling the full version of the Lemma, let us focus on the following special case. Let $g_{1},\ldots,g_{n}$ be elements of $G^{\mathrm{ad}}(\mathbb{Q})^{+}$. We will prove that the map $\epsilon_{0}:C\rightarrow D^{n}$ given by $(x_{1},\ldots,x_{n})\mapsto(g_{1}x_{1},\ldots,g_{n}x_{n})$ is a bijection between $C$ and $C':=\mathrm{sploc}((g_{1}x_{1},\ldots,g_{n}x_{n})/A)$. Clearly $\epsilon_{0}$ is injective. Let 
\begin{align*}
\epsilon_{0}':\qquad D_{V,i}\qquad &\xrightarrow{\qquad\qquad} \quad D^{n+m}\\ 
(x_{1},\ldots,x_{n+m})&\mapsto (g_{1}x_{1},\ldots,g_{n}x_{n},x_{n+1},\ldots,x_{n+m}).
\end{align*}
By the above observation on $X^{+}$, we know that the image of $\epsilon_{0}'$ is a special domain, which shows that the image of $\epsilon_{0}$ contains $C'$. But now we can repeat the argument backwards to show that $C$ is contained in $\epsilon_{0}^{-1}(C')$. And so $\epsilon(C) = C'$.

Now we prove the general case. Consider the map
\begin{align*}
\epsilon':\qquad D_{V,i}\qquad &\xrightarrow{\qquad\qquad}\quad D^{r+m}\\ (x_{1},\ldots,x_{n+m})&\mapsto(\overline{g}_{1}x_{1},\ldots,\overline{g}_{n}x_{n},x_{n+1},\ldots,x_{n+m}).
\end{align*}
The image of $\epsilon'$ is a special domain of $D^{r+m}$ (because of our above observation on $X^{+}$). Let $D_{W,j}$ be the image of $\epsilon'$. This shows that the image of $\epsilon$ is of the form $D_{W,j}(\overline{a})$, and therefore $\mathrm{sploc}(\overline{g}_{1}y_{1}\ldots,\overline{g}_{n}y_{n}/A)$ is contained in the image of $\epsilon$.

Say $D_{W',j}(\overline{a}'):=\mathrm{sploc}(\overline{g}_{1}y_{1}\ldots,\overline{g}_{n}y_{n}/A)$. Observe that the previous paragraph shows that every point of $D_{W',j}(\overline{a}')$ is of the form $(\overline{g}_{1}x_{1}\ldots,\overline{g}_{n}x_{n})$, for some $(x_{1},\ldots,x_{n})\in C$. Let $\mathrm{pr}$ be the coordinate projection such that $(\overline{g}_{1}y_{1}\ldots,\overline{g}_{n}y_{n},\overline{a})\in D_{W',j}$ gets mapped to $(g_{1}y_{1},\ldots,g_{n}y_{n},\overline{a}')$. Let $D_{V',i}$ be the special domain contained in $\mathrm{pr}(D_{W',j})$ such that $(g_{1}y_{1},\ldots,g_{n}y_{n},\overline{a}')\in D_{V',i}$, with $g_{i}$ being the first coordinate of the tuple $\overline{g}_{i}$. Thanks to the special case, we now know that there is an explicit bijection between $C$ and $\mathrm{sploc}((g_{1}y_{1},\ldots,g_{n}y_{n})/A)$, which finishes the proof.
\end{proof} 

\begin{cor}
\label{cor:sploc}
For any tuple $\overline{g}$ of elements of $G^{\mathrm{ad}}(\mathbb{Q})^{+}$ and any $x\in D$, $\mathrm{sploc}(\overline{g}x/x) = \left\{\overline{g}x\right\}$
\end{cor}

\begin{lem}
\label{lem:sploc4}
Given $x_{1},\ldots,x_{n}\in D$ and $A\subseteq D$\ we have that 
\begin{equation*}
    \mathrm{sploc}\left(\overline{x}/A\right) = \mathrm{sploc}\left(\overline{x}/G^{\mathrm{ad}}(\mathbb{Q})^{+}A\right),
\end{equation*}
where $G^{\mathrm{ad}}(\mathbb{Q})^{+} A$ denotes the union of orbits of elements in $A$.  
\end{lem}
\begin{proof}
Let $D_{V,i}\left(\overline{a}\right) = \mathrm{sploc}\left(\overline{x}/A\right)$. If $m$ is the length of the tuple $\overline{a}$, then consider the tuple $(e,\ldots,e,g_{1},\ldots,g_{m})$. The map $\epsilon:D_{V,i}\rightarrow D_{W,j}$ defined by that element is a bijection (because the statement that says that it is a bijection between the corresponding special domains of $X^{+}$ is a first-order sentence in the theory of $X^{+}$).
\end{proof}

We finish this section by describing the types in models of $\mathrm{Th}(\mathbf{p})$. By Proposition \ref{prop:qetp}, it is enough to describe the quantifier-free types. First let us set some notation. Let $\mathbf{q}\models\mathrm{Th}(\mathbf{p})$, $x_{1},\ldots,x_{m}\in D$ and $A\subseteq D$. If $v_{1},\ldots,v_{m}$ are first-order variables in $\mathcal{L}_{D}$, we will write 
\begin{equation*}
    \mathrm{sploc}\left((\overline{g}v_{1},\ldots,\overline{g}v_{m})/A\right) = \mathrm{sploc}\left((\overline{g}x_{1},\ldots,\overline{g}x_{m})/A\right)
\end{equation*}
to mean that $\mathrm{sploc}\left((\overline{g}v_{1},\ldots,\overline{g}v_{m})/A\right)$ is the set of all the formulas in $\mathrm{qftp}\left(\overline{x}/A\right)$ which are either the form $(\overline{g}v_{1},\ldots,\overline{g}v_{m})\in D_{V,i}(\overline{a})$ or of the form $\neg\left((\overline{g}v_{1},\ldots,\overline{g}v_{m})\in D_{V,i}(\overline{a})\right)$. 

\begin{prop}
\label{prop:typedet1}
Let $\mathbf{q}\models\mathrm{Th}(\mathbf{p})$, $x_{1},\ldots,x_{m}\in D$, and let $A\subseteq D$ be closed under the action of $G^{\mathrm{ad}}(\mathbb{Q})^{+}$. Then $\mathrm{qftp}\left(\overline{x}/A\right)$ is determined by:
\begin{equation*}
   Q(\overline{v}):= \bigcup_{\overline{g}}\left[\mathrm{qftp}\left(q(\overline{g}x_{1},\ldots,\overline{g}x_{m})/L\right)\cup \left\{\mathrm{sploc}\left((\overline{g}v_{1},\ldots,\overline{g}v_{m})/A\right) = \mathrm{sploc}\left((\overline{g}x_{1},\ldots,\overline{g}x_{m})/A\right)\right\}\right],
\end{equation*}
where the union is taken over all tuples $\overline{g}=(g_{1},\ldots,g_{n})$ of $G^{\mathrm{ad}}(\mathbb{Q})^{+}$, and $L$ is the field generated over $F_{0}$ by the coordinates of the points in $q(A)$. 
\end{prop}
\begin{proof}
Let $y_{1},\ldots,y_{m}$ be a realisation of $Q(\overline{v})$ in some elementary extension of $\mathbf{q}$ that also satisfies SF. We need to show that $\mathrm{qftp}(\overline{x}/U_{\mathbf{D}}) = \mathrm{qftp}(\overline{y}/U_{\mathbf{D}})$.

The non-trivial atomic formulae in $\mathrm{qftp}(\overline{x}/A)$ that can be written in the domain language $\mathcal{L}_{D}$ are of the form: $x_{i}= a$, $x_{i}= gx_{j}$, $(\overline{g}x_{1},\ldots,\overline{g}x_{m})\in D_{V,i}(\overline{a})$, and their negations, for all parameters $a$ and $\overline{a}$ from $A$, and all $g$, $\overline{g}=(g_{1},\ldots,g_{n})$ from $G^{\mathrm{ad}}(\mathbb{Q})^{+}$. The quantifier-free formulas in the language $\mathcal{L}_{F}$ in $\mathrm{qftp}(\overline{x}/A)$ are already contained in $Q(\overline{v})$. 

Observe that formulae of the forms $x_{i}= a$ or $x_{i}= gx_{j}$ are in fact equivalent to a formula of the form $(\overline{g}x_{1},\ldots,\overline{g}x_{m})\in D_{V,i}(\overline{a})$, using the MOD conditions of the special subvarieties $Z_{(e,e)}$ and $Z_{(e,g)}$. So we only need to focus on this last case.

Note that $(\overline{g}x_{1},\ldots,\overline{g}x_{m})\in D_{V,i}(\overline{a})$, for some $\overline{g}=(g_{1},\ldots,g_{n})$ from $G^{\mathrm{ad}}(\mathbb{Q})^{+}$ and some tuple $\overline{a}$ from $A$, if and only if $\mathrm{sploc}\left((\overline{g}x_{1},\ldots,\overline{g}x_{m})/A\right)\subseteq D_{V,i}(\overline{a})$. So $(\overline{g}x_{1},\ldots,\overline{g}x_{m})\in D_{V,i}(\overline{a})$ implies $(\overline{g}y_{1},\ldots,\overline{g}y_{m})\in D_{V,i}(\overline{a})$. The converse works the same way.
\end{proof}

\section{Models With Standard Fibres}
\label{sec:sf}
From its construction, we know that the fibres of $p:X^{+}\rightarrow S(\mathbb{C})$ consist of exactly one $\Gamma$-orbit. As $\Gamma$ is infinite, it seems unlikely at first sight that this property about the fibres of $p$ can be expressed with a first-order sentence in $\mathcal{L}$. This is verified once we consider the following partial types:
\begin{enumerate}
    \item Let $z\in S(\mathbb{C})$ be a non-Hodge-generic point, that is $\mathrm{spcl}(z) = V\subsetneq S(\mathbb{C})$. For $V$, there are (countably) infintely many corresponding special domains $X^{+}_{V,i}\subset X^{+}$. So the set of formulas (in the variable $v$ with parameter $z$):
    \begin{equation*}
        \left\{(q(v) = z)\wedge (v\notin D_{V,i})\right\}_{i\in\mathbb{N}}
    \end{equation*}
    is finitely satisfiable, and so it is part of a complete type. Therefore it is realised in a saturated model of $\mathrm{Th}(\mathbf{p})$. Observe that if $x$ is a realisation of such a type, then we will see discrepancies between special closures on the domain and special closure on the variety, specifically: $\mathrm{spcl}(q(x))\subsetneq q(\mathrm{spcl}(x))$. In this case, we will say that $x$ is a \emph{non-standard point}.
    \item Even if the domain of a Shimura structure consists only of standard points, we can still have that the fibres of $q$ consists of many $\Gamma$-orbits. To see this, choose $x\in X^{+}$ non-special, let $D_{V,i}=\mathrm{spcl}(x)$, and consider the set of formulas (with variable $v$ and parameter $x$):
    \begin{equation*}
        \left\{(q(v)=q(x))\wedge(v\in D_{V,i})\wedge(\gamma x\neq v)\right\}_{i\in\mathbb{N}, \gamma\in\Gamma}.
    \end{equation*}
    That this set is finitely satisfiable, is a consequence of the fact that, for every $z\in V$, $p^{-1}(z)\cap X^{+}_{V,i}$ is made up of exactly one $\Gamma^{H}$-orbit, where $\Gamma^{H} = \Gamma\cap H(\mathbb{Q})_{+}$, and $H=\mathrm{MT}(x)$. 
\end{enumerate}

We will deal with the presence non-standard points and new standard points when we study Shimura covers. In this section we will instead omit these types as we will restrict to the class of models of $\mathrm{Th}(\mathbf{p})$ that satisfy the \emph{standard fibres condition},\index{SF} which is the $\mathcal{L}_{\omega_{1},\omega}$-sentence:
\begin{equation*}
    \mathrm{SF}:=\forall x, y\in D\left(\left(q(x)=q(y)\right)\rightarrow\bigvee_{\gamma\in\Gamma}(x=\gamma y)\right).
\end{equation*}

\begin{defi}
Given $\mathbf{q}\models\mathrm{Th}(\mathbf{p})$, a point $x\in D$ is called a \emph{standard point} \index{standard point} if $q\left(\mathrm{spcl}(x)\right) = \mathrm{spcl}(q(x))$. A \emph{standard section} \index{standard section} of $q$ is a section $A\subseteq D$ that only contains standard points and such that if $z\in A$, then for every $g\in G^{\mathrm{ad}}(\mathbb{Q})^{+}$ there is $h\in  G^{\mathrm{ad}}(\mathbb{Q})^{+}$ such that $hz\in A$ and $q(hz) = q(gz)$.
\end{defi}

\begin{remark}
\label{rem:stptexist}
We point out that by $\mathrm{MOD}_{\overline{g},V,i}^{2}$, if $q:D\rightarrow S(F)\models \mathrm{Th}(\mathbf{p})$, then every $z\in S(F)$ has a standard point in its $q$-fibre. Also, is $x\in D$ is standard, then $gx$ is standard for every $g\in G^{\mathrm{ad}}(\mathbb{Q})^{+}$.
\end{remark}

Let $\mathrm{Th}_{\mathrm{SF}}(\mathbf{p}) := \mathrm{Th}(\mathbf{p})\cup SF$.\index{t@$\mathrm{Th}_{\mathrm{SF}}(\mathbf{p})$} As the next Proposition shows, we can construct models of $\mathrm{Th}_{\mathrm{SF}}(\mathbf{p})$ from models of $\mathrm{Th}(\mathbf{p})$.

\begin{lem}
\label{lem:sf}
Given any model $\mathbf{q}\models\mathrm{Th}(\mathbf{p})$, there is $\mathbf{q}'\prec\mathbf{q}$ such that $S(F) = S(F')$ and $\mathbf{q}'\models\mathrm{SF}$. 
\end{lem}
\begin{proof}
First observe that if $x\in D$ is standard, then $gx$ is standard for any $g\in  G^{\mathrm{ad}}(\mathbb{Q})^{+}$. Now choose a standard section of $q$, call it $A$. Standard sections exist because we can look at the space of orbits of $D$ modulo $G^{\mathrm{ad}}(\mathbb{Q})^{+}$ and first choose orbits which contain standard points, and then from each of these orbits, choose representatives modulo $\Gamma$. Let $D' = \Gamma A$ and let $q'$ be the restriction of $q$ to $D'$. Define $\mathbf{q}'$ as the structure given by $q':D'\rightarrow S(F)$, where $D'_{V,i} = D_{V,i}\cap \left(D'\right)^{n}$.

The action of $G^{\mathrm{ad}}(\mathbb{Q})^{+}$ restricts to an action on $D'$ because of the definition of standard section. No $D_{V,i}'$ can be empty because the special points of $D$ are in $D'$. Indeed, this is because the special domains of a special point in $S(F)$ are all singletons. 

It is clear that $\mathbf{q}'$ satisfies SS and, because the $D_{V,i}'$ are just restrictions of the $D_{V,i}$, we have that $\mathbf{q}'$ satisfies $\mathrm{T}(\mathbf{p})$. So $\mathbf{q}'\models\mathrm{Th}_{\mathrm{SF}}(\mathbf{p})$ by Proposition \ref{prop:qetp}.
\end{proof}

When our models have standard fibres, we can strengthen the result from Proposition \ref{prop:typedet1} to give a better description of the types.

\begin{prop}
\label{prop:typedet}
Let $\mathbf{q}\models\mathrm{Th}_{\mathrm{SF}}(\mathbf{p})$, $x_{1},\ldots,x_{m}\in D$, and let $A\subseteq D$ be closed under the action of $G^{\mathrm{ad}}(\mathbb{Q})^{+}$. Then $\mathrm{qftp}(\overline{x}/A)$ is determined (in $\mathrm{Th}_{\mathrm{SF}}(\mathbf{p})$) by:
\begin{equation*}
   Q(\overline{v}):= \bigcup_{\overline{g}}\left[\mathrm{qftp}\left(q(\overline{g}x_{1},\ldots,\overline{g}x_{m})/L\right)\cup \left\{(\overline{g}v_{1},\ldots,\overline{g}v_{m})\in\mathrm{sploc}\left((\overline{g}x_{1},\ldots,\overline{g}x_{m})/A\right)\right\}\right],
\end{equation*}
where the union is taken over all tuples $\overline{g}=(g_{1},\ldots,g_{n})$ of $G^{\mathrm{ad}}(\mathbb{Q})^{+}$, and $L$ is the field generated over $F_{0}$ by the coordinates of the points in $q(A)$. 
\end{prop}
\begin{proof}
Let $y_{1},\ldots,y_{m}$ be a realisation of $Q(\overline{v})$ in some elementary extension of $\mathbf{q}$ that also satisfies SF. By Proposition \ref{prop:typedet1} is will suffice to show that $\mathrm{sploc}\left((\overline{g}y_{1},\ldots,\overline{g}y_{m})/A\right) = \mathrm{sploc}\left((\overline{g}x_{1},\ldots,\overline{g}x_{m})/A\right)$.

Note that $(\overline{g}x_{1},\ldots,\overline{g}x_{m})\in D_{V,i}(\overline{a})$, for some $\overline{g}=(g_{1},\ldots,g_{n})$ from $G^{\mathrm{ad}}(\mathbb{Q})^{+}$ and some tuple $\overline{a}$ from $A$, if and only if $\mathrm{sploc}\left((\overline{g}x_{1},\ldots,\overline{g}x_{m})/A\right)\subseteq D_{V,i}(\overline{a})$. So $(\overline{g}x_{1},\ldots,\overline{g}x_{m})\in D_{V,i}(\overline{a})$ implies $(\overline{g}y_{1},\ldots,\overline{g}y_{m})\in D_{V,i}(\overline{a})$. Conversely, suppose $(\overline{g}y_{1},\ldots,\overline{g}y_{m})\in D_{V,i}(\overline{a})$. By hypothesis and SF, there is $\overline{\gamma}\in\Gamma^{\ell m+d}$, where $\ell$ is the length of $\overline{g}$ and $d$ is the length of $\overline{a}$, such that $\overline{\gamma}(\overline{g}x_{1},\ldots,\overline{g}x_{m},\overline{a}) = (\overline{g}y_{1},\ldots,\overline{g}y_{m},\overline{a})$. So $(\overline{g}x_{1},\ldots,\overline{g}x_{m},\overline{a})\in D_{V,k}$, for some $k$. But that means that $\mathrm{sploc}\left((\overline{g}x_{1},\ldots,\overline{g}x_{m})/A\right)\subseteq D_{V,k}(\overline{a})$. As special domains of the same special subvariety are either the same or disjoint, we conclude that $D_{V,i} =D_{V,k}$.
\end{proof}

\begin{remark}
\label{rem:orbit}
Let $x_{1},\ldots,x_{m}\in D$ be in different $G^{\mathrm{ad}}(\mathbb{Q})^{+}$-orbits, and let $\overline{g} = (e,g_{1},\ldots,g_{n})$ be a tuple of distinct elements of $G^{\mathrm{ad}}(\mathbb{Q})^{+}$. Let $L$ be the field generated over $F_{0}$ with the coordinates of the $q(x_{i})$, and let $V=\mathrm{spcl}\left(q(x_{1}),\ldots,q(x_{m})\right)$. As we have said before,  $\mathrm{qftp}\left(\left(q(\overline{g}x_{1}),\ldots,q(\overline{g}x_{n})\right)/L)\right)$ is equivalent to the minimal algebraic subset of $Z_{\overline{g}}^{V,i}$ defined over $L$ that contains the tuple. This is a subset of the fibre over $\left(q(x_{1}),\ldots,q(x_{m})\right)$ of the morphism
\begin{equation*}
    \psi_{\overline{g},(e)}^{V,i}:Z_{\overline{g}}^{V,i}\rightarrow S(\mathbb{C})^{m}.
\end{equation*}
In fact, it is the $\mathrm{Aut}(\mathbb{C}/L)$-orbit in this fibre containing this tuple. 
\end{remark}

\section{Shimura Covers}
\label{sec:covers}
From a model theoretic point of view, $\mathrm{Th}(\mathbf{p})$ may not be the most natural first-order theory one can associate with a Shimura variety. For instance, given $S(F)\models T$, it is not immediately clear if one can find an $\mathcal{L}_{D}$-structure $D$ and a map $q:D\rightarrow S(F)$ so that this becomes a model of $\mathrm{Th}(\mathbf{p})$. We address this question by first taking a seemingly different approach to defining Shimura structures.

\subsection{The covering sort}
Set $\Gamma_{\overline{g}}:=g_{1}^{-1}\Gamma g_{1}\cap\cdots\cap g_{n}^{-1}\Gamma g_{n}$, where $\overline{g} = (g_{1},\ldots,g_{n})$ is a  tuple of $G^{\mathrm{ad}}(\mathbb{Q})^{+}$. The quotients of $X^{+}$ by the groups $\Gamma_{\overline{g}}$ form an inverse system of varieties, where the morphisms are those induced by the inclusion of groups. Let $\widehat{S}(\mathbb{C})$ \index{s@$\widehat{S}(\mathbb{C})$} denote the inverse limit of this system. Denote an equivalence class of $\Gamma_{\overline{g}}\backslash X^{+}$ by $[\cdot]_{\overline{g}}$, with the implicit understanding that if $\Gamma_{\overline{g}} = \Gamma_{\overline{h}}$, then $[\cdot]_{\overline{g}} = [\cdot]_{\overline{h}}$. So a point in $\widehat{S}(\mathbb{C})$ can be thought of as a collection of points indexed by tuples $\overline{g}$, such that $\left[x_{\overline{g}}\right]_{\overline{g}}\in \Gamma_{\overline{g}}\backslash X^{+}$ and satisfying that if $\Gamma_{\overline{g}}\subseteq\Gamma_{\overline{h}}$, then $\psi_{\overline{g},\overline{h}}\left(\left[x_{\overline{g}}\right]_{\overline{g}}\right)=\left[x_{\overline{h}}\right]_{\overline{h}}$.

We can define an action of $G^{\mathrm{ad}}(\mathbb{Q})^{+}$ on $\widehat{S}(\mathbb{C})$ as the action of $\alpha\in G^{\mathrm{ad}}(\mathbb{Q})^{+}$ on $X^{+}$ induces a map
\begin{equation*}
    \Gamma_{\overline{g}}\backslash X^{+}\rightarrow\alpha\Gamma_{\overline{g}}\alpha^{-1}\backslash X^{+}
\end{equation*}
which, as a map of algebraic varieties, is defined over $E^{\mathrm{ab}}$. If we let $\overline{g}\alpha^{-1} := (g_{1}\alpha^{-1},\ldots,g_{n}\alpha^{-1})$, then we can describe the action on components by:
\begin{equation*}
    \alpha\left[x_{\overline{g}}\right]_{\overline{g}} := \left[\alpha x_{\overline{g}}\right]_{\overline{g}\alpha^{-1}}.
\end{equation*}

We cite the following Lemma with the observation that its proof can still be used for the case of mixed Shimura varieties.

\begin{lem}[see {{\cite[Lemma 4.3]{daw-harris}}}]
\label{lem:centre}
The group $\Gamma_{\infty}:=\bigcap_{\overline{g}}\Gamma_{\overline{g}}$, where the intersection ranges over all finite tuples of distinct elements of $G^{\mathrm{ad}}(\mathbb{Q})^{+}$, is contained in $Z_{G}(\mathbb{Q})$, where $Z_{G}$ stands for the centre of $G$. 
\end{lem}

\begin{lem}
\label{lem:embedding}
There is an embedding of $X^{+}$ into $\widehat{S}(\mathbb{C})$ given by:
\begin{equation*}
\begin{array}{cccc}
    \iota:& X^{+} & \longrightarrow &\widehat{S}(\mathbb{C})\\ 
       & x &\mapsto &\left([x]_{\overline{g}}\right)_{\overline{g}}
\end{array},
\end{equation*}
which is $G^{\mathrm{ad}}(\mathbb{Q})^{+}$-equivariant.
\end{lem}
\begin{proof}
Equivariance is immediate from the definition of the action of $G^{\mathrm{ad}}(\mathbb{Q})^{+}$ on $\widehat{S}(\mathbb{C})$. Let $x,y\in X^{+}$ be such that $\iota(x)=\iota(y)$. Let $\overline{g} = (e,g_{1},\ldots,g_{n})$. We know from \textsection \ref{subsec:galrep} that there is an isomorphism:
\begin{align*}
    \Gamma_{\overline{g}}\backslash X^{+}&\xrightarrow{\qquad\sim\qquad} Z_{\overline{g}}\\ [x]_{\overline{g}}&\mapsto (p(x),p(g_{1}x),\ldots,p(g_{n}x)).
\end{align*}
Therefore we have $\left(p(x),p(g_{1}x),\ldots,p(g_{n}x)\right) = \left(p(y),p(g_{1}y),\ldots,p(g_{n}y)\right)$, which means that $y\in \Gamma_{\overline{g}}x$. As this happens for every tuple of the form $\overline{g}$, we get that $y\in\Gamma_{\infty}x$, so that $y=x$ by Lemma \ref{lem:centre}. 
\end{proof}

Let $\widehat{p}:\widehat{S}(\mathbb{C})\rightarrow S(\mathbb{C})$ be given by $\widehat{p}\left(\left(\left[x_{\overline{g}}\right]_{\overline{g}}\right)_{\overline{g}}\right) = p(x_{e})$. Thus we get a commutative diagram:
\begin{center}
\begin{tikzcd}[column sep=small]
X^{+} \arrow[dr,"p"'] \arrow[rr,"\iota"] & & \widehat{S}(\mathbb{C})  \arrow[dl,"\widehat{p}"] \\
 & S(\mathbb{C}) &
\end{tikzcd}
\end{center}
where $\iota$ is the embedding given in Lemma \ref{lem:embedding}. 

Now we will use a new form of notation for the points of $\widehat{S}(\mathbb{C})$. From the construction, it is immediate that we can see $\widehat{S}(\mathbb{C})$ as the inverse limit of the system of special subvarieties $Z_{\overline{g}}$, for all tuples $\overline{g}$ from $G^{\mathrm{ad}}(\mathbb{Q})^{+}$. So a point $x\in\widehat{S}(\mathbb{C})$ can be written as $x=\left(z_{\overline{g}}\right)_{\overline{g}}$, where $z_{\overline{g}}\in Z_{\overline{g}}$, and whenever $\overline{g}\subseteq\overline{h}$, 
\begin{equation}
\label{eq:compat}
     \psi_{\overline{h},\overline{g}}\left(z_{\overline{h}}\right) = z_{\overline{g}}.  
\end{equation}

This structure comes equipped with natural maps $\widehat{p}_{\overline{g}}:\widehat{S}(\mathbb{C})\rightarrow Z_{\overline{g}}$,\index{p@$\widehat{p}_{\overline{g}}$} which are determined by the natural map $\widehat{p}:\widehat{S}(\mathbb{C})\rightarrow S(\mathbb{C})$\index{p@$\widehat{p}$} given by $\widehat{p}_{\overline{g}}\left(\widetilde{x}\right) = z_{\overline{g}}$.

\begin{lem}
For every $\widetilde{x}\in\widehat{S}(\mathbb{C})$ and any tuple $\overline{g}$ of $G^{\mathrm{ad}}(\mathbb{Q})^{+}$, we have that:
\begin{equation*}
    \widehat{p}_{\overline{g}}\left(\widetilde{x}\right) = \widehat{p}\left(\overline{g}\widetilde{x}\right).
\end{equation*}
\end{lem}
\begin{proof}
In order to avoid messy notation, we will prove the statement for an arbitrary tuple from $G^{\mathrm{ad}}(\mathbb{Q})^{+}$ of the form $(\alpha_{1},\alpha_{2})$. The general case can be proven using the same argument. Choose $\widetilde{x}\in\widehat{S}(\mathbb{C})$. We denote its components using both systems of notation, i.e. $\widetilde{x} = (z_{\overline{g}})_{\overline{g}}$ (with $z_{\overline{g}}\in Z_{\overline{g}}$), and $\widetilde{x} = \left(\left[x_{\overline{g}}\right]_{\overline{g}}\right)_{\overline{g}}$, where $z_{\overline{g}} = p\left(\overline{g}x_{\overline{g}}\right)$. 

On one hand $\widehat{p}_{(\alpha_{1},\alpha_{2})}\left(\widetilde{x}\right) = z_{(\alpha_{1},\alpha_{2})}$. On the other hand, we know that $\alpha_{1}[x_{\overline{g}}]_{\overline{g}} = \left[\alpha_{1} x_{\overline{g}}\right]_{\overline{g}\alpha_{1}^{-1}}$, so in particular $\alpha_{1}\left[x_{\alpha_{1}}\right]_{\alpha_{1}} = \left[\alpha_{1} x_{\alpha_{1}}\right]_{e}$. Therefore $\widehat{p}\left(\alpha_{1}\widetilde{x}, \alpha_{2}\widetilde{x}\right) = \left(\left[\alpha_{1} x_{\alpha_{1}}\right]_{e}, \left[\alpha_{2} x_{\alpha_{2}}\right]_{e}\right)$. 

By compatibility of components, we have that $\left[x_{(\alpha_{1},\alpha_{1})}\right]_{\alpha_{i}} = \left[x_{\alpha_{i}}\right]_{\alpha_{i}}$ for $i=1,2$. This means that there are $\gamma_{1}\in\alpha_{1}^{-1}\Gamma\alpha_{1}$ and $\gamma_{2}\in\alpha_{2}^{-1}\Gamma\alpha_{2}$ such that $x_{\alpha_{1}} = \gamma_{1}x_{(\alpha_{1},\alpha_{2})}$ and $x_{\alpha_{2}} = \gamma_{2} x_{(\alpha_{1},\alpha_{2})}$. So there are $\gamma_{1}',\gamma_{2}'\in\Gamma$ such that $\alpha_{1}x_{\alpha_{1}} = \gamma_{1}'\alpha_{1}x_{(\alpha_{1},\alpha_{2})}$ and $\alpha_{2}x_{\alpha_{2}} = \gamma_{2}'\alpha_{2}x_{(\alpha_{1},\alpha_{2})}$, so that $\widehat{p}\left(\alpha_{1}\widetilde{x}, \alpha_{2}\widetilde{x}\right)=p\left(\alpha_{1}x_{(\alpha_{1},\alpha_{2})}, \alpha_{2}x_{(\alpha_{1},\alpha_{2})}\right) = z_{(\alpha_{1},\alpha_{2})}$.
\end{proof}
 
If $x_{1},\ldots,x_{m}\in\widehat{S}(\mathbb{C})$, we define the map $\widehat{p}_{\overline{g}}(x_{1},\ldots,x_{m}):=\left(\widehat{p}_{\overline{g}}(x_{1}),\ldots,\widehat{p}_{\overline{g}}(x_{m})\right)$. Given a special subvariety $V$ of $S(\mathbb{C})^{m}$, we will define the set $\widehat{S}(\mathbb{C})_{V,i}$ \index{s@$\widehat{S}(\mathbb{C})_{V,i}$} as the set of points $x\in\widehat{S}^{m}$ such that for every tuple $\overline{g}$ from $G^{\mathrm{ad}}(\mathbb{Q})^{+}$, $\widehat{p}_{\overline{g}}(x)\in Z_{\overline{g}}^{V,i}$. 

And so, we get a two-sorted structure $\widehat{\mathbf{p}}:=\left<\widehat{S}(\mathbb{C}), S(\mathbb{C}), \left\{\widehat{p}_{\overline{g}}\right\}_{\overline{g}}\right>$,\index{p@$\widehat{\mathbf{p}}$} with each $\widehat{p}_{\overline{g}}$ denoting a function from $\widehat{S}(\mathbb{C})$ to a corresponding power of $S(\mathbb{C})$. The language for $S(\mathbb{C})$ is the same language $\mathcal{L}_{F}$ from before, but now the language on $\widehat{S}(\mathbb{C})$ consists only of relation symbols for each $\widehat{S}(\mathbb{C})_{V,i}$, without an explicit action from $G^{\mathrm{ad}}(\mathbb{Q})^{+}$. 

\begin{defi}[Shimura covers]
A \emph{covering structure} for a Shimura variety $p:X^{+}\rightarrow S(\mathbb{C})$ is defined in the following way.
\begin{enumerate}[(a)]
\item Let $\mathcal{L}_{\mathrm{cov}}$ \index{l@$\mathcal{L}_{\mathrm{cov}}$} be the language consisting of relation symbols for every $\widehat{S}(\mathbb{C})_{V,i}$, which can be seen as a reduct of the language $\mathcal{L}_{D}$. Let $\widetilde{\mathcal{L}}$ be the language $\widetilde{\mathcal{L}}:=\mathcal{L}_{\mathrm{cov}}\cup\mathscr{L}_{F}\cup\left\{q_{\overline{g}}\right\}$,\index{l@$\widetilde{\mathcal{L}}$} where every $q_{\overline{g}}$ is a function symbol, indexed by a tuple $\overline{g}$ from $G^{\mathrm{ad}}(\mathbb{Q})^{+}$.
\item Let $T=\mathrm{Th}(S(\mathbb{C}))$ as before, and let $\widetilde{T}(\widehat{\mathbf{p}})$ \index{t@$\widetilde{T}(\widehat{\mathbf{p}})$} be the complete first-order theory of of the two-sorted structure $\widehat{\mathbf{p}}$ in the language $\widetilde{\mathcal{L}}$. Models of $\widetilde{T}$ will be called \emph{Shimura covers}. \index{Shimura covers}
\item Models of $T$ will usually be denoted simply by $M$, and sometimes we will be more explicit by saying that $M=S(F)$ for some algebraically closed field $F$. Models of $\widetilde{T}(\widehat{\mathbf{p}})$ will be denoted $\widetilde{\mathbf{q}}:=\left<\widetilde{M}, M, \left\{q_{\overline{g}}\right\}\right>$ \index{q@$\widetilde{\mathbf{q}}$}. Given a special subvariety $V$ of $M$, let $\widetilde{M}_{V,i}$\index{m@$\widetilde{M}_{V,i}$} denote a subset of the corresponding power of $\widetilde{M}$. These sets will still be called \emph{special domains}.\index{special domain}
\item Given an $\widetilde{\mathcal{L}}$-structure $\left<\widetilde{M}, M, \left\{q_{\overline{g}}\right\}\right>$ and $B\subseteq\widetilde{M}$, let $\widehat{B}$ denote the set of points in $M$ of the form $q_{g}(b)$, for all $b\in B$ and all $g\in G^{\mathrm{ad}}(\mathbb{Q})^{+}$.
\end{enumerate}
\end{defi}

The expectation now is that given $S(F)\models\mathrm{Th}(S(\mathbb{C}))$, we can construct $\widehat{S}(F)$ as the inverse limit of the $Z_{\overline{g}}(F)$, and then the corresponding structure $\left<\widehat{S}(F), S(F), \left\{q_{\overline{g}}\right\}\right>$ will be a model of $\widetilde{T}(\widehat{\mathbf{p}})$. There are, however, a few important details that we need to consider before proceeding.

On one hand, it is true that for every special point $z\in S(\mathbb{C})$, the special domains $\widehat{S}(\mathbb{C})_{z,i}$ consist of exactly one point, which is also true of the special domains $X_{z,i}^{+}$. However, if $p(x) = z$, then Theorem \ref{thm:special} says that there is some $\alpha\in G^{\mathrm{ad}}(\mathbb{Q})^{+}$ such that $x$ is the unique fixed point of $\alpha$. This means then that the intersection of special domains $X_{Z_{(e,\alpha)},1}^{+}\cap X^{+}_{Z_{(e,e)},1}$ consists of exactly one point: $(x,x)$. But the intersection $\widehat{S}(\mathbb{C})_{Z_{(e,\alpha)},1}\cap\widehat{S}(\mathbb{C})_{Z_{(e,e)},1}$ has more than one point, as $\alpha$ now fixes more elements of $\widehat{S}(\mathbb{C})$. In other words, if we view $\widehat{S}(\mathbb{C})$ as an $\mathcal{L}_{D}$-structure, then it is not elementarily equivalent to $X^{+}$. 

It makes sense then to consider the $\mathcal{L}_{\mathrm{cov}}$-substructure of $\widehat{S}(\mathbb{C})$ which consists only of standard points, call it $\widetilde{S}(\mathbb{C})$. Given a tuple $\overline{g}$ of $G^{\mathrm{ad}}(\mathbb{Q})^{+}$, let $\widetilde{p}_{\overline{g}}:\widetilde{S}(\mathbb{C})\rightarrow S(\mathbb{C})^{n}$ be the restriction of $\widehat{p}_{\overline{g}}$. We therefore get an $\widetilde{\mathcal{L}}$-structure $\widetilde{\mathbf{p}}:=\left<\widetilde{S}(\mathbb{C}), S(\mathbb{C}), \left\{\widetilde{p}_{\overline{g}}\right\}\right>$. In fact, $\widetilde{S}(\mathbb{C})$ also inherits a $G^{\mathrm{ad}}(\mathbb{Q})^{+}$-action from $\widehat{S}(\mathbb{C})$. 

\begin{lem}
\label{lem:standardness}
$\widetilde{x}\in\widehat{S}(\mathbb{C})$ is non-standard if and only if there is a tuple $\overline{g}$ of elements of $G^{\mathrm{ad}}(\mathbb{Q})^{+}$ and a standard point $\widetilde{y}\in\mathrm{spcl}(\widetilde{x})$ such that $\widehat{p}(\widetilde{x}) = \widehat{p}(\widetilde{y})$ and $\widetilde{p}_{\overline{g}}(\widetilde{x})\notin\mathrm{spcl}(\widehat{p}_{\overline{g}}(\widetilde{y}))$.
\end{lem}
\begin{proof}
Suppose first that $\widetilde{x}$ is non-standard. Let $V=\mathrm{spcl}(\widehat{p}(\widetilde{x}))$ and $\widehat{S}(\mathbb{C})_{W,j} = \mathrm{spcl}(\widetilde{x})$. Then $V\subsetneq W$. For some $i\in\mathbb{N}$ we can find $\widetilde{y}\in\widehat{S}(\mathbb{C})_{V,i}\subsetneq\widehat{S}(\mathbb{C})_{W,j}$ such that $\widehat{p}(\widetilde{y}) = \widehat{p}(\widetilde{x})$. As $\widetilde{x}\notin\widehat{S}(\mathbb{C})_{V,i}$, then by definition there must be a tuple $\overline{g}$ such that $\widehat{p}_{\overline{g}}(\widetilde{x})\notin Z_{\overline{g}}^{V,i}$. But by Lemma \ref{lem:sploc3}, $Z_{\overline{g}}^{V,i} = \mathrm{spcl}(\widehat{p}(\widetilde{y}))$. 

Conversely, assume that there exist $\overline{g}$ and $\widetilde{y}$ satisfying the conditions for $\widetilde{x}$. Then $\mathrm{spcl}(\widehat{p}(\widetilde{x}))=\mathrm{spcl}(\widehat{p}(\widetilde{y}))=V$, and as $\widetilde{y}$ is standard, there is some $i\in\mathbb{N}$ such that $\widetilde{y}\in\widehat{S}(\mathbb{C})_{V,i}$. As $\widehat{p}(\widetilde{x})\notin\mathrm{spcl}(\widehat{p}(\widetilde{y})) = Z_{\overline{g}}^{V,i}$, we get that $\widetilde{x}\notin\widehat{S}(\mathbb{C})_{V,i}$. But because $\widetilde{y}\in\mathrm{spcl}(\widetilde{x})$, it must be that $\widehat{S}(\mathbb{C})_{V,i}\subsetneq\mathrm{spcl}(\widetilde{x})$, showing that $\widetilde{x}$ is non-standard.
\end{proof}

In what remains of this section, we will consider both $\widetilde{T}(\widehat{\mathbf{p}})$ and $\widetilde{T}(\widetilde{\mathbf{p}})$, the latter being defined as the complete first-order theory of $\widetilde{\mathbf{p}}$ in the language $\widetilde{\mathcal{L}}$. Many of the results that follow hold in both theories, but in \textsection\ref{sec:classification} we will restrict to $\widetilde{T}(\widetilde{\mathbf{p}})$, because this theory is modeled by $\left<X^{+}, S(\mathbb{C}), \left\{p_{\overline{g}}\right\}\right>$, whereas the same is not true about $\widetilde{T}(\widehat{\mathbf{p}})$.

\subsection{Axiomatisation and quantifier elimination}
\label{subsec:axiomscov}
Fix a Shimura variety $p:X^{+}\rightarrow S(\mathbb{C})$. Define:

\begin{description}
\item[(A1)] $M\models T$.
\item[(A2)] For every special domain $\widetilde{M}_{V,i}$ and all tuples $\overline{g}$ from $G^{\mathrm{ad}}(\mathbb{\mathbb{Q}})^{+}$, $q_{\overline{g}}\left(\widetilde{M}_{V,i}\right) = Z_{\overline{g}}^{V,i}$.
\item[(A3)] For every special domain $\widetilde{M}_{V,i}$, if $\overline{g}\subseteq\overline{h}$, then $\psi_{\overline{h},\overline{g}}^{V,i}\circ q_{\overline{h}} = q_{\overline{g}}$
\item[(A4)] If $V_{1}$ and $V_{2}$ are special subvarieties and if, then for every triple of indices $i,j,k$ and every special subvariety $W$, we have that $\widetilde{M}_{W,k} \subseteq \widetilde{M}_{V_{1},i}\cap\widetilde{M}_{V_{2},j}$ if and only if $\widehat{S}(\mathbb{C})_{W,k} \subseteq \widehat{S}(\mathbb{C})_{V_{1},i}\cap\widehat{S}(\mathbb{C})_{V_{2},j}$. Furthermore, $ \widetilde{M}_{V_{1},i}\cap\widetilde{M}_{V_{2},j}\subseteq\widetilde{M}_{W_{1},k_{1}}\cup\cdots\cup\widetilde{M}_{W_{r},k_{r}}$ if and only if $ \widehat{S}(\mathbb{C})_{V_{1},i}\cap\widehat{S}(\mathbb{C})_{V_{2},j}\subseteq\widehat{S}(\mathbb{C})_{W_{1},k_{1}}\cup\cdots\cup\widehat{S}(\mathbb{C})_{W_{r},k_{r}}$.
\item[(A5)] Let $V$ and $W$ be special subvarieties and fix a coordinate projection. Then for every pair of indices $i,j$ we have that $\widetilde{M}_{W,j}$ is contained in the the corresponding coordinate projection of $\widetilde{M}_{V,i}$ if and only if $\widehat{S}(\mathbb{C})_{W,j}$ is contained in the corresponding coordinate projection of $\widehat{S}(\mathbb{C})_{V,i}$. Furthermore, a coordinate projection of $\widetilde{M}_{V,i}$ is contained in $\widetilde{M}_{W,j}$, if and only if the corresponding coordinate projection of $\widehat{S}(\mathbb{C})_{V,i}$ is contained in $\widehat{S}(\mathbb{C})_{W,j}$.
\item[(A6)] If $z\in S(\mathbb{C})$ is a special point, then for every $i$, $\widetilde{M}_{z,i}$ is a singleton.
\end{description}

Let $\widetilde{T}'(\widehat{\mathbf{p}})$ be the first-order theory axiomatised by (A1)--(A6). Let $\widetilde{T}'(\widetilde{\mathbf{p}})$ be the first-order theory axiomatised by similar axioms, only changing every appearance of $\widehat{S}(\mathbb{C})$ by $\widetilde{S}(\mathbb{C})$.

\begin{lem}
$\widehat{\mathbf{p}}\models\widetilde{T}'(\widehat{\mathbf{p}})$, so $\widetilde{T}'(\widehat{\mathbf{p}})\subseteq \widetilde{T}(\widehat{\mathbf{p}})$. Similarly, $\widetilde{T}'(\widetilde{\mathbf{p}})\subseteq\widetilde{T}(\widetilde{\mathbf{p}})$.
\end{lem}
\begin{proof}
Clearly $\widehat{\mathbf{p}}$ satisfies (A1). For (A2) first observe that by definition, any $x\in\widehat{S}(\mathbb{C})_{V,i}$ and any tuple $\overline{g}$ satisfy $\widehat{p}_{\overline{g}}(x) \in Z_{\overline{g}}^{V,i}$. Also, any $z\in Z_{\overline{g}}^{V,i}$ is part of some compatible system $\left(z_{\overline{g}}\right)_{\overline{g}}\in \widehat{S}(\mathbb{C})$. (A3) is satisfied by the definition of $\widehat{S}(\mathbb{C})_{V,i}$ and the compatibility condition (\ref{eq:compat}). (A4) and (A5) are immediate. Finally, as we discussed earlier when we defined $\widetilde{T}(\widetilde{\mathbf{p}})$, (A6) is satisfied. 
\end{proof}

\begin{prop}
\label{prop:qe2}
$\widetilde{T}'(\widehat{\mathbf{p}})$ and $\widetilde{T}'(\widetilde{\mathbf{p}})$ are complete and have quantifier elimination. Therefore $\widetilde{T}(\widehat{\mathbf{p}}) = \widetilde{T}'(\widehat{\mathbf{p}})$ and $\widetilde{T}'(\widetilde{\mathbf{p}}) = \widetilde{T}(\widetilde{\mathbf{p}})$.
\end{prop}
\begin{proof}
Let $\left<\widetilde{M}, M, \left\{q_{\overline{g}}\right\}\right>$ and $\left<\widetilde{M}', M', \left\{q'_{\overline{g}}\right\}\right>$ be two $\omega$-saturated models of $\widetilde{T}'$. Let $x_{1},\ldots,x_{m}\in\widetilde{M}$ and $x_{1}',\ldots,x_{m}'\in\widetilde{M}'$ be such that $\mathrm{qftp}(x_{1},\ldots,x_{m}) = \mathrm{qftp}(x_{1}',\ldots,x_{m}')$. Choose $y\in\widetilde{M}$ and let $\mathrm{spcl}(x_{1},\ldots,x_{m},y) = \widetilde{M}_{V,i}$. Let $L$ be the field generated over $F_{0}$ by the coordinates of $q(x_{1}),\ldots,q(x_{m})$. Given a tuple $\overline{g}$, the type $\mathrm{qftp}(q_{\overline{g}}(x_{1},\ldots,x_{m},y)/L)$ is determined by the smallest algebraic subvariety defined over $L$ which contains the point $q_{\overline{g}}(x_{1},\ldots,x_{m},y)$, call it $W$. Clearly $W\subseteq Z_{\overline{g}}^{V,i}$. 

If $M = S(F)$ and $M' = S\left(F'\right)$, then there is an embedding $\sigma:L\hookrightarrow F'$ defined by the finite partial isomorphism. Let $\mathrm{pr}$ denote the coordinate projection sending $(x_{1}\ldots,x_{m},y)\mapsto(x_{1},\ldots,x_{m})$. We will use the same name to denote the corresponding coordinate projections 
\begin{equation*}
    \mathrm{pr}:Z_{(g_{1},\ldots,g_{n})}^{n(m+1)}\rightarrow Z_{(g_{1},\ldots,g_{n})}^{nm}.
\end{equation*}

Observe that the argument we will use now will be slightly more subtle than the one used in Proposition \ref{prop:qetp} because although Remark \ref{rem:proj} still holds in $\widetilde{\mathbf{p}}$, it does not hold in $\widehat{\mathbf{p}}$. But we can still get through with the same strategy. We can write $\mathrm{pr}(W) = W_{1}\cup\cdots\cup W_{t}$, with each member of the union defined over $L$. The projection $\mathrm{pr}\left(\widetilde{M}_{V,i}\right)$ will equal a finite union $\widetilde{M}_{U_{1},j_{1}}\cup\cdots\cup\widetilde{M}_{U_{r},j_{r}}$ of special domains in the case of $\widetilde{T}'(\widetilde{\mathbf{p}})$, but in the case of $\widetilde{T}'(\widehat{\mathbf{p}})$ we can only say that $\mathrm{pr}\left(\widetilde{M}_{V,i}\right)$ contains this finite union. What is still true in both theories is the formula:
\begin{equation}
\label{eq:formqe}
    \forall\overline{v}\left[\left(\overline{v}\in\bigcup_{s=1}^{r}\widetilde{M}_{U_{s},j_{s}}\wedge q_{\overline{g}}(\overline{v})\in\bigcup_{d=1}^{t} W_{d}\right)\rightarrow\exists u\left((\overline{v},u)\in\widetilde{M}_{V,i}\wedge q_{\overline{g}}(\overline{v},u)\in W\right)\right].
\end{equation}

By (A2), (A5) and $\omega$-saturation, we can choose $y'\in \widetilde{M}'$ so that $\left(x_{1}',\ldots,x_{m}',y'\right)\in\widetilde{M'}_{V,i}$,  $q_{\overline{g}}\left(x_{1}',\ldots,x'_{m},y'\right)\in W(F')$ and $q_{\overline{g}}\left(x_{1}',\ldots,x'_{m},y'\right)$ does not belong to a smaller algebraic subset defined over $L$. Indeed, by (A5) we know that the projection of $\widetilde{M}_{V,i}$ onto the first $m$ coordinates contains $\mathrm{spcl}(x_{1},\ldots,x_{m})$. By (A3) we know that $\widetilde{M}_{V,i}'$ is mapped surjectively onto $Z_{\overline{g}}^{V,i}\left(F'\right)$. Thus, using (\ref{eq:formqe}) we can find $y'$ so that $\left(x_{1}',\ldots,x_{m}',y'\right)\in\widetilde{M}_{V,i}'$ and 
\begin{equation*}
\mathrm{qfpt}\left(q_{\overline{g}}\left(x_{1},\ldots,x_{m},y\right)/L\right) = \mathrm{qfpt}\left(q_{\overline{g}}\left(x_{1}',\ldots,x_{m}',y'\right)/L\right)
\end{equation*}

By $\omega$-saturation we can now find $y''$ so that 
\begin{equation*}
\bigcup_{\overline{g}}\mathrm{qfpt}\left(q_{\overline{g}}\left(x_{1},\ldots,x_{m},y\right)/L\right) = \bigcup_{\overline{g}}\mathrm{qftp}\left(q_{\overline{g}}\left(x_{1}',\ldots,x_{m}',y''\right)/L\right)
\end{equation*}
and, as $\widetilde{M}_{V,i}$ is no covered by finitely many proper special subdomains, $\omega$-saturation also gives us that $\mathrm{spcl}(x_{1}',\ldots,x_{m}',y'') = \widetilde{M}_{V,i}'$. 
\end{proof}

Now given a model $S(F)\models T$, we can obtain two covering structures for it: $\widehat{S}(F)$ which is defined as the inverse limit the varieties $Z_{\overline{g}}$, and $\widetilde{S}(F)$, defined to be the subset consisting of standard points of $\widehat{S}(F)$. The maps $\widehat{q}_{\overline{g}}:\widehat{S}(F)\rightarrow Z_{\overline{g}}$ are the natural maps, and $\widetilde{q}_{\overline{g}}$ is the restriction of $\widehat{q}_{\overline{g}}$ to $\widetilde{S}(F)$.

\begin{cor}
\label{cor:promodels}
$\left<\widehat{S}(F),S(F),\left\{\widehat{q}_{\overline{g}}\right\}\right>\models\widetilde{T}(\widehat{\mathbf{p}})$, and $\left<\widetilde{S}(F),S(F),\left\{\widetilde{q}_{\overline{g}}\right\}\right>\models\widetilde{T}(\widetilde{\mathbf{p}})$. 
\end{cor}

The following Corollary is proven just like Proposition \ref{prop:typedet1}.

\begin{cor}
\label{cor:typedet2}
Let $\widetilde{\mathbf{q}}\models\widehat{T}(\widetilde{\mathbf{p}})$ or $\widetilde{T}(\widetilde{\mathbf{p}})$. Choose $x_{1},\ldots,x_{m}\in\widetilde{M}$ and $A\subseteq\widetilde{M}$. Then we have that $\mathrm{qftp}\left(x_{1},\ldots,x_{m}/A\right)$ is determined by:
\begin{equation*}
Q(\overline{v}):=\bigcup_{\overline{g}}\mathrm{qfpt}\left(q_{\overline{g}}(x_{1},\ldots,x_{m})/L\right) \cup \left\{\mathrm{sploc}((v_{1},\ldots,v_{m})/A) = \mathrm{sploc}((x_{1},\ldots,x_{m})/A)\right\},
\end{equation*}
where $L$ is the field generated over $F_{0}$ by the elements of $\widehat{A}$. 
\end{cor}

\begin{prop}
\label{prop:x+}
$\left<X^{+},S(\mathbb{C}), \left\{p_{\overline{g}}\right\}\right>\models\widetilde{T}(\widetilde{\mathbf{p}})$, where the structure of $X^{+}$ consists only of the names of the special domains $X^{+}_{V,i}$, and we define $p_{\overline{g}}(x) := p(\overline{g}x)$, for all $x\in X^{+}$ and all tuples $\overline{g}$ from $G^{+}(\mathbb{Q})^{+}$.
\end{prop}
\begin{proof}
By Proposition \ref{prop:qe2}, we need to verify that $\left<X^{+},S(\mathbb{C}),\left\{p_{\overline{g}}\right\}\right>$ satisfies (A1)--(A6). (A1), (A2), (A3) and (A6) are straightforward. Observe that the axiom scheme (A2) is essentially a combination of the axiom schemes SS and MOD from \textsection \ref{subsec:axiom}.

To verify (A4) and (A5), first recall the map $\iota$ from Lemma \ref{lem:embedding}, and notice that its image is completely contained in $\widetilde{S}(\mathbb{C})$. From the definition of special domains in $\widetilde{S}(\mathbb{C})$, we get then that $x\in X_{V,i}^{+}$ if and only if $\iota(x)\in\widetilde{S}(\mathbb{C})_{V,i}$. This takes care of some of the implications needed to show (A4) and (A5). For example, if $\widetilde{S}(\mathbb{C})_{W,k}\subseteq\widetilde{S}(\mathbb{C})_{V_{1},i}\cap\widetilde{S}(\mathbb{C})_{V_{2},j}$, then we get that $X^{+}_{W,k}\subseteq X^{+}_{V_{1},i}\cap X^{+}_{V_{2},j}$. Conversely, if $X^{+}_{W,k}\subseteq X^{+}_{V_{1},i}\cap X^{+}_{V_{2},j}$, then $Z_{\overline{g}}^{W,k}\subseteq Z_{\overline{g}}^{V_{1},i}\cap Z_{\overline{g}}^{V_{2},j}$ for all tuples $\overline{g}$. Similarly, $\widetilde{S}(\mathbb{C})_{V,i}\subseteq\mathrm{pr}\left(\widetilde{S}_{W,j}\right)$ if and only if $X^{+}_{V,i}\subseteq\mathrm{pr}\left(X^{+}_{W,j}\right)$.

Suppose now that $\widetilde{S}(\mathbb{C})_{V_{1},i}\cap\widetilde{S}(\mathbb{C})_{V_{2},j}\subseteq\widetilde{S}(\mathbb{C})_{W_{1},k_{1}}\cup\cdots\cup\widetilde{S}(\mathbb{C})_{W_{r},k_{r}}$. As before, we get that $X^{V_{1},i}\cap X^{+}_{V_{2},j}\subseteq X^{+}_{W_{1},k_{1}}\cup\cdots\cup X^{+}_{W_{r},k_{r}}$. Conversely, suppose $X^{V_{1},i}\cap X^{+}_{V_{2},j}\subseteq X^{+}_{W_{1},k_{1}}\cup\cdots\cup X^{+}_{W_{r},k_{r}}$. If the corresponding statement is not true in $\widetilde{S}(\mathbb{C})$, then there must be a point $\widetilde{x}\in\widetilde{S}(\mathbb{C})_{V_{1},i}\cap\widetilde{S}(\mathbb{C})_{V_{2},j}$ such that $\widetilde{x}\notin\widetilde{S}(\mathbb{C})_{W_{1},k_{1}}\cup\cdots\cup\widetilde{S}(\mathbb{C})_{W_{r},k_{r}}$. Choose $y\in X^{+}_{V_{1},i}\cap X^{+}_{V_{2},j}$ such that $p(y) = \widetilde{p}(\widetilde{x})$. By Lemma \ref{lem:standardness}, we get that $\widetilde{x}$ is non-standard, but $\widetilde{S}(\mathbb{C})$ by definition only has standard points. This contradiction shows that (A4) is satisfied. 

To verify what remains of (A5), observe that $\mathrm{pr}\left(X^{+}_{W,j}\right)\subseteq X^{+}_{V,i}$ implies that the corresponding projection of $Z_{\overline{g}}^{W,j}$ is contained in $Z_{\overline{g}}^{V,i}$, for all $\overline{g}$. On the other hand, it is immediate that if $\mathrm{pr}\left(\widetilde{S}(\mathbb{C})_{W,j}\right)\subseteq\widetilde{S}(\mathbb{C})_{V,i}$, then $\mathrm{pr}\left(X^{+}_{W,j}\right)\subseteq X^{+}_{V,i}$.
\end{proof}

\begin{cor}
\label{cor:gact}
If $\widetilde{\mathbf{q}}\models\widetilde{T}(\widetilde{\mathbf{p}})$, then $\widetilde{M}$ carries an action of $G^{\mathrm{ad}}(\mathbb{Q})^{+}$. 
\end{cor}
\begin{proof}
Every element $g\in G^{\mathrm{ad}}(\mathbb{Q})^{+}$ defines a special subvariety $Z_{(e,g)}$, and it has corresponding special domains $\widetilde{M}_{Z_{(e,g)},i}$. Because $\left<X^{+},S(\mathbb{C}),\left\{p_{\overline{g}}\right\}\right>\models\widetilde{T}(\widetilde{\mathbf{p}})$ by Proposition \ref{prop:x+}, then one of the special domains $X^{+}_{Z_{(e,g)},i}$ corresponds to the graph of the action of $g$ on $X^{+}$, say that it is $X^{+}_{Z_{(e,g)},1}$. And as $\widetilde{T}$ is complete (by Proposition \ref{prop:qe2}), then the formula that says that $\widetilde{M}_{Z_{(e,g)},1}$ is the graph of a function, and that this function is a bijection from $\widetilde{M}$ to itself, must hold in all models of $\widetilde{T}(\widetilde{\mathbf{p}})$. And so we can define
\begin{equation*}
    y=gx \iff (x,y)\in\widetilde{M}_{Z_{(e,g)},1}.
\end{equation*}
\end{proof}

From this Corollary we see that, even if $\widetilde{T}(\widetilde{\mathbf{p}})$ is constructed from a Shimura variety in a more natural way than $\mathrm{Th}(\mathbf{p})$, the difference in the end is only cosmetic. 

\begin{prop}[Minimal models]
\label{prop:min}
Suppose $\widetilde{\mathbf{q}}\models\widetilde{T}(\widetilde{\mathbf{p}})$, and let $B\subseteq M$. Let $S(F)\models T$ be prime over $B$. Then there is $\widetilde{\mathbf{q}}'\models\widetilde{T}(\widetilde{\mathbf{p}})$ such that $\widetilde{\mathbf{q}}'\prec\widetilde{\mathbf{q}}$, $M'=S(F)$, and $\widetilde{\mathbf{q}}'$ is minimal over $q^{-1}(B)\cup S(F)$.
\end{prop}
\begin{proof}
Choose a section of $q$ which only contains standard points (which exist by (A2)), call it $A$. Let $\widetilde{M}'$ consist of $q^{-1}(B)$ and the subset of $\widetilde{M}$ of points which are constructible over $A$. Let $q'_{\overline{g}}:\widetilde{M}'\rightarrow S(F)^{n}$ be the restriction of $q_{\overline{g}}$ to $\widetilde{M}'$. By Corollary \ref{cor:gact}, we get that the set of constructible points over $A$ is the union of the orbits of points in $A$ under the action of $G^{\mathrm{ad}}(\mathbb{Q})^{+}$.
\end{proof}

\begin{cor}
Let $\widetilde{\mathbf{q}}\models\widetilde{T}(\widetilde{\mathbf{p}})$ and suppose $M'\prec M$, where $M$ is the variety sort of $\widetilde{\mathbf{q}}$. Set $\widetilde{M}' = q^{-1}\left(M'\right)$ and let $q_{\overline{g}}'$ be the restriction of $q_{\overline{g}}$ to $\widetilde{M}'$. Then $\left<\widetilde{M}',M', \left\{q_{\overline{g}}'\right\}\right>\models\widetilde{T}(\widetilde{\mathbf{p}})$.
\end{cor}

\begin{remark}
\label{rem:iota}
For $\widetilde{\mathbf{q}}\models\widetilde{T}(\widetilde{\mathbf{p}})$ with $M=S(F)$, we can define $\iota:\widetilde{M}\rightarrow\widehat{S}(F)$ in the following way: $\iota(x) = (z_{\overline{g}})_{\overline{g}}$ if and only if for all tuples $\overline{g}$, we have that $q_{\overline{g}}(x) = z_{\overline{g}}$. 

If $x,y\in\widetilde{M}$ are distinct and such that $x = \gamma y$ for some $\gamma\in\Gamma$, then by Lemma \ref{lem:nofix} this is the only possible element of $\Gamma$ sending $y$ to $x$ (notice that for each $\gamma\in\Gamma$, the statement of Lemma \ref{lem:nofix} can be expressed by a first-order sentence). By Lemma \ref{lem:centre}, there is a tuple $\overline{g}$ such that $\gamma\notin\Gamma_{\overline{g}}$ (as $x\neq y$, then $\gamma$ cannot be central). Now, the following sentence is in $\mathrm{Th}(\mathbf{p})$:
\begin{equation*}
    \forall x,y\in D\left((x=\gamma y)\rightarrow q\left(\overline{g}x\right)\neq q\left(\overline{g} y\right)\right).
\end{equation*}
This shows that $\iota(x)\neq\iota(y)$. In particular, if $\widetilde{\mathbf{q}}$ satisfies SF, then $\iota$ is injective. Furthermore, this shows that if $\mathbf{q}\models\mathrm{Th}_{\mathrm{SF}}(\mathbf{p})$, then $\iota: D\rightarrow\widehat{S}(F)$ is not only injective, but its image is contained in $\widetilde{S}(F)$
\end{remark}

\subsection{Stability Aspects}
We will now discuss some results about the stability properties of $\widetilde{T}(\widehat{\mathbf{p}})$. We remark that the same results and proofs hold for $\widetilde{T}(\widetilde{\mathbf{p}})$. We start with a result which follows immediately from quantifier elimination.

\begin{cor}
\label{cor:dnf}
Let $\widetilde{\mathbf{q}}\models \widetilde{T}(\widehat{\mathbf{p}})$ and let $B\subseteq \widetilde{M}$. Suppose $A\subseteq \widetilde{M}^{n}$ is definable over $B$. Then there is a tuple $\overline{g}$ of $G^{\mathrm{ad}}(\mathbb{Q})^{+}$, special domains $\widetilde{M}_{V_{j},i_{j}}$, tuples $\overline{b}^{j}$ of $B$, and $\emptyset\neq Y_{j}\subseteq V_{j}\left(q_{\overline{g}}\left(b^{j}\right)\right)$, with $j$ ranging on a finite set, and with each $Y_{j}$ $T$-definable over $q_{\overline{g}}(B)$, such that:
\begin{equation*}
\bigcup_{j}\left(\widetilde{M}_{V_{j},i_{j}}\left(\overline{b}^{j}\right)\cap q_{\overline{g}}^{-1}\left(Y_{j}\right)\right)\subseteq A\subseteq\bigcup_{j}\widetilde{M}_{V_{j},i_{j}}\left(\overline{b}^{j}\right).
\end{equation*}
\end{cor}
\begin{proof}
By Proposition \ref{prop:qe2}, we know that $\widetilde{T}(\widehat{\mathbf{p}})$ has quantifier elimination. Using the disjunctive normal form of quantifier-free formulas, the fact that for two distinct special domains in the power of $\widetilde{M}$ either one is contained in the other or they are disjoint, and (A2), we get the desired result. 
\end{proof}

\begin{prop}
\label{prop:superstable}
$\widetilde{T}(\widehat{\mathbf{p}})$ is superstable.
\end{prop}
\begin{proof}
We use the classification of the stability spectrum of complete countable theories (\cite[Theorem 8.6.5]{tent-ziegler}). By quantifier elimination and Corollary \ref{cor:typedet2}, $\mathrm{tp}(x/A)$ is determined by:
\begin{equation*}
\bigcup_{\overline{g}}\mathrm{qftp}(q(\overline{g}x)/\overline{g}A)\cup\left\{\mathrm{sploc}(v/A)=\mathrm{sploc}(x/A)\right\}.
\end{equation*}
Now $\mathrm{qftp}(q_{\overline{g}(x)}/q_{\overline{g}}(A))$ is determined by the smallest algebraic subvariety defined over $F_{0}\left(q_{\overline{g}}(A)\right)$. This will be a subvariety of $Z_{\overline{g}}$. The fibres of the maps $Z_{\overline{h}}\rightarrow Z_{\overline{g}}$ are finite, so once we have a realisation of $\mathrm{qftp}\left(q_{\overline{g}}(x)/q_{\overline{g}}(A)\right)$, there are only finitely many possibilities for $\mathrm{qftp}\left(q_{\overline{h}}(x)/q_{\overline{g}}(A)\right)$. On the other hand, given a special subvariety $V$, there are at most countably many corresponding special domains $\widetilde{M}_{V,i}$. Therefore, if $A\subseteq \widetilde{M}$ is such that $|A|\geq 2^{\aleph_{0}}$, then the space of $1$-types over $A$ has cardinality at most $|A|$.  
\end{proof}

\begin{prop}
\label{prop:forking}
Given $\widetilde{\mathbf{q}}\models\widetilde{T}(\widehat{\mathbf{p}})$, $x\in\widetilde{M}$ and $B\subseteq\widetilde{M}$. Then $\mathrm{tp}(x/B)$ forks over $A\subseteq B$ if and only if $\mathrm{tp}\left(q(x)/\widehat{B}\right)$ forks over $\widehat{A}$ or $\mathrm{sploc}(x/B)$ is not definable over $A$.
\end{prop}
\begin{proof}
Suppose first that $\mathrm{tp}(x/B)$ forks over $A$, say that $\varphi\left(v,\overline{b}\right)\in\mathrm{tp}(x/B)$ divides over $A$. Let $C=\mathrm{sploc}(x/B)$. We can assume that $\varphi\left(v/\overline{b}\right)\models v\in C$ (recall that if $\varphi$ divides over $A$ and $\varphi$ and $\psi$ are consistent, then $\varphi\wedge\psi$ divides over $A$, because $\varphi\wedge\psi\rightarrow\varphi$). Suppose that $C$ is defined over $A$. Then for any $\overline{b}'\equiv_{A}\overline{b}$ we have $\varphi\left(v,\overline{b}'\right)\models v\in C$. By Corollary \ref{cor:dnf}, $\varphi\left(v,\overline{b}\right)$ is implied by a formula in $\mathrm{tp}(x/B)$ of the form
\begin{equation*}
v\in C\wedge\psi\left(q_{\overline{g}}(v)\right),
\end{equation*}
where $\psi$ is a formula in $\mathcal{L}_{\mathrm{cov}}(B)$ implying $q_{\overline{g}}(v)\in q_{\overline{g}}(C)$. As $\varphi$ divides over $A$, $\psi$ must divide over $A$ (we are assuming $C$ is defined over $A$, so $v\in C$ does not fork over $A$), which means that $\mathrm{tp}\left(q_{\overline{g}}(x)/\widehat{B}\right)$ divides over $\widehat{A}$. Since $q(x)$ is algebraic over $q_{\overline{g}}(x)$, $\mathrm{tp}\left(q(x)/\widehat{B}\right)$ forks over $A$. 

Conversely, suppose $C$ is not defined over $A$. Then $v\in C$ forks over $A$.
\end{proof}

\begin{lem}
$\widetilde{T}(\widehat{\mathbf{p}})$ has finite $U$-rank. 
\end{lem}
\begin{proof}
We know that $T$ has finite Morley rank, which bounds the length of possible chains of types in the language $\mathcal{L}_{F}$. On the other hand, the $\mathrm{sploc}$ is defined by finitely many parameters. By Proposition \ref{prop:forking}, this shows that $\widetilde{T}(\widehat{\mathbf{p}})$ has finite $\mathrm{SU}$-rank, and as $\mathrm{SU}$-rank and $U$-rank coincide in stable theories (see  \cite[Exercise 8.6.2]{tent-ziegler}), by Proposition \ref{prop:superstable} we are done.
\end{proof}

\section{Necessary Conditions for Categoricity}
\label{subsec:necessary}
We now come to the question of categoricity of the theory of a Shimura variety. Let $p:X^{+}\rightarrow S$ be fixed throughout this entire section. It is easy to see that $\mathrm{Th}(\mathbf{p})$,  $\widetilde{T}(\widehat{\mathbf{p}})$ and $\widetilde{T}(\widetilde{\mathbf{p}})$ are not categorical because there is no restriction on the size of the fibres of $p$, $\widehat{p}$ and $\widetilde{p}$, respectively. Also, from the properties of $\mathrm{ACF}_{0}$ we know that if $S(F)$ is a countable model of $T$, then it determined up to isomorphism by the transcendence degree of $F$ over $\mathbb{Q}$.

One way to address this is to add two (infinitely long) conditions to these theories. To exhibit how this would work, let us add the conditions to $\mathrm{Th}(\mathbf{p})$. Let $\mathrm{Th}^{\infty}_{\mathrm{SF}}(\mathbf{p})$ \index{t@$\mathrm{Th}^{\infty}_{\mathrm{SF}}(\mathbf{p})$} be the union of $\mathrm{Th}(\mathbf{p})$, SF, and the $\mathcal{L}_{\omega_{1},\omega}$-sentence that says that the transcendence degree of $F$ is infinite. So our question now is whether $\mathrm{Th}^{\infty}_{\mathrm{SF}}(\mathbf{p})$ is categorical. 

We will show that it is necessary for categoricity (of any of the three theories being considered) that we have open image for every Galois representation attached to points obtained as explained in \textsection\ref{subsec:galrep}. To show this, one can rewrite the argument of \cite[\textsection 4.1]{daw-harris}. Even though the language used in \cite{daw-harris} is not the same as the one we have used, the proof can be adapted easily. Instead of doing that, will show the same argument of \cite[Theorem 4.1]{daw-harris}, but in a different way. We start by recalling a theorem of Keisler, which is central to all the known categoricity results regarding arithmetic varieties.

\begin{thm}[see {{\cite[Corollary 5.6]{keisler}}}]
\label{thm:keisler}
If an $\mathcal{L}_{\omega_{1},\omega}$ sentence is $\aleph_{1}$-categorical, then the set of $m$-types over the empty set in models of this sentence is at most countable.
\end{thm}

Consider the inverse system over $S(\mathbb{C})$ of all those special subvarieties of the form $Z_{\overline{g}}$, where $\overline{g}$ is chosen so that $\Gamma_{\overline{g}}$ is normal in $\Gamma$, with the corresponding maps $\psi_{\overline{g},e}:Z_{\overline{g}}\rightarrow S(\mathbb{C})$. For every point $z\in S(\mathbb{C})$, the fibre over $z$ of $\psi_{\overline{g},)}$ consists of finitely many points, and it carries a simply transitive action of $\Gamma/\Gamma_{\overline{g}}$ and a compatible action of $\mathrm{Aut}(\mathbb{C}/L)$, where $L$ is a finitely generated field containing $F_{0}$ and the coordinates of $z$. If $p(x) = z$, then $\mathrm{tp}(x)$ knows in which of the $\mathrm{Aut}(\mathbb{C}/L)$-orbits in $\psi_{\overline{g},e}^{-1}(z)$ lies $p_{\overline{g}}(x)$. 

The action of $\mathrm{Aut}(\mathbb{C}/L)$ may not be transitive on $\psi_{\overline{g},e}^{-1}(z)$, so by choosing different $\mathrm{Aut}(\mathbb{C}/L)$-orbits in the fibre, we can construct new types from $\mathrm{tp}(x)$. If the action $\mathrm{Aut}(\mathbb{C}/L)$ is never transitive (or if there are infintely many tuples $\overline{g}$ for which the action is not transitive), then we would be able to define uncountable many new 1-types. Keisler's theorem would then say that none of the theories we are looking at is $\aleph_{1}$-categorical. 

However, if a Galois representation $\rho_{\widetilde{y}}:\mathrm{Aut}(\mathbb{C}/L)\rightarrow\overline{\Gamma}$ (following the notation of \textsection\ref{subsec:galgen}) has open image, then the number of tuples for which the action of $\mathrm{Aut}(\mathbb{C}/L)$ is non-transitive, is finite. 

Now, it does not make sense to expect $\rho_{\widetilde{y}}$ to have open image unless $z$ is Hodge-generic. If $\mathrm{spcl}(z) = V\subsetneq S(\mathbb{C})$, then we need to ``relativise'' the inverse system of varieties $Z_{\overline{g}}$ to $V$, meaning that we need to consider the inverse system of varieties of the form $Z_{\overline{g}}^{V,i}$ (where we choose an index $i$ and fix it), and so the sensible condition to ask in this case is that a representation $\mathrm{Aut}(\mathbb{C}/L)\rightarrow\overline{\Gamma^{V,i}}$ have open image.

For future reference, we will now specialise this result for $\mathrm{Th}^{\infty}_{\mathrm{SF}}(\mathbf{p})$ in the following Corollary, but it should be clear that the argument holds for $\widetilde{T}(\widehat{\mathbf{p}})$ and $\widetilde{T}(\widetilde{\mathbf{p}})$ as well.

\begin{cor}
\label{cor:fic}
Let $p:X^{+}\rightarrow S(\mathbb{C})$ be a Shimura variety. Let $\overline{z}\in S(\mathbb{C})^{n}$ and let $V$ be its special closure. If $\mathrm{Th}^{\infty}_{\mathrm{SF}}(\mathbf{p})$ is $\aleph_{1}$-categorical, then the image of a homomorphism
\begin{equation*}
    \mathrm{Aut}(\mathbb{C}/L)\rightarrow\overline{\Gamma^{V,i}}
\end{equation*}
associated with $\overline{z}$ has finite index (for all $i$).
\end{cor}

\section{Categoricity Via Quasiminimality}
\label{sec:quasi}
We would like to show that, under the conditions of Corollary \ref{cor:fic}, $\mathrm{Th}_{\mathrm{SF}}^{\infty}(\mathbf{p})$ is categorical. Now, the Shimura structure $\mathbf{p}$ is not quasiminimal when the dimension of S is at least 2 (even if we restrict to the class of models satisfying SF) as it does not satisfy (QM4), and therefore, we cannot apply Theorem \ref{thm:quasi} straightaway. Instead, we need to first pass to a sufficiently general curve, as is explained below. If we are working with a Shimura curve, then we can use the quasiminimal method straightaway as is shown in \cite{daw-harris}.

The strategy is the following. Given an $\mathcal{L}_{F}$-definable curve $C$ in $S$, consider the corresponding two-sorted structure $\left<p^{-1}(C(\mathbb{C})), C(\mathbb{C}), \left\{p^{C}_{\overline{g}}\right\}\right>$, where $p^{C}_{\overline{g}}$ is the restriction of $p_{\overline{g}}$ to $p^{-1}(C(\mathbb{C}))$. Suppose that $\left<X^{+}, S(\mathbb{C}), \left\{p_{\overline{g}}\right\}\right>$ is $\widetilde{\mathcal{L}}$-definable $\left<p^{-1}(C(\mathbb{C})), C(\mathbb{C}), \left\{p^{C}_{\overline{g}}\right\}\right>$. Then if we can prove categoricity of the latter structure, we will obtain categoricity of $p:X^{+}\rightarrow S(\mathbb{C})$.

\subsection{Hecke-generic Curves}
\label{subsec:heckegen}
\begin{defi}
Let $\overline{g}$ be a tuple from $G^{\mathrm{ad}}(\mathbb{Q}){+}$ such that $\Gamma_{\overline{g}}$ is normal in $\Gamma$. We say that an irreducible subvariety $C\subseteq S$ is $\overline{g}$-\emph{Hecke-generic} \index{Hecke-generic} if $\psi_{\overline{g},(e)}^{-1}(C)$ is irreducible in $Z_{\overline{g}}$. We say that $C$ is \emph{Hecke-generic} if $\psi_{\overline{g},(e)}^{-1}(C)$ is irreducible in $Z_{\overline{g}}$ whenever $\Gamma_{\overline{g}}$ is normal in $\Gamma$ (i.e. it is $\overline{g}$-Hecke-generic for all $\overline{g}$).
\end{defi}

For the existence of such curves, we use the following.

\begin{lem}
Let $C\subseteq S$ be an irreducible subvariety defined over $F_{0}$. Let $b\in C(\mathbb{C})$ be a generic point over $\overline{\mathbb{Q}}$. Let $\overline{g}$ be such that $\Gamma_{\overline{g}}$ is normal in $\Gamma$. Then $C$ is $\overline{g}$-Hecke-generic if and only if some (all) corresponding homomorphism $\mathrm{Aut}(\mathbb{C}/\overline{\mathbb{Q}}(b))\rightarrow\Gamma/\Gamma_{\overline{g}}$ is surjective.
\end{lem}
\begin{proof}
The map $\psi_{\overline{g},(e)}:Z_{\overline{g}}\rightarrow S(\mathbb{C})$, being finite, is closed. So for some irreducible component $Y\subseteq\psi_{\overline{g},(e)}^{-1}(C)$ we must have that $\psi_{\overline{g},(e)}(Y) = C$. If $Y'$ is another irreducible component, then $Y = \gamma Y'$ for some $\gamma\in\Gamma/\Gamma_{\overline{g}}$ ($\Gamma$ acts transitively in analytic components, cf \cite[Theorem 4.1]{ullmo-yafaev2}). Choose $b_{\overline{g}}\in Y$ such that $\psi_{\overline{g},(e)}\left(b_{\overline{g}}\right) = b$. Thus $b_{\overline{g}}$ is generic in $Y$ over $L$. Therefore, by considering the action of $\mathrm{Aut}(\mathbb{C}/F_{0}(b))$ on $b_{\overline{g}}$, the homomorphism $\mathrm{Aut}\left(\mathbb{C}/\overline{Q}(b)\right)\rightarrow\Gamma/\Gamma_{\overline{g}}$ is surjective if and only if $Y = \gamma Y'$ for all $\gamma\in\Gamma/\Gamma_{\overline{g}}$. In other words, $\psi_{\overline{g},(e)}^{-1}(C)$ is irreducible.
\end{proof}

So, to show the existence of a curve $C\subseteq S$ which is Hecke-generic, it suffices to show that $S(\mathbb{C})$ has a point of transcendence degree 1 which is Galois generic. For this we refer the reader to the proof of the main theorem of \cite{baldi}.



Let $C\subseteq S$ be a Hecke-generic curve defined over $F_{0}$. Let $X^{+}_{C} := p^{-1}(C(\mathbb{C}))$. Consider the two-sorted structure $\mathbf{r}:=\left<X^{+}_{C},C(\mathbb{C}),\left\{r_{\overline{g}}\right\}_{\overline{g}}\right>$, where the structure on $X^{+}_{C}$ is inherited from the $\mathcal{L}_{\mathrm{cov}}$-structure on $X^{+}$, the structure on $C(\mathbb{C})$ is inherited from the $\mathcal{L}_{F}$-structure on $S(\mathbb{C})$, and the maps $r_{\overline{g}}$ are just the restrictions of $p_{\overline{g}}$ to $X^{+}_{C}$. Thus $\mathbf{r}$ is a definable (without parameters) $\widetilde{\mathcal{L}}$-substructure of $\mathbf{p}$. In particular, $X^{+}_{C}$ has a $\Gamma$-action. By analogy, given $\mathbf{q}\models\widetilde{T}$ or $\mathrm{Th}(\mathbf{p})$, let $D_{C} = q^{-1}(C(F))$, and denote by $\mathbf{r}_{\mathbf{q}}$ the two-sorted structure $\left<D_{C}, C(F), \left\{r_{\overline{g}}\right\}_{\overline{g}}\right>$.

\begin{remark}
Let $\mathrm{Th}(\mathbf{r})$ be the complete first-order theory of $\mathbf{r}$ in the language $\widetilde{\mathcal{L}}$. Given $\mathbf{q}\models\widetilde{T}$, then the corresponding $\widetilde{\mathcal{L}}$-structure $\mathbf{r}_{\mathbf{q}}$ is a model of $\mathrm{Th}(\mathbf{r})$. Also, $\mathrm{Th}(\mathbf{r})$ has quantifier elimination because $\widetilde{T}$ has quantifier elimination and $\mathbf{r}$ is a $\emptyset$-definable substructure of a model of $\widetilde{T}$. Just like in Corollary \ref{cor:gact}, the models of $\mathrm{Th}(\mathbf{r})$ have a $\Gamma$-action. So it makes sense to consider the class $\mathrm{Th}_{\mathrm{SF}}^{\infty}(\mathbf{r})$ of models which satisfy SF and the field $F$ has infinite transcendence degree over $\mathbb{Q}$. 
\end{remark}

\begin{lem}
\label{lem:def}
Let $\mathbf{q}\models\mathrm{Th}(\mathbf{p})$. Then $\mathbf{r}_{\mathbf{q}}$ determines $\mathbf{q}$ i.e. $\mathbf{q} = \mathrm{dcl}\left(D_{C}\cup C(F)\right)$, where $\mathrm{dcl}$ denotes the definable closure operator in $\mathbf{q}$. 
\end{lem}
\begin{proof}
The domain sorts $C(F)$ and $S(F)$ are both bi-interpretable with the field sort $F$ (with the same constants), so we know that at the level of the variety structures, $S(F)$ is definable over $C(F)$. 

Let $x\in D$. As $S(F)$ is definable, then the fibre $q^{-1}(q(x))$ is definable. Now consider $\mathrm{sploc}\left(x/D_{C}\right)$, which we know is definable by Lemma \ref{lem:sploc}. It remains to show that for all non-central $\gamma\in\Gamma$, $\gamma x\notin\mathrm{sploc}\left(x/D_{C}\right)$. By Remark \ref{rem:iota}, for every non-central $\gamma\in\Gamma$ there is $g\in G^{\mathrm{ad}}(\mathbb{Q})^{+}$ such that $q(gx)\neq q(g\gamma x)$. As $\mathrm{sploc}\left(gx/D_{C}\right) = g\mathrm{sploc}\left(x/D_{C}\right)$ by Lemma \ref{lem:sploc3} and $q\left(\mathrm{sploc}\left(gx/D_{C}\right)\right) = q(gx)$ by Lemma \ref{lem:sploc}, we conclude that $\gamma x\notin\mathrm{sploc}\left(x/D_{C}\right)$.
\end{proof}

To prove categoricity of $\mathrm{Th}^{\infty}_{\mathrm{SF}}(\mathbf{p})$ it will be enough to prove categoricity of $\mathrm{Th}^{\infty}_{\mathrm{SF}}(\mathbf{r})$. Let $\mathscr{K}_{\mathbf{r}}$ \index{k@$\mathscr{K}_{\mathbf{r}}$} be the class of models of $\mathrm{Th}^{\infty}_{\mathrm{SF}}(\mathbf{r})$. Given $\mathbf{t}=\left<\mathbf{R},\mathbf{C},t\right>\in\mathscr{K}_{\mathbf{r}}$ we define a pregeometry $\mathrm{cl}_{\mathbf{t}}$ \index{c@$\mathrm{cl}_{\mathbf{t}}$} on $\mathbf{t}$ as follows. Given a set $B\subseteq C(F)$, let $\mathrm{acl}(B)$ denote the set of all points in $C(F)$ with coordinates that are algebraic over the coordinates of the points in $B$ (this is just the model-theoretic definition of $\mathrm{acl}$). The operator $\mathrm{acl}$ is always a pregeometry in strongly minimal structures (like C). Given $A\subseteq R\cup C(F)$ we write $A = A_{R}\cup A_{C}$, where $A_{R}= (A\cap R)$ and $A_{C} = A\cap C(F)$, and then set $A_{1} = A_{R}\cup q^{-1}(A_{R})$. Then we define $\mathrm{cl}_{\mathbf{t}}(A) := \left<q^{-1}\circ\mathrm{acl}\circ q(A_{1}), \mathrm{acl}\circ q(A_{1}), q\right>$, where we take the induced structure on each sort. For example, the sort $\mathrm{acl}\circ q(A_{1})$ is simply the $\mathcal{L}_{F}$-structure $C(K)$, where $K$ is the algebraic closure of the field generated over $F_{0}$ by the coordinates of $q(A_{1})$.

As $\mathrm{acl}$ is a pregeometry, then it is straightforward to check that $\mathrm{cl}_{\mathbf{t}}$ is also a pregeometry. It is simple to adapt the axiomatisation from \textsection\ref{subsec:axiomscov} to give an axiomatisation for $\mathrm{Th}(\mathbf{r})$, and from this it follows that  $\mathrm{cl}_{\mathbf{t}}(A)\models\mathrm{Th}_{\mathrm{SF}}(\mathbf{r})$. Lastly, the dimension defined by $\mathrm{cl}_{\mathbf{t}}$ is equal to the transcendence degree of $K$. 

The strategy for proving categoricity of $\mathscr{K}_{\mathbf{r}}$ now is to first show that $\left(\mathbf{r},\mathrm{cl}_{\mathbf{r}}\right)$ satisfies the axioms (QM1)--(QM5) of a quasiminimal pregeometry structure described in \textsection \ref{subsec:quasi}. Aside from (QM5), the other axioms are straightforward to check.

\subsection{Finite-Index Conditions}
\label{subsec:fic}
Let $\mathscr{K}_{\mathbf{p}}$ \index{k@$\mathscr{K}_{\widehat{\mathbf{p}}}$} be the class of models of $\mathrm{Th}^{\infty}_{\mathrm{SF}}(\mathbf{p})$.

\begin{defi}
We say that a Shimura variety $p:X^{+}\rightarrow S(\mathbb{C})$ satisfies the \emph{first finite-index condition} (FIC1) \index{FIC1} if for any tuple $x_{1},\ldots,x_{m}\in D$ of non-special points in distinct $G^{\mathrm{ad}}(\mathbb{Q})^{+}$-orbits, the induced homomorphism 
\begin{equation*}
    \mathrm{Aut}(\mathbb{C}/L)\rightarrow\overline{\Gamma^{V,i}}
\end{equation*}
has finite index, where $D_{V,i} = \mathrm{spcl}(x_{1},\ldots,x_{m})$, and $L$ is finitely generated over $F_{0}$ by the coordinates of the $p(x_{i})$.
\end{defi}

\begin{prop}[cf. {{\cite[Lemma 4.11]{daw-harris}}}]
\label{prop:main}
Suppose $p:X^{+}\rightarrow S(\mathbb{C})$ has FIC1. Then $\mathscr{K}_{\mathbf{p}}$ satisfies $\aleph_{0}$-homogeneity over the empty set. 
\end{prop}
\begin{proof}
Suppose $\mathbf{q}:=\left<\mathbf{D},\mathbf{S},q\right>$ and $\mathbf{q}':=\left<\mathbf{D}',\mathbf{S}',q'\right>$ are two models of $\mathrm{Th}_{\mathrm{SF}}^{\infty}(\mathbf{p})$ and suppose that $x_{1}\ldots,x_{m}\in D$ is a collection of non-special points in distinct $G^{\mathrm{ad}}(\mathbb{Q})^{+}$-orbits. Suppose that $x_{1}',\ldots,x_{m}'\in D'$ is another collection of non-special points in distinct $G^{\mathrm{ad}}(\mathbb{Q})^{+}$-orbits, such that 
\begin{equation*}
    \mathrm{qftp}\left(x_{1},\ldots,x_{m}\right) = \mathrm{qftp}\left(x_{1}',\ldots,x_{m}'\right).
\end{equation*}
Let $L$ be field generated over $F_{0}$ by the coordinates of the $q(x_{i})$. This finite partial isomorphism yields a partial $\mathcal{L}_{D}$-isomorphism $\xi:D\dashrightarrow D'$ and an embedding $\sigma: L\rightarrow F'$ fixing $F_{0}$. 

Let $\tau\in D$ be a non-special point in a separate $G^{\mathrm{ad}}(\mathbb{Q})^{+}$-orbit from the $x_{i}$ such that the coordinates of $q(\tau)$ are algebraic over $L$ (so $q(\tau)\in\mathrm{acl}(q(x_{1}),\ldots,q(x_{m}))$). Let $V$ be the special closure of $\left(q(x_{1}),\ldots,q(x_{m}),q(\tau)\right)$. Let $\overline{g} = (e, g_{1},\ldots,g_{n})$ be a tuple of elements of $G$. Let $L_{\overline{g}}$ denote the field generated over $L$ by the coordinates of the $q(\overline{g}x_{i})$ and $q(\overline{g}\tau)$ (so $q(\tau)$ is defined over $L_{(e)}$, which means that $L$ and $L_{(e)}$ may be different). Also, let 
\begin{equation*}
    \overline{\Gamma^{V,i}}_{\overline{g}}:=\varprojlim \Gamma_{\overline{g},\overline{h}}^{V,i},
\end{equation*}
where $\overline{h}$ varies over all tuples such that $\Gamma_{\overline{h}}$ is contained in $\Gamma_{\overline{g}}$ and is normal in $\Gamma$. By FIC1, the image of the homomorphism 
\begin{equation*}
    \mathrm{Aut}(\mathbb{C}/L)\rightarrow\overline{\Gamma^{V,i}},
\end{equation*}
corresponding to the point $\left(q(x_{1}),\ldots,q(x_{m}),q(\tau)\right)\in S(\mathbb{C})^{m+1}$, has finite index. Therefore the image of the homomorphism is open, and so there exists a tuple $\overline{g}$ such that the homomorphism
\begin{equation*}
    \mathrm{Aut}\left(\mathbb{C}/L_{\overline{g}}\right)\rightarrow\overline{\Gamma^{V,i}}_{\overline{g}},
\end{equation*}
corresponding to the point $\left(q\left(\overline{g}x_{1}\right),\ldots,q\left(\overline{g}x_{m}\right),q\left(\overline{g}\tau\right)\right)\in Z^{V,i}_{\overline{g}}$, is surjective. Now observe that
\begin{equation*}
    \bigcup_{\overline{g}}\mathrm{qftp}\left(q\left(\overline{g}x_{1}\right),\ldots,q\left(\overline{g}x_{m}\right),q\left(\overline{g}\tau\right)\right)
\end{equation*}
is equivalent to
\begin{equation*}
    \mathrm{qftp}\left(q\left(\overline{g}x_{1}\right),\ldots,q\left(\overline{g}x_{m}\right),q\left(\overline{g}\tau\right)\right)\cup\bigcup_{\overline{h}}\mathrm{qftp}\left(q\left(\overline{h}x_{1}\right),\ldots,q\left(\overline{h}x_{m}\right),q\left(\overline{h}\tau\right)/L_{\overline{g}}\right),
\end{equation*}
where $\overline{h}$ varies over the tuples such that $\Gamma_{\overline{h}}$ is contained in $\Gamma_{\overline{g}}$ and normal in $\Gamma$. Using proposition \ref{prop:genrealisation}, we choose $\tau'\in D'$ such that
\begin{equation*}
    \xi\left(\mathrm{sploc}\left(\overline{g}\tau,\overline{g}x_{1},\ldots,\overline{g}x_{m}\right)\right) = \mathrm{sploc}\left(\overline{g}\tau',\overline{g}x_{1}',\ldots,\overline{g}x_{m}'\right)
\end{equation*}
 and
\begin{equation*}
    \mathrm{qftp}\left(q\left(\overline{g}x_{1}\right),\ldots,q\left(\overline{g}x_{m}\right),q\left(\overline{g}\tau\right)\right) = \mathrm{qftp}\left(q'\left(\overline{g}x'_{1}\right),\ldots,q'\left(\overline{g}x'_{m}\right),q'\left(\overline{g}\tau'\right)\right).
\end{equation*}

We now extend $\sigma:L\rightarrow F'$ to $\sigma:L_{\overline{g}}\rightarrow F'$. By Proposition \ref{prop:typedet}, to finish the proof we need to show that 
\begin{equation*}
    \xi\left(\mathrm{sploc}\left(\overline{h}\tau,\overline{h}x_{1},\ldots,\overline{h}x_{m}/\overline{g}x_{1},\ldots,\overline{g}x_{m},\overline{g}\tau\right)\right) = \mathrm{sploc}\left(\overline{h}\tau',\overline{h}x_{1}',\ldots,\overline{h}x_{m}'/\overline{g}x',\ldots,\overline{g}x_{m}',\overline{g}\tau'\right)
\end{equation*}
 (but this is already guaranteed by Corollary \ref{cor:sploc}), and that
\begin{equation}
\label{eq:claim}
    \mathrm{qftp}\left(q\left(\overline{h}x_{1}\right),\ldots,q\left(\overline{h}x_{m}\right),q\left(\overline{h}\tau\right)/L_{\overline{g}}\right) = \mathrm{qftp}\left(q\left(\overline{h}x'_{1}\right),\ldots,q\left(\overline{h}x'_{m}\right),q\left(\overline{h}\tau'\right)/\sigma\left(L_{\overline{g}}\right)\right),
\end{equation}
where $\overline{h}$ is any tuple such that $\Gamma_{\overline{h}}$ is contained in $\Gamma_{\overline{g}}$ and normal in $\Gamma$. The left-hand side of (\ref{eq:claim}) is determined by the $\mathrm{Aut}\left(\mathbb{C}/L_{\overline{g}}\right)$-orbit of $\left(q\left(\overline{h}x_{1}\right),\ldots,q\left(\overline{h}x_{m}\right),q\left(\overline{h}\tau\right)\right)\in Z^{V,i}_{\overline{h}}$ in the fibre of $\psi_{\overline{h},\overline{g}}$ over $\left(q\left(\overline{g}x_{1}\right),\ldots,q\left(\overline{g}x_{m}\right),q\left(\overline{g}\tau\right)\right)\in Z_{\overline{g}}^{V,i}$. Since we are assuming FIC1, this orbit is the whole fibre of the morphism:
\begin{equation*}
    \psi_{\overline{h},\overline{g}}^{V,i}:Z_{\overline{h}}^{V,i}\rightarrow Z_{\overline{g}}^{V,i}.
\end{equation*}
Similarly, the right-hand side of (\ref{eq:claim}) is determined by the $\mathrm{Aut}\left(\mathbb{C}/\sigma(L_{\overline{g}})\right)$-orbit containing $\left(q\left(\overline{h}x'_{1}\right),\ldots,q\left(\overline{h}x'_{m}\right),q\left(\overline{h}\tau'\right)\right)$ in the fibre of $\psi_{\overline{h},\overline{g}}^{V,i}$ over $\left(q\left(\overline{g}x'_{1}\right),\ldots,q\left(\overline{g}x'_{m}\right),q\left(\overline{g}\tau'\right)\right)$. We need to show that this orbit equals the fibre of $\psi_{\overline{h},\overline{g}}^{V,i}$. 

Extend $\sigma$ to an automorphism of $\mathbb{C}$ and choose a point $y$ in the fibre of $\psi_{\overline{h},\overline{g}}^{V,i}$ over $\left(q\left(\overline{g}x_{1}\right),\ldots,q\left(\overline{g}x_{m}\right),q\left(\overline{g}\tau\right)\right)$, and let $y'=\sigma(y)$. Let $\rho\in\mathrm{Aut}\left(\mathbb{C}/\sigma\left(L_{\overline{g}}\right)\right)$, then $\sigma^{-1}\rho\sigma\in\mathrm{Aut}\left(\mathbb{C}/L_{\overline{g}}\right)$ and
\begin{equation*}
    \rho\left(y'\right) = \rho\sigma(y) = \sigma\sigma^{-1}\rho\sigma(y) = \sigma\left(\phi\left(\sigma^{-1}\rho\sigma\right)\cdot y\right) = \phi\left(\sigma^{-1}\rho\sigma\right)\cdot\sigma(y) = \phi\left(\sigma^{-1}\rho\sigma\right)\cdot y'
\end{equation*}
(recall that by Lemma \ref{lem:actcom} the action of the automorphism commutes with that of the elements of the group). Therefore the homomorphism $\mathrm{Aut}\left(\mathbb{C}/\sigma\left(L_{\overline{g}}\right)\right)\rightarrow\Gamma_{\overline{g},\overline{h}}^{V,i}$, obtained from the composition of $\rho\mapsto\sigma^{-1}\rho\sigma$ and $\phi$, is surjective, which completes the proof of (\ref{eq:claim}). 
\end{proof}

The next definition is a translation of the condition on $\aleph_{0}$-homogeneity over countable models. We do not specify that $F$ be countable because, as we are only working with finitely many points, we will be working with a countable subfield of $F$ anyway. 

\begin{defi}
We say that a Shimura variety $p:X^{+}\rightarrow S(\mathbb{C})$ satisfies the \emph{second finite-index condition} (FIC2) \index{FIC2} if for any algebraically closed field $F\subseteq\mathbb{C}$ and any tuple $x_{1},\ldots,x_{m}\in D$ of non-special points in distinct $G^{\mathrm{ad}}(\mathbb{Q})^{+}$-orbits such that:
\begin{enumerate}
    \item $p(x_{i})\notin S(F)$ for all $i\in\left\{1,\ldots,m\right\}$
    \item the field $L$ generated over $F$ by the coordinates of the $p(x_{i})$ has transcendence degree $1$ over $F$,
\end{enumerate}
the image of the induced homomorphism 
\begin{equation*}
    \mathrm{Aut}(\mathbb{C}/L)\rightarrow\overline{\Gamma^{V,i}},
\end{equation*}
has finite index, where  $D_{V,i}=\mathrm{sploc}\left(x_{1},\ldots,x_{m}\right)$.
\end{defi}

\begin{prop}
\label{prop:main2}
If $p:X^{+}\rightarrow S(\mathbb{C})$ satisfies FIC2, then $\mathscr{K}_{\mathbf{p}}$ satisfies $\aleph_{0}$-homogeneity over countable models.
\end{prop}
\begin{proof}
The proof is the same as that of Proposition \ref{prop:main}, now assuming that there is a common countable model embedded in both $\mathbf{q}$ and $\mathbf{q}'$, and working over this countable subset using FIC2.
\end{proof}

\subsection{Conclusions}
In the proofs that follow we will preserve the notation of Hecke-generic curves introduced in \textsection \ref{subsec:heckegen}. In particular, we fix a Hecke-generic curve $C$ in $S$.

\begin{cor}
\label{cor:countable}
If $p:X^{+}\rightarrow S(\mathbb{C})$ satisfies FIC1 and FIC2, then $\mathscr{K}_{\mathbf{p}}$ is $\kappa$-categorical for $\kappa=\aleph_{0}$ and $\aleph_{1}$. 
\end{cor}
\begin{proof}
Every field of characteristic $0$ and cardinality at most $2^{\aleph_{0}}$ can be embedded in $\mathbb{C}$. Given $\mathbf{q}\in\mathscr{K}_{\mathbf{p}}$ of cardinality at most $\aleph_{1}$, the corresponding structure $\mathbf{r}_{\mathbf{q}}$ will define a model in $\mathscr{K}_{\mathbf{r}}$. Following the proofs of Propositions \ref{prop:main} and \ref{prop:main2}, we get that  $\left(\mathbf{r},\mathrm{cl}_{\mathbf{r}}\right)$ is a quasiminimal pregeometry structure. The Corollary now follows from Proposition \ref{prop:kirby}.
\end{proof}

\begin{thm}[First Main Categoricity Result]
\label{thm:main1}
The Shimura variety $p:X^{+}\rightarrow S$ satisfies FIC1 and FIC2 if and only if $\mathscr{K}_{\mathbf{p}}$ is $\kappa$-categorical for all infinite cardinalities $\kappa$.
\end{thm}
\begin{proof}
The necessary conditions are given by Corollary \ref{cor:fic}. So now we focus on the sufficient conditions. As we mentioned in Corollary \ref{cor:countable}, $\left(\mathbf{r}, \mathrm{cl}_{\mathbf{r}}\right)$ is a quasiminimal pregeometry structure. Using the notation from \textsection \ref{subsec:quasi}, the models of $\mathscr{K}(\mathbf{r})$ are determined (up to isomorphism) by their dimension by Theorem \ref{thm:quasi}. As we saw at the beginning of this section, the dimension of a model of $\mathscr{K}(\mathbf{r})$ is the transcendence degree (over $\mathbb{Q}$) of the field  corresponding to the variety of that model. As algebraically closed fields are isomorphic if and only if they have the same characteristic, cardinality and transcendence degree over the prime field, then there is a unique model $\mathbf{r}_{0}\in \mathscr{K}(\mathbf{r})$ of cardinality $\aleph_{0}$ and infinite dimension.

In order to prove categoricity of $\mathscr{K}_{\mathbf{r}}$, we need to show that $\mathscr{K}_{\mathbf{r}}=\mathscr{K}(\mathbf{r})$. It is clear that $\mathscr{K}(\mathbf{r})\subseteq\mathscr{K}_{\mathbf{r}}$, since $\mathbf{r}\in\mathscr{K}_{\mathbf{r}}$. For the converse, first note that the class $\mathscr{K}(\mathbf{r})$ is quasiminimal under our hypothesis. By Corollary \ref{cor:countable}, $\mathscr{K}_{\mathbf{r}}$ has a unique model $\mathbf{t}_{0}$ of cardinality $\aleph_{0}$. But as $\mathscr{K}(\mathbf{r})\subseteq\mathscr{K}_{\mathbf{r}}$, then $\mathbf{t}_{0}$ is isomorphic to $\mathbf{r}_{0}$, and so $\mathscr{K}(\mathbf{t_{0}})=\mathscr{K}(\mathbf{r_{0}})$. Let $\mathbf{t}\in\mathscr{K}_{\mathbf{r}}$, and let $B$ be a $\mathrm{cl}_{\mathbf{t}}$-basis of $\mathbf{t}$. Then $\mathbf{t} = \bigcup\left\{\mathrm{cl}_{\mathbf{t}}(B') : B'\subseteq B, |B'|=\aleph_{0}\right\}$. Also, for every $B'\subseteq B$ countable, we have that $\mathrm{cl}_{\mathbf{t}}(B')\in\mathscr{K}_{\mathbf{r}}$ is isomorphic to  $\mathbf{t}_{0}$. As $ \mathscr{K}(\mathbf{t_{0}})$ is closed under unions of chains of closed embeddings (all embeddings in $\mathscr{K}_{\mathbf{r}}$ are closed with respect to the pregeometries $\mathrm{cl}_{\mathbf{t}}$), then it is also closed under taking unions of directed systems, which shows that $\mathbf{t}\in \mathscr{K}(\mathbf{t_{0}})$. Therefore $\mathscr{K}_{\mathbf{r}} = \mathscr{K}(\mathbf{t_{0}})=\mathscr{K}(\mathbf{r_{0}}) =\mathscr{K}(\mathbf{r})$.  
\end{proof}

\begin{remark}
We have worked in a theory where we assume that the field has infinite transcendence degree. This is so that there is only one countable model. But the theory of quasiminimal classes also says that countable models will be determined by their dimension, i.e. their transcendence degree over $\mathbb{Q}$. So, assuming FIC1 and FIC2, two countable models of $\mathscr{K}_{\mathbf{p}}$ with respective fields $F_{1}$ and $F_{2}$ will be isomorphic if and only if $\mathrm{tr.deg.}_{\mathbb{Q}}(F_{1})=\mathrm{tr.deg.}_{\mathbb{Q}}(F_{2})$.
\end{remark}

Categoricity of Shimura curves was already proven in the main theorem of \cite{daw-harris}, although in a slightly different language. However, \cite[Lemma 3.4]{ribet} implies that the results of \cite[\textsection 5]{daw-harris} are enough to obtain categoricity in the language we have used here.

\section{Classification of Shimura Covers}
\label{sec:classification}
In this section we will be focused on classifying the minimal models of $\widetilde{T}(\widetilde{\mathbf{p}})$. We will use the techniques of independent systems and local isolation as described in \cite[\textsection 3]{bays-pillay} to classify minimal models. As explained in that paper, these methods provide a more robust approach to the categoricity question than quasiminimality. Intuitively, the method is the following: let $M_{0} = S\left(\overline{\mathbb{Q}}\right)$ be the prime model of $T$, and suppose we know that this is part of a model $\left<\widetilde{M}_{0}, M_{0}, \left\{q_{\overline{g}}\right\}_{\overline{g}}\right>\models \widetilde{T}(\widetilde{\mathbf{p}})$. If $F\models\mathrm{ACF}_{0}$, then $M\prec S(F)$, and we want to show that there is exactly one way in which to extend $q_{0}:\widetilde{M}_{0}\rightarrow M_{0}$ to a Shimura cover $q:\widetilde{M}\rightarrow S(F)$ such that the $q$-fibre over every point $z\in S(F)\setminus M_{0}$ is standard. This will be done by induction on the transcendence basis of $F$ over $\mathbb{Q}$ via independent systems.

Fix a monster model $\widetilde{\mathfrak{q}}:=\left<\widetilde{\mathbb{M}},\mathbb{M}, \left\{\mathfrak{q}_{\overline{g}}\right\}_{\overline{g}}\right>$ of $\widetilde{T}(\widetilde{\mathbf{p}})$. In particular, $\mathbb{M}$ is a monster model of $T$.

\subsection{Preliminaries}
Given $\widetilde{\mathbf{q}}_{1}:=\left<\widetilde{M}_{1},M_{1},\left\{q_{\overline{g}}\right\}_{\overline{g}}\right>$ and $\widetilde{\mathbf{q}}_{2}:=\left<\widetilde{M}_{2},M_{2},\left\{q_{\overline{g}}'\right\}_{\overline{g}}\right>$, two models of $\widetilde{T}(\widehat{\mathbf{p}})$, we write $\widetilde{\mathbf{q}}_{1}\prec^{\ast}\widetilde{\mathbf{q}}_{2}$ to mean that $\widetilde{\mathbf{q}}_{1}\prec\widetilde{\mathbf{q}}_{2}$ and that for every $z\in M_{1}$, $q'^{-1}(z)=q^{-1}(z)$. Such elementary extensions are called \emph{fibre-preserving}.\index{fibre-preserving extension}

\begin{remark}
\label{rem:fp}
Suppose $\widetilde{\mathbf{q}}\prec\widetilde{\mathbf{q}}'$. Then $\widetilde{\mathbf{q}}\prec^{\ast}\widetilde{\mathbf{q}}'$ if and only if  in $\widetilde{M}_{2}$ we have that $\widetilde{M}_{1}=q'^{-1}(M_{1})$.
\end{remark}

\begin{lem}
\label{lem:3.3}
Let $\widetilde{\mathbf{q}}\models\widetilde{T}(\widehat{\mathbf{p}})$ and $M'\prec M$. Let $\widetilde{M}':=q^{-1}(M')$. Then by taking the structure on $\widetilde{M}'$ induced by $\widetilde{M}$ and letting $q_{\overline{g}}'$ denote the restriction of $q_{\overline{g}}$ to $\widetilde{M}'$, we get that $\widetilde{\mathbf{q}}':=\left<\widetilde{M}',M',\left\{q_{\overline{g}}'\right\}_{\overline{g}}\right>\prec^{\ast}\widetilde{\mathbf{q}}$.
\end{lem}
\begin{proof}
By Remark \ref{rem:fp} and quantifier elimination, it suffices to show that $\widetilde{\mathbf{q}}\models\widetilde{T}(\widetilde{\mathbf{p}})$. But this is immediate as it satisfies axioms (A1)--(A6).
\end{proof}

\begin{lem}
\label{lem:3.6b}
Let $\widetilde{\mathbf{q}},\widetilde{\mathbf{q}}'\models\widetilde{T}(\widetilde{\mathbf{p}})$. Suppose we have a model $M_{1}\models T$ contained in $\mathbb{M}$ such that $M\prec M_{1}$. Let $A:=\widetilde{M}\cup M_{1}$, and suppose that $\widetilde{\mathbf{q}}'$ is $\ell$-atomic over $A$. Then $\widetilde{\mathbf{q}}\prec^{\ast}\widetilde{\mathbf{q}}'$ and $M' = M_{1}$. 
\end{lem}
\begin{proof}
First we show that fibres are preserved. For this we use part (b) of Lemma \ref{lem:ell}. Fix $z\in M$ and let $\phi(v)$ be the formula $q(v)=z$. Let $\beta\in\mathrm{dcl}^{\mathrm{eq}}(A)\cap\mathrm{dcl}^{\mathrm{eq}}\left(\phi\left(\widetilde{\mathbb{M}}\right)\right)$ and choose a finite tuple $\overline{x}$ of $\phi\left(\widetilde{\mathbb{M}}\right)$ such that $\beta\in\mathrm{dcl}^{\mathrm{eq}}(\overline{x})$. Observe that the components of $\overline{x}$ are all in the fibre above $z$, and as $z\in M$, we get by Corollary \ref{cor:typedet2} that $\mathrm{tp}(\overline{x}/A)$ is determined by $\mathrm{tp}\left(\overline{x}/\widetilde{M}\right)$. Let $\theta_{1}(u,\overline{x})$ be a formula over $\overline{x}$ defining $\beta$ and let $\theta_{2}(u,\overline{a})$ be a formula over $A$ defining $\beta$. So 
\begin{equation}
\label{eq:beta}
    \exists ! u\left(\theta_{1}\left(u,\overline{v}\right)\wedge\theta_{2}\left(u,\overline{a}\right)\right)\in\mathrm{tp}(\overline{x}/A),
\end{equation}
where $\mathrm{tp}\left(\overline{x}/A\right)$ is now taken over $\widetilde{\mathcal{L}}^{\mathrm{eq}}$. Then there is an $\widetilde{\mathcal{L}}^{\mathrm{eq}}$-formula $\psi(\overline{v})\in\mathrm{tp}\left(\overline{x}/\widetilde{M}\right)$ which implies (\ref{eq:beta}). As $\widetilde{T}(\widetilde{\mathbf{p}})$ is complete and $\psi(\overline{v})$ is realised in $\widetilde{\mathfrak{q}}$, then it is realised in $\widetilde{\mathbf{q}}$. We conclude that $\beta\in \widetilde{\mathbf{q}}^{\mathrm{eq}}$. As $\widetilde{\mathbf{q}}'$ is $\ell$-atomic over $A$, then for any $x\in\widetilde{M}'$ such that $q'(x)=z$, $\mathrm{tp}(x/A)$ is $\ell$-isolated, and so by part b) of Lemma \ref{lem:ell} we conclude that $x\in\widetilde{M}$, thus finishing the proof that fibres are preserved. 

Now we prove that $M' = M_{1}$. Let $\psi(w)$ be the formula $\exists v (q(v) = w)$. Let $\beta\in\mathrm{dcl}^{\mathrm{eq}}(A)\cap\mathrm{dcl}^{\mathrm{eq}}(\psi(\mathbb{M}))$ and let $\overline{z}$ be a finite tuple of $\psi(\mathbb{M})$ such that $\beta\in\mathrm{dcl}^{\mathrm{eq}}(\overline{z})$. As before, by quantifier elimination $\mathrm{tp}\left(\overline{z}/A\right)$ is determined by $\mathrm{tp}\left(\overline{z}/M_{1}\right)$. Set $\widetilde{M}_{1}=\widetilde{S}(F)$, where $M_{1} = S(F)$. We know from Corollary \ref{cor:promodels} that the natural maps $q^{1}_{\overline{g}}:\widetilde{M}_{1}\rightarrow M_{1}^{n}$ define an elementary substructure $\widetilde{\mathbf{q}}_{1}:=\left<\widetilde{M}_{1}, M_{1}, \left\{q^{1}_{\overline{g}}\right\}_{\overline{g}}\right>$ of $\widetilde{\mathfrak{q}}$, so $\mathrm{tp}\left(\overline{z}/A\right)$ is determined by $\mathrm{tp}\left(\overline{z}/\widetilde{M}_{1}\right)$, and therefore $\beta\in\widetilde{\mathbf{q}}_{1}^{\mathrm{eq}}$ just as before. By $\ell$-atomicity and part (b) of Lemma \ref{lem:ell}, we conclude that if $z\in M'$ satisfies $\psi(z)$, then $z\in M_{1}$. Therefore $M_{1} = M'$.
\end{proof}

\begin{lem}
\label{lem:3.7}
Suppose $\widetilde{\mathbf{q}}_{1}$, $\widetilde{\mathbf{q}}_{2}$ and $\widetilde{\mathbf{q}}_{3}$ are elementary submodels of $\widetilde{\mathfrak{q}}$, with $\widetilde{\mathbf{q}}_{1}\prec^{\ast}\widetilde{\mathbf{q}}_{2}$, $\widetilde{\mathbf{q}}_{1}\prec^{\ast}\widetilde{\mathbf{q}}_{3}$, and $\widetilde{\mathbf{q}}_{2}\forkindep[\widetilde{\mathbf{q}}_{1}]\widetilde{\mathbf{q}}_{3}$. Suppose $\widetilde{\mathbf{q}}_{4}$ be an $\ell$-atomic model over $\widetilde{\mathbf{q}}_{2}\cup\widetilde{\mathbf{q}}_{3}$. Then $\widetilde{\mathbf{q}}_{i}\prec^{\ast}\widetilde{\mathbf{q}}_{4}$, for $i=1,2,3$.  
\end{lem}
\begin{proof}
Because of Proposition \ref{prop:qe2}, it remains to check that fibres are preserved. To that end, suppose that they are not preserved, say that there is $z\in M_{3}$ for which there is $x\in\widetilde{M}_{4}$ such that $q_{4}(x)=z$ but $x\notin q_{3}^{-1}(z)$. Let $\psi(v,u)$ be the formula 
\begin{equation*}
\psi(v,u):=\left(q(v)=q(u)\right)\wedge(v\neq u)
\end{equation*}
For $y\in\widetilde{M}_{3}$ such that $q_{3}(y)=z$, we have that $\psi(v,y)\in\mathrm{tp}\left(x/\widetilde{M}_{2}\cup\widetilde{M}_{3}\right)$, in particular $\psi(v,y)\in\mathrm{tp}\left(x/\widetilde{M}_{2}\cup\widetilde{M}_{3}\right)_{\psi}$. By $\ell$-atomicity there is $\phi(v,\overline{a_{2}},\overline{a_{3}})$, with $\overline{a_{i}}$ a tuple of $\widetilde{M}_{i}$, such that
\begin{equation*}
\phi(v,\overline{a_{2}},\overline{a_{3}})\models\mathrm{tp}\left(x/\widetilde{M}_{2}\cup\widetilde{M}_{3}\right)_{\psi}.
\end{equation*}
Observe that for any $y\in\widetilde{M}_{2}$ such that $q_{2}(y)=z$, we will have $\psi(v,y)\in\mathrm{tp}\left(x/\widetilde{M}_{2}\cup\widetilde{M}_{3}\right)_{\psi}$.  

By Corollary \ref{cor:typedet2} we may assume that $\phi(v,\overline{a_{2}},\overline{a_{3}})$ is of the form 
\begin{equation*}
    \left(v\in\widetilde{M}_{V,i}(\overline{a_{2}},\overline{a_{3}})\right)\wedge \left(q_{\overline{g}}(v)\in q_{\overline{g}}\left(\widetilde{M}_{V,i}(\overline{a_{2}},\overline{a_{3}})\right)\right).
\end{equation*}
By independence, $\mathrm{tp}\left(\overline{a_{2}}/\widetilde{M}_{3}\right)$ is finitely satisfiable in $\widetilde{M}_{1}$, and using the form we have chosen for $\phi$, we can choose $\overline{a_{1}}$ in $\widetilde{M}_{1}$ such that
\begin{equation*}
\widetilde{\mathbf{q}}_{3}\models \exists v\left(\phi\left(v,\overline{a_{1}},\overline{a_{3}}\right)\wedge q(v)=z\right),
\end{equation*}
which is witnessed by some $x'\in q_{3}^{-1}(z)$. But:
\begin{equation*}
\forall v\left(\phi\left(v,\overline{u},\overline{a_{3}}\right)\rightarrow\psi\left(v,x'\right)\right)\in\mathrm{tp}\left(\overline{a_{2}}/\widetilde{M}_{3}\right),
\end{equation*}
which leads to a contradiction, as we would have that $\widetilde{\mathbf{q}}_{3}\models\psi\left(x',x'\right)$.
\end{proof}

\begin{lem}
\label{lem:3.9}
Let $\widetilde{\mathbf{q}}\models\widetilde{T}(\widetilde{\mathbf{p}})$. Suppose $B\subseteq\widetilde{M}$ and $a\in\widetilde{M}$ are such that for any tuple $\overline{g}$ of elements of $G^{\mathrm{ad}}(\mathbb{Q})^{+}$ we have that $\mathrm{tp}\left(q_{\overline{g}}(a)/\widehat{B}\right)$ is isolated. Then $\mathrm{tp}(a/B)$ is $\ell$-isolated. 
\end{lem}
\begin{proof}
By quantifier elimination, it suffices to prove that $\mathrm{tp}(a/B)_{\phi}$ is isolated for an atomic formula $\phi(x,y)$. If $\phi$ is of the form $\psi\left(q_{\overline{g}}(x),q_{\overline{g}}(y)\right)$, then this follows by isolation of $\mathrm{tp}\left(q_{\overline{g}}(a)/\widehat{B}\right)$. For $\phi$ of the form $(x,y)\in\widetilde{M}_{V,i}$, it follows from the fact that for $b\in B$, $(a,b)\in\widetilde{M}_{V,i}$ if and only if $\mathrm{sploc}(a/B)\subseteq \widetilde{M}_{V,i}(b)$.
\end{proof}

\begin{lem}
\label{lem:constructible}
Suppose $\widetilde{\mathbf{q}}\models\widetilde{T}(\widetilde{\mathbf{p}})$ satisfies SF and that it is atomic over a set $U\subseteq\widetilde{M}\cup M$. If $M$ is countable, then $\widetilde{\mathbf{q}}$ is constructible over $U$. 
\end{lem}
\begin{proof}
Let $A\subseteq\widetilde{M}$ be a section of $q$. Then $A$ is countable and atomic over $U$, hence it is constructible. Let $c_{1}\in\widetilde{M}$ and choose $c_{2}\in A$ such that $q(c_{1}) = q(c_{2})$. By Corollary \ref{cor:gact} and SF, $c_{1} = \gamma c_{2}$ for some $\gamma\in\Gamma$, and so $c_{1}$ is constructible over $c_{2}$. 
\end{proof}

\begin{defi}
An $\widetilde{I}$-system \index{i@$\widetilde{I}$-system} in an $I$-system $\left(\widetilde{\mathbf{q}}_{s}\right)_{s\in I}$ in $\widetilde{\mathfrak{q}}$ such that:
\begin{itemize}
    \item $\left(M_{s}\right)_{s\in I}$ is an independent atomic $I$-system in $T$. 
    \item If $s\subseteq t$, then $\widetilde{\mathbf{q}}_{s}\prec^{\ast}\widetilde{\mathbf{q}}_{t}$.
\end{itemize}
\end{defi}


\begin{prop}
An $\widetilde{I}$-system $\left(\widetilde{\mathbf{q}}_{s}\right)_{s\in I}$ is an independent $I$-system.
\end{prop}
\begin{proof}
Let $I=(s_{i})_{i\in\lambda}$ be an enumeration of $I$. By Lemma \ref{lem:fact} it suffices to show that, given $i\in\lambda$, $\widetilde{\mathbf{q}}_{s_{i}}\forkindep[\widetilde{\mathbf{q}}_{<s_{i}}]\widetilde{\mathbf{q}}_{s_{<i}}$, where we may assume inductively that the restriction of $\left(\widetilde{\mathbf{q}}_{s}\right)_{s}$ to $s_{<i}$ is and independent system. Furthermore, by Proposition \ref{prop:forking} and the independence of $\left(M_{s}\right)_{s}$, it suffices to show that for every $a\in\widetilde{M}_{s_{i}}$, $\mathrm{sploc}\left(a/\widetilde{M}_{s_{<i}}\right)$ is defined over $\widetilde{M}_{<s_{i}}$. 




Say that $\mathrm{sploc}\left(a/\widetilde{M}_{s_{<i}}\right) = \widetilde{M}_{V,j}(\overline{c})$, where $\overline{c}$ is a tuple from $\widetilde{M}_{s_{<i}}$. So $q(a)\in V\left(q\left(\overline{c}\right)\right)$. 
In $\mathrm{ACF}_{0}$, model-theoretic independence ($\forkindep$) can be understood as algebraic independence (see \cite[Corollary 6.4.5]{tent-ziegler}) in the following sense: if $F,K,L\models\mathrm{ACF}_{0}$, then $K\forkindep[F] L$ if and only if for every finite tuple $\overline{k}$ from $K$, we have that $\mathrm{tr.deg}_{F}\left(\overline{k}\right) = \mathrm{tr.deg}_{L}\left(\overline{k}\right)$. This means then that there is a tuple $\overline{z}$ in $M_{<s_{i}}$ such that $V\left(q\left(\overline{c}\right)\right) = V\left(\overline{z}\right)$, and so we can assume that the coordinates of $q\left(\overline{c}\right)$ are in $M_{<s_{i}}$. 
As $\widetilde{I}$-systems preserve fibres, then $\overline{c}\in\widetilde{M}_{<s_{i}}$, so $\mathrm{sploc}\left(a/\widetilde{M}_{s_{<i}}\right)$ is defined over $\widetilde{M}_{<s_{i}}$.
\end{proof}

\subsection{\texorpdfstring{$\omega$}{w}-Stability Over Models}
We will say that $\widetilde{T}(\widetilde{\mathbf{p}})$ \emph{satisfies FIC} \index{FIC} if  $p:X^{+}\rightarrow S(\mathbb{C})$ satisfies FIC1 and FIC2.

\begin{prop}
\label{prop:3.11}
Assume $\widetilde{T}(\widetilde{\mathbf{p}})$ satisfies FIC. Let $\widetilde{\mathbf{q}}\prec\widetilde{\mathfrak{q}}$ and $b\in\mathbb{M}$. Let $M(b)\models T$ be prime over $Mb$. Suppose $\widetilde{\mathbf{q}}'$ is a model such that $\widetilde{\mathbf{q}}\prec^{\ast}\widetilde{\mathbf{q}}'\prec\widetilde{\mathfrak{q}}$ and $M' = M(b)$. Then $\widetilde{\mathbf{q}}'$ is atomic over $\widetilde{\mathbf{q}}b$. If $M$ is also countable and $\widetilde{\mathbf{q}}'$ satisfies SF, 
then $\widetilde{\mathbf{q}}'$ is constructible over $\widetilde{\mathbf{q}}b$. Furthermore, if $\widetilde{\mathbf{q}}$ satisfies SF, then such a $\widetilde{\mathbf{q}}'$ exists.  
\end{prop}
\begin{proof}
First we show atomicity. Let $c\in\widetilde{M}'$, we need to show that $\mathrm{tp}\left(c/\widetilde{M}b\right)$ is isolated. Let $\widetilde{M}'_{V,i}(\overline{x}):=\mathrm{sploc}\left(c/\widetilde{M}\right)$, where $\overline{x}$ is a tuple from $\widetilde{M}$. Frist we will show that there is a tuple $\overline{g}$ from $G^{\mathrm{ad}}(\mathbb{Q})^{+}$ such that 
\begin{equation}
\label{eq:models}
    \mathrm{tp}\left(q_{\overline{g}}(c)/M\right)\models \mathrm{tp}\left(\widehat{c}/M\right).
\end{equation}
By Propositions \ref{prop:main} and \ref{prop:main2} (which assume FIC), we have that there is a tuple $\overline{g}$  satisfying (\ref{eq:models}) when $M$ is contained in $\mathbb{C}$. Now, if $z$ is any element in $\mathbb{M}$ transcendental over $M$, then as $T$ has finite Morley rank, $\mathrm{tp}\left(z/Mq(c)\right)$ has finite multiplicity (see \cite[Exercise 8.5.7]{tent-ziegler}). Therefore, $\mathrm{tp}\left(\widehat{c}/M\right)\cup\mathrm{tp}\left(q(c)/Mz\right)$ has finitely many extensions to $Mz$. So again, for some tuple $\overline{g}$ we have
\begin{equation*}
    \mathrm{tp}\left(q_{\overline{g}}(c)/Mz\right)\models \mathrm{tp}\left(\widehat{c}/Mz\right).
\end{equation*}
And so, an inductive argument on a transcendence basis of $M$ over the prime model shows that there is a tuple $\overline{g}$ satisfying (\ref{eq:models}) in general. Furthermore, by the same argument on finite extensions, we get then that there is a tuple $\overline{g}$ satisfying
\begin{equation}
\label{eq:models2}
    \mathrm{tp}\left(q_{\overline{g}}(c)/Mb\right)\models \mathrm{tp}\left(\widehat{c}/Mb\right).
\end{equation}

Combining (\ref{eq:models2}) with Corollary \ref{cor:typedet2}, we get
\begin{equation*}
    \mathrm{tp}\left(q_{\overline{g}}(c)/Mb\right)\cup\left\{v\in\widetilde{M}'_{V,i}\right\}\models \mathrm{tp}\left(c/\widetilde{M}b\right).
\end{equation*}
As $q_{\overline{g}}(c)\in (M')^{n}$ and $M'$ is prime over $Mb$,  $\mathrm{tp}\left(q_{\overline{g}}(c)/Mb\right)$ is isolated. So $\mathrm{tp}\left(c/\widetilde{M}b\right)$ is isolated, which finishes the proof of atomicity.

For constructibility assuming countability and SF, use Lemma \ref{lem:constructible}. 

It remains to prove that such a $\widetilde{\mathbf{q}}'$ exists. By Lemma \ref{lem:ell}(a), there exists a model which is $\ell$-constructible over $\widetilde{M}\cup M'$. By Lemma \ref{lem:3.6b} we get that $\widetilde{\mathbf{q}}\prec^{\ast}\widetilde{\mathbf{q}}'$ and $M'=M(b)$. That $\widetilde{\mathbf{q}}'$ satisfies SF can be assumed by Proposition \ref{prop:min}. 
\end{proof}

\subsection{Atomicity Over Independent Systems}
\label{subsec:indepsys}
We first fix some notation. Let $M:=S(F)\models T$, and let $K_{0} = \overline{F_{0}}$. Then $M_{0} = S(K_{0})$ is the prime model of $T$. If $B$ is a transcendence basis for $F$ over $K_{0}$, then $M$ is constructible and prime over $M_{0}B$.

Let $\mathcal{P}^{\mathrm{fin}}(B)$ be the set of finite subsets of $B$. Let $M_{\emptyset}:= M_{0}$, and for $s\in\mathcal{P}^{\mathrm{fin}}(B)$ define inductively $M_{s}\prec M(F)$ to be prime over $M_{<s}\cup s$, that is $M_{s} = S\left(\overline{F(s)}\right)$.

\begin{lem}
\label{lem:3.23}
$\left(M_{s}\right)_{s\in\mathcal{P}^{\mathrm{fin}}(B)}$ is a constructible independent $\mathcal{P}^{\mathrm{fin}}(B)$-system, and 
\begin{equation*}
\bigcup_{s\in\mathcal{P}^{\mathrm{fin}}(B)}M_{s} = S(F). 
\end{equation*}
\end{lem}
\begin{proof}
$\bigcup_{s\in\mathcal{P}^{\mathrm{fin}}(B)}M_{s}$ is an elementary submodel of $M$ containing $M_{0}B$, and as $M$ is minimal over $M_{0}B$, then we must have $\bigcup_{s\in\mathcal{P}^{\mathrm{fin}}(B)}M_{s} = S(F)$.

$T$ is totally transcendental, so every subset of a model has a constructible prime extension (see \cite[Theorem 5.3.3]{tent-ziegler}), and thus we get constructibility of the system. It remains to show that the system is independent. Also, in totally transcendental theories, prime extensions are atomic (and, hence, $\ell$-atomic). Using the finite character of $\forkindep$ and Lemma \ref{lem:3.22}, it will suffice to show that, for $b\in B$, $M_{\left\{b\right\}}\forkindep[M_{\emptyset}]M_{s}$ when $b\notin s\in\mathcal{P}^{\mathrm{fin}}(B)$.

So let $s\in\mathcal{P}^{\mathrm{fin}}(B)$ be such that $b\notin s$. By induction, we can assume that the restriction of the system to $s$ is independent. By Lemma \ref{lem:3.21}, $M_{s}$ is constructible over $M_{0}s$. Observe that $b\notin M_{s}$ and $M_{\left\{b\right\}}\forkindep[M_{\emptyset}] M_{s}$, which finishes the proof. 
\end{proof}

The following Proposition is a straightforward restatement of \cite[Proposition 3.26]{bays-pillay}. With what we have shown so far, the proof of \cite[Proposition 3.26]{bays-pillay} shows that Proposition \ref{prop:3.26} is a formal consequence of previous results, so we refer the reader to that paper for the proof. 

\begin{prop}[see {{\cite[Proposition 3.26]{bays-pillay}}}]
\label{prop:3.26}
Assume $\widetilde{T}(\widetilde{\mathbf{p}})$ satisfies FIC. Let $\left(\widetilde{\mathbf{q}}_{s}\right)_{s\in I}$ be an $\widetilde{I}$-system, with $I$ Noetherian. Then the system is atomic. Also, if each $M_{s}$ is countable and each $\widetilde{\mathbf{q}}_{t}$ satisfies SF, then the system is constructible. 
\end{prop}

\subsection{Classification of Models with Standard Fibres}

\begin{lem}
\label{lem:3.30}
Assume that $\widetilde{T}(\widetilde{\mathbf{p}})$ satisfies FIC. Let $S(F)=M\models T$, let $S(\overline{\mathbb{Q}})\cong M_{0}\prec M$ be the prime model of $T$, and let $B$ be a transcendence basis of $F$ over $\mathbb{Q}$. Let $\widetilde{\mathbf{q}}_{0}\models\widetilde{T}(\widetilde{\mathbf{p}})$ be such that it satisfies SF and $M_{0}$ is the variety sort of $\widetilde{\mathbf{q}}_{0}$. Then there exists $\widetilde{\mathbf{q}}\models\widetilde{T}(\widetilde{\mathbf{p}})$ such that $M$ is the variety sort of $\widetilde{\mathbf{q}}$, $\widetilde{\mathbf{q}}_{0}\prec^{\ast}\widetilde{\mathbf{q}}$, and $\widetilde{\mathbf{q}}$ is constructible over $B\cup\widetilde{\mathbf{q}}_{0}$, and it satisfies SF.
\end{lem}
\begin{proof}
Set $I=\mathcal{P}^{\mathrm{fin}}(B)$. Let $(M_{s})_{s\in I}$ be defined as in the beginning of  \textsection\ref{subsec:indepsys}, which is a constructible independent $I$-system by Lemma \ref{lem:3.23}. By Lemma \ref{lem:ell}(a) and Proposition \ref{prop:min}, there is a model $\widetilde{\mathbf{q}}$ which is $\ell$-constructible and minimal over $M\widetilde{\mathbf{q}}_{0}$, so it will have standard fibres. Set $\widetilde{M}_{s} = q^{-1}(M_{s})\subseteq\widetilde{M}$. By Lemma \ref{lem:3.6b}, $\widetilde{\mathbf{q}}_{\emptyset} = \widetilde{\mathbf{q}}_{0}$ and $q\left(\widetilde{M}\right) = M$, and by Lemma \ref{lem:3.3}, $(\widetilde{\mathbf{q}}_{s})_{s}$ is an $\widetilde{I}$-system. By Proposition \ref{prop:3.26}, $(\widetilde{\mathbf{q}}_{s})_{s}$ is a constructible independent system. By Proposition \ref{prop:3.11}, for each $b\in B$ the model $\widetilde{\mathbf{q}}_{\left\{b\right\}}$ is constructible over $\widetilde{\mathbf{M}}_{\emptyset}b$, and so by Lemma \ref{lem:3.21}, $\widetilde{\mathbf{q}} = \widetilde{\mathbf{q}}_{I}$ is constructible over $\widetilde{\mathbf{q}}_{\emptyset}B=\widetilde{\mathbf{q}}_{0}B$. That $\widetilde{\mathbf{q}}$ satisfies SF comes from the fact that, as the $\widetilde{I}$-system preserves fibres, then the fibre a point $z\in M$ is equal to the fibre in $\widetilde{M}_{s}$ of $z$ seen as an element of $M_{s}$, for some $s\in I$. As every model of the $\widetilde{I}$-system satisfies SF, then so does $\widetilde{\mathbf{q}}$.
\end{proof}

\begin{thm}[Second Main Categoricity Result]
Assume $\widetilde{T}(\widetilde{\mathbf{p}})$ satisfies FIC. Let $\widetilde{\mathbf{q}}_{0},\widetilde{\mathbf{q}},\widetilde{\mathbf{q}}'\models\widetilde{T}(\widetilde{\mathbf{p}})$ be such that $\widetilde{\mathbf{q}}_{0}$ satisfies SF, $\widetilde{\mathbf{q}}_{0}\prec^{\ast}\widetilde{\mathbf{q}}$ (where $M_{0}$ is the prime model of $T$), $\widetilde{\mathbf{q}}_{0}\prec^{\ast}\widetilde{\mathbf{q}}'$, and $\widetilde{\mathbf{q}},\widetilde{\mathbf{q}}'$ are minimal over $\widetilde{\mathbf{q}}_{0}$. Suppose that there is an isomorphism $M\cong M'$ which is the identity on $M_{0}$. Then $\widetilde{\mathbf{q}}\cong_{\widetilde{\mathbf{q}}_{0}}\widetilde{\mathbf{q}}'$ (but this isomorphism might not agree with the isomorphism $M\cong M'$). 
\end{thm}
\begin{proof}
If $S(F)=M\cong M'=S(F')$, let $B$ be a transcendence basis of $F$ over $\overline{\mathbb{Q}}$, and let $B'$ be the image of $B$ in $F'$. Then $\widetilde{\mathbf{q}}$ and $\widetilde{\mathbf{q}}'$ are constructible and minimal over $\widetilde{\mathbf{q}}_{0}B$ by Lemma \ref{lem:3.30}. By quantifier elimination of $\widetilde{T}(\widetilde{\mathbf{p}})$, $B\cup\widetilde{\mathbf{q}}_{0}\equiv B'\cup\widetilde{\mathbf{q}}_{0}$, and by constructibility of $\widetilde{\mathbf{q}}$ over $B\cup\widetilde{\mathbf{q}}_{0}$, this extends to an elementary embedding $\widetilde{\mathbf{q}}\prec\widetilde{\mathbf{q}}'$. By minimality of $\widetilde{\mathbf{q}}'$ over $B'\cup\widetilde{\mathbf{q}}_{0}$, this embedding is an isomorphism.
\end{proof}

\section{Additional Remarks}
\label{sec:remarks}
\subsection{Reduction to the Neat Case}
\label{subsec:neat}
For completeness, we now give the argument shown in \cite[\textsection 5.3]{daw-harris} to show that if we have categoricity when $\Gamma$ is neat, then we also have categoricity when we take a group which is commensurable with $\Gamma$ but not necessarily neat. 

Suppose that $\Gamma$ is a congruence subgroup that gives rise to a locally symmetric variety $\Gamma\backslash X^{+}$ which has a canonical model $S$ over an abelian extension $E^{S}$ of $E$, but is not necessarily neat (like it happens with the $j$ function). Let $p: X^{+}\rightarrow S(\mathbb{C})$ be the usual map. We want to know if $\mathrm{Th}_{\mathrm{SF}}^{\infty}(\mathbf{p})$ is categorical. Take two models of the same cardinality: $q:D\rightarrow S(F)$ and $q':D'\rightarrow S(F')$. By Proposition \ref{prop:borel}, we know that $\Gamma$ is virtually neat and that the finite-index neat subgroup can be taken to be a principal congruence subgroup. We observe that this subgroup contains a subgroup of the form $\Gamma_{\overline{g}}$, for some tuple $\overline{g}$ of $G^{\mathrm{ad}}(\mathbb{Q})^{+}$, and so it is also neat.

Let $Z_{\overline{g}}(\mathbb{C}) = \Gamma_{\overline{g}}\backslash X^{+}$. As we have seen before, we can produce a morphism of Shimura varieties $\pi:Z_{\overline{g}}\rightarrow S$ defined over the composite of the respective reflex fields, but as both varieties have the same Shimura datum, it follows that $\pi$ is defined over $E^{S}(\Sigma)$. Consider the maps
\begin{equation*}
\begin{array}{ccccccc}
q_{\overline{g}} : D  &\longrightarrow&  Z_{\overline{g}}(F) &,& q'_{\overline{g}} : D' & \longrightarrow & Z_{\overline{g}}(F')\\
x  &\mapsto&  \left(q(\overline{g}x)\right) &&  x & \mapsto & \left(q'(\overline{g}x)\right).
\end{array}
\end{equation*}
Observe that $p_{\overline{g}}$ is the uniformisation map of $Z_{\overline{g}}(\mathbb{C})$ and that, becasuse $q$ and $q'$ satisfy the conditions SS and SF determined by $p$, then $q_{\overline{g}}$ and $q'_{\overline{g}}$ satisfy the conditions SS and SF determined by $p_{\overline{g}}$. Applying Theorem \ref{thm:main1} to $p_{\overline{g}}$ we get that there is an isomorphism of models
\begin{equation*}
\begin{CD}
D @>\xi>> D'\\
@VVq_{\overline{g}}V @VVq'_{\overline{g}}V\\
Z_{\overline{g}}(F) @>\sigma>> Z_{\overline{g}}(F'),
\end{CD}.   
\end{equation*}
From the work we did in \textsection \ref{subsec:galrep}, it is immediate that the special structure induced on $X^{+}$ by $Z_{\overline{g}}(\mathbb{C})$ is the same as that induced by $S(\mathbb{C})$. As $p = \pi\circ p_{\overline{g}}$, then $q = \pi\circ q_{\overline{g}}$ and $q' = \pi\circ q_{\overline{g}}$. As $\pi$ is defined over $E^{S}(\Sigma)$, we conclude that the automorphism $\sigma$ descends to give an isomorphism of the original models:
\begin{equation*}
\begin{CD}
D @>\xi>> D'\\
@VVpV @VVqV\\
S(F) @>\sigma>> S(F').
\end{CD}    
\end{equation*}

\subsection{Categoricity of \texorpdfstring{$\mathcal{A}_{g}$}{A}}

We will now focus on the case of $\mathcal{A}_{g}$, the moduli space of principally polarised abelian varieties of dimension $g$. We recall that in this case $G = \mathrm{GSp}_{2g}$ and we can view $\mathcal{A}_{g}$ as the quotient $\Gamma\backslash\mathbb{H}^{+}_{g}$, where $\Gamma = \mathrm{Sp}_{2g}(\mathbb{Z})$. Also  $E^{\mathcal{A}_{g}}=\mathbb{Q}$ (cf \cite[p. 352]{milne}). Because we will need the statement of the Mumford-Tate conjecture, we will use $F_{0} = \mathbb{Q}^{\mathrm{ab}}(\Sigma)$. 

Unfortunately, in order to stick to standard notation, we will have to use the same symbol, $g$, for the dimension of the abelian variety and for elements of the group $G$. However this should not lead to confusion.

We remark first that, just as with the case of the $j$ function, the map $\mathbb{H}^{+}_{g}\rightarrow\mathcal{A}_{g}$ is ramified as $\mathrm{Sp}_{2g}(\mathbb{Z})$ is not neat. So $\mathcal{A}_{g}$ does not fall into the class of Shimura varieties we have been considering. This can be fixed by considering principal congruence subgroups of $\mathrm{Sp}_{2g}(\mathbb{Z})$. 

Let $A$ be an abelian variety of dimension $g$ defined over a number field $K$. Given an integer $N$, denote the $N$-torsion of $A$ by $A[N]$ (here we think of $A[N]$ as a subgroup of $A\left(\overline{K}\right)$). A \emph{level $N$-structure} on an $A$ is an isomorphism
\begin{equation}
\label{eq:symiso}
    \psi_{N} : A[N] \rightarrow \left(\mathbb{Z}/N\mathbb{Z}\right)^{2g} 
\end{equation}
which preserves the standard symplectic form. 

We will now look at the varieties:
\begin{equation*}
    \mathcal{A}_{g,N} := \Gamma(N)\backslash\mathbb{H}_{g}^{+},
\end{equation*}
where $\Gamma(N)$ is the principal congruence subgroup of $\mathrm{Sp}_{2g}(\mathbb{Z})$ of level $N$. $\mathcal{A}_{g,N}$ is the \emph{moduli space of principally polarised abelian varieties of dimension $g$ with a level-$N$ structure}. If $N\geq 4$, then $\mathcal{A}_{g,N}$ is smooth and $\mathbb{H}^{+}_{g}\rightarrow\mathcal{A}_{g,N}$ is the topological universal cover. 

A point $x\in\mathcal{A}_{g,N}$ corresponds to a triple $(A,\lambda, \psi_{N})$, where $(A,\lambda)$ denotes an abelian variety $A$ with a polarisation $\lambda$, and $\psi_{N}$ is an isomorphism as in (\ref{eq:symiso}). Let 
\begin{equation*}
\widehat{\Gamma}: = \varprojlim \Gamma/\Gamma(N),
\end{equation*}
where $\Gamma(N)$ is the congruence subgroups of of level $N$ of $\Gamma$. Therefore we have a short exact sequence:
\begin{equation*}
    1\rightarrow \Gamma_{\overline{g}}/\Gamma(N)\rightarrow\Gamma/\Gamma(N)\rightarrow\Gamma/\Gamma_{\overline{g}}\rightarrow 1.
\end{equation*}
As the Mittag-Leffler conditions are satisfied, we conclude that there is surjective homomorphism $\pi:\widehat{\Gamma}\rightarrow\overline{\Gamma}$. \footnote{See https://stacks.math.columbia.edu/tag/0594 for the relevant details showing the existence of a surjective homomorphism.} Observe that $\widehat{\Gamma}$ is open in $\mathrm{GSp}_{2g}(\mathbb{A}_{f})$. Given a special domain $D_{V,i}\subseteq\mathbb{H}_{g}^{+}$, let $\widehat{\Gamma^{V,i}}:=\pi^{-1}\left(\overline{\Gamma}^{V,i}\right)$.

In order to obtain neatness, we will consider $N\geq 3$ (which, by the comments in \textsection\ref{subsec:neat}, will be enough for our needs). It is shown in \cite{ullmo-yafaev} that the Galois representation attached to a point $z\in\mathcal{A}_{g,N}$ (when $N\geq 3$) as we described it in \textsection \ref{subsec:galgen}, coincides with the action of the Galois group of $L$ (here we follow the notation of \textsection \ref{subsec:galgen}) on the Tate module of the abelian variety associated to $z$. This means that the image of $\mathrm{Aut}(\mathbb{C}/L)\rightarrow\widehat{\Gamma^{V,i}}$ is open if the Mumford-Tate conjecture holds for the abelian variety associated to $z$.

The Mumford-Tate conjecture is known for elliptic curves (by Theorem \ref{thm:pink}), simple abelian surfaces, and simple abelian varieties of dimension 3 (see \cite[Theorem 6.1]{lombardo}). By \cite[Theorem 9.5]{commelin}, we know that it then also holds for any finite product of these abelian varieties, which means that Mumford-Tate holds for all abelian surfaces and all abelian varieties of dimension 3. So we know that in these cases the corresponding homomorphisms 
\begin{equation}
\label{eq:open}
    \mathrm{Aut}(\mathbb{C}/\mathbb{Q}^{\mathrm{ab}}(p(x_{1}),\ldots,p(x_{m})))\rightarrow\widehat{\Gamma^{V,i}}
\end{equation}
have open image. As the proof of \cite[Theorem 9.5]{commelin} considers abelian varieties defined over finitely generated fields, this almost takes care of both FIC1 and FIC2. What is missing, is to extend the field $\mathbb{Q}^{\mathrm{ab}}(p(x_{1}),\ldots,p(x_{m}))$ so that it includes the coordinates of special points, which are part of the field of constants we are using in our models. For that, we use the following lemma.

\begin{lem}
\label{lem:abext}
If $g=1$, $g=2$ or $g=3$, then $\mathbb{Q}^{\mathrm{ab}}(\Sigma)$ is a solvable extension of $\mathbb{Q}^{\mathrm{ab}}$.
\end{lem}
\begin{proof}
Recall that the case $g=1$ was treated in \cite{daw-harris}. Given $x\in \mathbb{H}^{+}_{g}$ special, $p(x)$ represents a CM abelian abelian variety of dimension $g$. Recall that the reflex field depends only on the Shimura datum, and not on the choice of $\Gamma$. Let $E_{p(x)}$ denote the reflex field corresponding to the Shimura datum $\left(\mathrm{MT}(x), \left\{x\right\}\right)$. $E_{p(x)}$ is a CM-field of degree at most $2g$ over $\mathbb{Q}$, but it may not be Galois over $\mathbb{Q}$. In any case, we can find another special point $x'\in\mathbb{H}^{+}_{g}$ such that the compositum $E_{p(x)}E_{p(x')}$ is the Galois closure of $E_{p(x)}$ over $\mathbb{Q}$ (see \cite[1.5.3. Example]{chai-conrad-oort}). This Galois closure is of degree at most $8$ over $\mathbb{Q}$ if $g=2$, and of degree at most 12 over $\mathbb{Q}$ if $g=3$. In any case, the Galois group will be solvable. So, if we let $E_{\mathrm{sp}}$ be the compositum of the reflex fields of all special points, then the extension  $\mathbb{Q}^{\mathrm{ab}}E_{\mathrm{sp}}$ is solvable over $\mathbb{Q}^{\mathrm{ab}}$, because the compositum of solvable extensions is solvable. 


To finish, we use that $\mathbb{Q}(p(x))$ is an abelian extension of $E_{p(x)}$, and that the compositum of abelian extensions is abelian.
\end{proof}

We therefore get (recall that categoricity of $j:\mathbb{H}^{+}\rightarrow\mathbb{C}$ was already proven in \cite{daw-harris}):

\begin{cor}
$p:\mathbb{H}_{2}^{+}\rightarrow\mathcal{A}_{2}$ and $p:\mathbb{H}_{3}^{+}\rightarrow\mathcal{A}_{3}$ are categorical.
\end{cor}
\begin{proof}
The group $\mathrm{Sp}_{2g}\left(\widehat{\mathbb{Z}}\right)$ satisfies the conditions of Lemma 3.4 of \cite{ribet} (because its Lie algebra is its own derived algebra, see \cite[Remark 2, p.253]{ribet}), that is, if $\mathcal{U}$ is an open subgroup, then the closure of the commutator of $\mathcal{U}$ is open in $\mathrm{Sp}_{2g}\left(\widehat{\mathbb{Z}}\right)$. To extend this to $\widehat{\Gamma^{V,i}}$, we use \cite[Remark 3, p.253]{ribet}.

By Lemma \ref{lem:abext} then, we know that there are groups $\mathcal{U}_{1},\ldots,\mathcal{U}_{n}$ in $\widehat{\Gamma^{V,i}}$ such that, $\mathcal{U}_{1}$ is the image of $\mathrm{Aut}(\mathbb{C}/\mathbb{Q}^{\mathrm{ab}}(p(x_{1}),\ldots,p(x_{m})))$ under the homomorphism (\ref{eq:open}), $\mathcal{U}_{i+1}$ is normal in $\mathcal{U}_{i}$, $\mathcal{U}_{i}/\mathcal{U}_{i+1}$ is abelian, and $\mathcal{U}_{n}$ is the image of $\mathrm{Aut}\left(\mathbb{C}/\mathbb{Q}^{\mathrm{ab}}(\Sigma)(p(x_{1}),\ldots,p(x_{m}))\right)$. Then, $\mathcal{U}_{1}$ is open and $\mathcal{U}_{2}$ contains the closure of the commutator subgroup of $\mathcal{U}_{1}$, so $\mathcal{U}_{2}$ is open. Inductively then, the image of $\mathrm{Aut}(\mathbb{C}/L)$ is open in $\widehat{\Gamma^{V,i}}$. 
\end{proof}

\begin{akn}
I would like to thank Jonathan Pila for invaluable guidance and supervision. I would also like to thank Gregorio Baldi, Martin Bays, Christopher Daw, and Boris Zilber for very fruitful discussions around the topics presented here. 
\end{akn}

\end{document}